\documentclass[reqno]{amsart}
\pdfoutput=1

\usepackage[british]{babel}
\usepackage[utf8]{inputenc}
\usepackage[T1]{fontenc}

\usepackage{amssymb}
\usepackage{amsfonts}
\usepackage{amsmath}
\usepackage{amsthm}
\usepackage{mathtools}
\usepackage{bbm}
\usepackage{csquotes}
\usepackage{tikz-cd}
\usepackage{adjustbox}
\usepackage{accents}
\usepackage{mathrsfs}
\usepackage{MnSymbol}
\usepackage{soul}
\usepackage{stmaryrd}
\usepackage{xurl}
\usepackage{dsfont}
\usepackage{pdfpages}
\usepackage{graphicx}
\usepackage{comment}
\usepackage[dvipsnames]{xcolor}
\usepackage{etoolbox}
\usepackage{stmaryrd}

\makeatletter
\patchcmd{\@setaddresses}{\indent}{\noindent}{}{}
\patchcmd{\@setaddresses}{\indent}{\noindent}{}{}
\patchcmd{\@setaddresses}{\indent}{\noindent}{}{}
\patchcmd{\@setaddresses}{\indent}{\noindent}{}{}
\makeatother

\PassOptionsToPackage{hyphens}{url}
\usepackage[hypertexnames=false, breaklinks=true]{hyperref}
\hypersetup{
 colorlinks,
 linkcolor={red!40!black},
 citecolor={green!40!black},
 urlcolor={blue!50!black}
}
\usepackage[compress]{cleveref}
\usepackage{thmtools, thm-restate}
\declaretheoremstyle[spaceabove=.8\baselineskip,spacebelow=.8\baselineskip,headfont=\bfseries,notefont=\normalfont,bodyfont=\itshape,postheadspace=.5em]{thms}

\declaretheoremstyle[spaceabove=.8\baselineskip,spacebelow=.8\baselineskip,headfont=\bfseries,notefont=\normalfont,bodyfont=\normalfont,postheadspace=.5em]{defn}

\numberwithin{equation}{section}\theoremstyle{thms}
\newtheorem{thm}[equation]{Theorem}
\newtheorem*{thm*}{Theorem}
\newtheorem{cor}[equation]{Corollary}

\newtheorem{lem}[equation]{Lemma}
\newtheorem{prop}[equation]{Proposition}
\newtheorem*{prop*}{Proposition}

\theoremstyle{defn}
\newtheorem{cons}[equation]{Construction}
\newtheorem{defi}[equation]{Definition}
\newtheorem*{defi*}{Definition}
\newtheorem{ex}[equation]{Example}

\newtheorem{nota}[equation]{Notation}
\newtheorem*{nota*}{Notation}
\newtheorem{rem}[equation]{Remark}

\newtheorem{warn}[equation]{Warning}
\newtheorem{rec}[equation]{Recollection}

\crefname{prop}{Proposition}{Propositions}
\crefname{hyp}{Hypothesis}{Hypotheses}
\crefname{cor}{Corollary}{Corollaries}
\crefname{lem}{Lemma}{Lemmata}
\crefname{cons}{Construction}{Constructions}
\crefname{def}{Definition}{Definitions}
\crefname{rem}{Remark}{Remarks}
\crefname{nota}{Notation}{Notations}
\crefname{ex}{Example}{Examples}
\crefname{thm}{Theorem}{Theorems}
\crefname{rec}{Recollection}{Recollections}
\usepackage[shortlabels,inline]{enumitem}

\newcommand{\nA}{\tn{A}}

\newcommand{\nP}{\tn{P}}

\newcommand{\bA}{\mathbb{A}}

\newcommand{\bC}{\mathbb{C}}
\newcommand{\bD}{\mathbb{D}}
\newcommand{\bE}{\mathbb{E}}
\newcommand{\bF}{\mathbb{F}}
\newcommand{\bG}{\mathbb{G}}

\newcommand{\bP}{\mathbb{P}}
\newcommand{\bQ}{\mathbb{Q}}
\newcommand{\bR}{\mathbb{R}}
\newcommand{\bS}{\mathbb{S}}

\newcommand{\bZ}{\mathbb{Z}}

\newcommand{\cC}{\mathcal{C}}
\newcommand{\cD}{\mathcal{D}}
\newcommand{\cE}{\mathcal{E}}
\newcommand{\cF}{\mathcal{F}}
\newcommand{\cG}{\mathcal{G}}

\newcommand{\cI}{\mathcal{I}}

\newcommand{\cK}{\mathcal{K}}

\newcommand{\cM}{\mathcal{M}}
\newcommand{\cN}{\mathcal{N}}
\newcommand{\cO}{\mathcal{O}}
\newcommand{\cP}{\mathcal{P}}
\newcommand{\cQ}{\mathcal{Q}}

\newcommand{\cS}{\mathcal{S}}

\newcommand{\cU}{\mathcal{U}}

\newcommand{\tn}[1]{\textnormal{#1}}

\newcommand{\tsf}[1]{\textsf{\textup{#1}}}
\newcommand{\tbf}[1]{\textbf{\textup{#1}}}
\newcommand{\mcal}[1]{\mathcal{#1}}
\newcommand{\mfr}[1]{\mathfrak{#1}}

\newcommand{\what}[1]{\widehat{#1}}
\newcommand{\underl}[1]{\underline{#1}}
\newcommand{\overl}[1]{\overline{#1}}

\newcommand{\colimit}[1]{\underset{#1}{\tn{colim}} \;}

\newcommand*{\defeq}{\mathrel{\vcenter{\baselineskip0.4ex \lineskiplimit0pt
\hbox{\scriptsize.}\hbox{\scriptsize.}}}%
=}
\newcommand{\op}{\tn{op}}
\newcommand{\fun}{\tn{Fun}}
\newcommand{\Hom}{\tn{Hom}}

\newcommand{\map}{\tn{map}}
\newcommand{\Map}{\tn{Map}}
\newcommand{\imap}{\underl{\tn{map}}}
\newcommand{\psigma}{\cP_{\Sigma}}
\newcommand{\sheaves}{\mcal{S}\tn{h}}
\newcommand{\alg}{\tn{Alg}}
\newcommand{\calg}{\tn{CAlg}}
\newcommand{\bbone}{\mathbbm{1}}

\newcommand{\modgr}{\tsf{Mod}(G;R)}
\newcommand{\Mod}{\tsf{Mod}}

\newcommand{\Spec}{\tn{Spec}}

\newcommand{\spans}{\tsf{span}}
\newcommand{\Spans}{\tsf{Span}}
\newcommand{\spec}{\tsf{Sp}}
\newcommand{\specg}{\tsf{Sp}^G}

\newcommand{\spc}{\tsf{Spc}}

\newcommand{\spcg}{\tsf{Spc}^G}
\newcommand{\spcgp}{\tsf{Spc}^G_{\ast}}
\newcommand{\catinf}{\tsf{Cat}_{\infty}}

\newcommand{\prl}{\tsf{Pr}^{\tn{L}}}
\newcommand{\prlst}{\tsf{Pr}^{\tn{L}}_{\tn{st}}}

\newcommand{\prr}{\tsf{Pr}^{\tn{R}}}
\newcommand{\prrst}{\tsf{Pr}^{\tn{R}}_{\tn{st}}}

\newcommand{\cmackgr}{\tsf{Mack}_R^{\tsf{coh}}(G)}

\newcommand{\perm}{\tsf{perm}}
\newcommand{\permgr}{\tsf{perm}(G;R)}
\newcommand{\Perm}{\tsf{Perm}}

\newcommand{\gset}{G\tsf{-set}}

\newcommand{\sms}{\tsf{Sm}_S}
\newcommand{\fets}{\tsf{Fét}_S}
\newcommand{\fetslice}[1]{\tsf{Fét}_{{#1}/S}}
\newcommand{\ets}{\tsf{Ét}_S}
\newcommand{\corsms}[1]{\tsf{SmCor}(S;#1)}
\newcommand{\corfets}[1]{\tsf{FétCor}(S;#1)}
\newcommand{\corfetsslice}[2]{\tsf{FétCor}({#2}/S;#1)}

\newcommand{\ind}{\tn{Ind}}

\newcommand{\glo}{\tsf{Glo}}

\newcommand{\lat}{\tsf{Lat}}

\newsavebox{\pullback}
\sbox\pullback{%
\begin{tikzpicture}%
\draw (0,0) -- (1ex,0ex);%
\draw (1ex,0ex) -- (1ex,1ex);%
\end{tikzpicture}}

\newsavebox{\pushout}
\sbox\pushout{%
\begin{tikzpicture}%
\draw (0,0) -- (0ex,1ex);%
\draw (0ex,1ex) -- (1ex,1ex);%
\end{tikzpicture}}

\DeclareFontFamily{U}{mathx}{}
\DeclareFontShape{U}{mathx}{m}{n}{<-> mathx10}{}
\DeclareSymbolFont{mathx}{U}{mathx}{m}{n}
\DeclareMathAccent{\widecheck}{0}{mathx}{"71}

\tikzcdset{scale cd/.style={every label/.append style={scale=#1}, cells={nodes={scale=#1}}}}

\date{}
\author{Yorick Fuhrmann}
\address{Yorick Fuhrmann, Warwick Mathematics Institute, Coventry CV4 7AL, UK}
\email{yorick.fuhrmann@warwick.ac.uk}
\urladdr{https://warwick.ac.uk/fac/sci/maths/people/staff/fuhrmann}

\hypersetup{pdfauthor={Yorick Fuhrmann},pdftitle={Profinite Borel completeness and smooth Artin motives}}

\usepackage[backend=biber, citestyle=authoryear, style=alphabetic, maxbibnames=99]{biblatex}
\AtBeginBibliography{\footnotesize}

\begin{document}

\title[Profinite Borel completeness and smooth Artin motives]{Profinite Borel completeness \\ and smooth Artin motives}

\begin{abstract}
The purpose of this paper is twofold. In the first part, we revisit the description of the $\infty$-category of Borel complete equivariant spectra for a finite group given by Mathew--Naumann--Noel, introduce a version with coefficients, and then consider Borel equivariance for profinite groups. Here we identify two generally differing notions: `levelwise' Borel completeness and the hypercompletion thereof. \\
In the second part, we study variants of smooth Artin motives, which are subcategories of the $\infty$-categories of effective Nisnevich and étale Voevodsky motives over a base scheme $S$ that are controlled by the étale fundamental group $\pi_1^{\tn{ét}}(S)$. In the Nisnevich case, we extend a theorem of Voevodsky and identify smooth Artin motives with modules over the Bredon cohomology spectrum for the profinite group $\pi_1^{\tn{ét}}(S)$. In the étale case, we show that the difference between our two notions of profinite Borel completeness is precisely the difference between étale sheaves and hypersheaves on finite étale schemes.
\end{abstract}

%\subjclass[]{}
%\keywords{}

\maketitle
\tableofcontents

\section{Introduction}\label{sec:intro}

It is a classical result of Grothendieck \cite[§5, Théorème 4.1]{sga1} that finite étale schemes over a connected base scheme $S$ are nothing but finite discrete sets with a continuous action of the étale fundamental group $\pi_1^{\tn{ét}}(S)$, in symbols
\begin{equation}\label{eq:fetspione}
\fets \simeq \pi_1^{\tn{ét}}(S)\tsf{-set}.
\end{equation}
When $S$ is the spectrum of a field $\bF$ we can follow Voevodsky's approach to motives and construct an additive category of \emph{$\bF$-correspondences} $\tsf{SmCor}(\bF;\bZ)$ from the category of smooth $\bF$-schemes $\tsf{Sm}_\bF$. For a commutative ring $R$, additive $R$-module valued presheaves on $\tsf{SmCor}(\bF;\bZ)$ which are sheaves for the Nisnevich topology are called \emph{sheaves with transfers}. From the category they comprise one constructs effective Voevodsky motives $\mcal{DM}_{\tn{Nis}}^{\tn{eff}}(\bF;R)$ through passage to the derived category followed by $\bA^1$-localisation, and there is an `associated motive' functor
$$\cM_\tn{Nis}: \; \tsf{Sm}_\bF \to \mcal{DM}_{\tn{Nis}}^{\tn{eff}}(\bF;R).$$
Under this linearisation procedure \emph{Artin motives} $\mcal{DAM}_{\tn{Nis}}(\bF;R)$, defined as the localising subcategory inside $\mcal{DM}_{\tn{Nis}}^{\tn{eff}}(\bF;R)$ generated by the motives associated to $0$-dimensional smooth $\bF$-schemes, are closely related to permutation representations of $\tn{Gal}(\bF^{\tn{sep}}/\bF)$, the absolute Galois group of the field $\bF$. This was originally observed by Voevodsky \cite[§3.4]{voe00} and has recently been revisited by Balmer--Gallauer \cite[§7]{bg23b}. Namely, there is a symmetric monoidal equivalence
\begin{equation}\label{eq:damaspermfield}
\mcal{DAM}_{\tn{Nis}}(\bF;R) \simeq \cD\tsf{Perm}(\tn{Gal}(\bF^{\tn{sep}}/\bF);R),
\end{equation}
where the latter is the category of derived permutation modules defined in \cite{bg22a}. 
\noindent Since the absolute Galois group $\tn{Gal}(\bF^{\tn{sep}}/\bF)$ is precisely the étale fundamental group of the scheme $\Spec(\bF)$, the equivalence \ref{eq:fetspione} suggests a natural question:
\vspace{0.2cm}
\begin{center}
\emph{Is \ref{eq:damaspermfield} still true when we replace $\Spec(\bF)$ by a general base scheme $S$ \\ and $\tn{Gal}(\bF^{\tn{sep}}/\bF)$ by its étale fundamental group?}
\end{center}
\vspace{0.2cm}
We answer this question in the affirmative, under mild assumptions on the base scheme. To this end, we define the $\infty$-category of effective Voevodsky motives 
$$\mcal{DM}_{\tn{Nis}}^{\tn{eff}}(S;R) \defeq L_{\bA^1}\sheaves_{\tn{Nis}}(\corsms{R};\spec)$$
over a base scheme $S$ with coefficients in $R$ as the $\infty$-category of $\bA^1$-invariant Nisnevich sheaves with transfers with values in spectra.\footnote{Under the assumptions on $S$ we work with we could also use hypersheaves instead of sheaves.} For this, we employ the theory of finite correspondences with coefficients \cite[§9]{cd19}. Voevodsky motives are usually defined by first deriving the $1$-category of Nisnevich sheaves with transfers. In Appendix \ref{sec:sheaveswithtransfers} we show that both of these approaches are equivalent. \\
We are then in a position to consider specific subcategories of motives, and we define the $\infty$-category of \emph{smooth Artin motives} 
$$\mcal{DAM}_\tn{Nis}(S;R) \defeq \tn{Loc}(\cM_{\tn{Nis}}(X) \mid X \tn{ finite étale over } S)$$
as the localising subcategory of $\mcal{DM}_\tn{Nis}^{\tn{eff}}(S;R)$ generated by the motives associated to finite étale $S$-schemes. When $S$ is the spectrum of a field, this recovers the notion of Artin motives that was considered by Voevodsky. Crucially relying on the equivalence \ref{eq:fetspione}, we then show:

\begin{thm*}[\Cref{thm:damnisasmodunderlr}]
Let $S$ be a connected, noetherian scheme of finite Krull dimension. Then there is a symmetric monoidal equivalence
$$\mcal{DAM}_\tn{Nis}(S;R) \simeq \cD\tsf{Perm}(\pi_1^{\tn{ét}}(S,s);R),$$
where $s: \Spec(k) \to S$ is a geometric point of $S$.
\end{thm*}

\noindent The $\infty$-categories above are presentably symmetric monoidal stable $\infty$-categories and hence their homotopy categories are `large' tt-categories. By recent work of Balmer--Gallauer \cite{bg25a,bg25b}, the Balmer spectrum of the compact part of $\cD\tsf{Perm}(G;R)$ (and hence of \emph{geometric} smooth Artin motives) is known for any profinite group $G$ when $R$ is a field, and by work of Dubey--Gómez \cite{dg25} its set is known when $G$ is finite and $R$ is noetherian. \\
In fact, in order to prove the above theorem, we work with an equivariant description of the $\infty$-category $\cD\tsf{Perm}(\pi_1^{\tn{ét}}(S,s);R)$, which by Corollary 4.8 or Theorem 5.12 of \cite{fuh25} is given by $\Mod_{\underl{R}}(\specg)$, modules over the $R$-linear Bredon cohomology spectrum in equivariant spectra for the profinite group $G=\pi_1^{\tn{ét}}(S,s)$. In the proof we rely on its spectral Mackey functor description, see \cite[Example 2.15]{bcn25} or \cite[Corollary 3.18]{fuh25}.\\
So far we have used the Nisnevich topology for the motives we considered, but replacing it by the étale topology comes with no difficulty, yielding étale versions
$$\mcal{DM}_{\tn{ét}}^{(\wedge),\tn{eff}}(S;R) \defeq L_{\bA^1}\sheaves_{\tn{ét}}^{(\wedge)}(\corsms{R};\spec)$$
inside which we can similarly define smooth Artin motives. As indicated by the superscript $(\wedge)$, in the étale case it is crucial to distinguish between sheaves and hypersheaves, since the étale topos of a scheme is generally not hypercomplete. For this we again use sheaves with transfers - versions of Artin motives for the étale topology but without transfers have among others been considered by Ayoub--Zucker \cite{az12}, Cavicchi--Déglise--Nagel \cite{cdn23} and Ruimy \cite{rui25a,rui25b}. It is now an easy observation that étale (hyper)sheafification on sheaves with transfers descends to motives, and there is a second natural question:
\vspace{0.2cm}
\begin{center}
\emph{Do étale smooth Artin motives also admit descriptions in terms of \\ representation theory and equivariant homotopy theory?}
\end{center}
\vspace{0.2cm}
To tackle this question, we will take a digression into equivariant homotopy theory, which will take up the first two sections. Let us outline our considerations. \\

\noindent When $G$ is a finite group, a $G$-spectrum $X$ is \emph{Borel complete} if maps $Y \to X$ are null whenever $Y$ is non-equivariantly contractible, i.e. $\tn{res}^G_1(Y) \simeq 0$, or equivalently $Y \otimes \bD(G) \simeq 0$, where $\bD(G)$ is the equivariant function dual of the orbit $\Sigma^{\infty}_+G/1$. This already suggests that in this case the `genuine' $G$-action on $X$ should come from a `naive' one, which is made precise by the identification of the Borel complete subcategory of $\specg$ as a functor category
$$\specg_{\tn{Borel}} \simeq \fun(BG,\spec)$$
established by Mathew--Naumann--Noel in \cite[§6.3]{mnn17}. In terms of spectral Mackey functors \cite{bar17}, we can think of Borel completion as an evaluation at the orbit $G/1$. If one tries to generalise these statements to profinite groups, one faces the difficulty that the equivariant homotopy theory of such groups is already controlled by the finite (discrete) $G$-sets, and the orbit $G/1$ is not among such. \\
Instead, it is an easy observation that Borel complete equivariant spectra are preserved by categorical fixed points for arbitrary normal subgroups. Since $G$-spectra for profinite $G=\tn{lim}_i \, G_i$ are themselves a limit along categorical fixed points, we can define a full subcategory
$$\specg_{\tn{lwBorel}} \defeq \tn{lim}_i \; \Spec^{G_i}_{\tn{Borel}} \subseteq \specg,$$
the objects of which we call \emph{levelwise Borel complete}. We can view objects of $\specg$ as families of $G_i$-spectra $(X_i)_{i \in I}$ which are compatible with respect to categorical fixed points, i.e. $X_i^{N_{ij}} \simeq X_j$, where $N_{ij}$ is the kernel of the canonical map $G_i \twoheadrightarrow G_j$. Such an object is levelwise Borel complete if each $X_i$ is a Borel complete $G_i$-spectrum.\\
A second possible approach is motivated by the observation that in a profinite group the trivial subgroup is the intersection of all open normal subgroups. Instead of working with the dual $\bD(G)$ of the orbit $\Sigma^{\infty}_+G/1$ as one does in the finite group case, one should thus consider the duals $\bD(G/N_i)$ of all open normal subgroups at the same time, where $N_i$ is the kernel of the natural map $G \twoheadrightarrow G_i$. They assemble into a filtered diagram in $\calg(\specg)$, and we denote its colimit by $A \defeq \tn{colim}_i \, \bD(G/N_i)$. It naturally identifies with $\tn{coind}^G_1(\bS)$, the coinduction of the non-equivariant sphere spectrum to $G$. The $\infty$-category of Borel complete $G$-spectra is then defined as
$$\specg_{\tn{Borel}} \defeq \specg_{A\tn{-cpl}},$$
that is, as complete objects (in the sense of \cite[§2.2]{mnn17}) with respect to the algebra object $A$. This means that a $G$-spectrum is Borel complete if all maps it receives from non-equivariantly contractible $G$-spectra are null, as for finite groups. \\
In \Cref{prop:tstructureonlevelwiseborelcomplete} we observe that there is a right complete t-structure on $\specg_{\tn{lwBorel}}$, and for any t-structure we can make sense of the hypercomplete objects as the ones in the right orthogonal of the acyclics. By an acyclic object we mean one for which all t-structure homotopy objects vanish. In the main theorem of \Cref{sec:borelprofinite} we then show that $\specg_{\tn{Borel}}$ identifies with the full subcategory of hypercomplete objects with respect to this t-structure.

\begin{thm*}[\Cref{thm:borelandhyperlewelwise}]
There is a symmetric monoidal equivalence
$$\specg_{\tn{Borel}} \simeq (\specg_{\tn{lwBorel}})^h.$$
\end{thm*}

\noindent In fact, we work with coefficients in any connective commutative non-equivariant ring spectrum, and non-equivariant coefficients are not less general than equivariant ones when dealing with Borel complete categories, see \Cref{cor:justrestriction}.\\
We then connect the above two versions of Borel complete $G$-spectra for profinite groups $G=\lim_i \, G_i$ to $\infty$-categories of representations and étale sheaves. The latter will further down the line be the bridge to étale motives. 

\begin{prop*}[\Cref{prop:sheavesborelcomplete} and \Cref{prop:sheavesgrmodules}]
There is a commutative diagram with horizontal equivalences and fully faithful vertical functors
\[\begin{tikzcd}
	{\Mod_{R_G}(\specg)_{\tn{Borel}}} & {\sheaves^{\wedge}(\gset;R)} & {\cD(\modgr)} \\
	{\Mod_{R_G}(\specg)_{\tn{lwBorel}}} & {\sheaves(\gset;R)} & {\tn{colim}_i \; \cD(\Mod(G_i;R)).}
	\arrow["\sim", from=1-1, to=1-2]
	\arrow[hook, from=1-1, to=2-1]
	\arrow["\sim", from=1-2, to=1-3]
	\arrow[hook, from=1-2, to=2-2]
	\arrow[hook, from=1-3, to=2-3]
	\arrow["\sim", from=2-1, to=2-2]
	\arrow["\sim", from=2-2, to=2-3]
\end{tikzcd}\]
\end{prop*}

\noindent Here $R_G$ is the inflation of a discrete commutative ring $R$ to $G$-spectra, the sheaf categories consist of (hyper)sheaves of $R$-modules in spectra defined on the site of finite discrete $G$-sets with jointly surjective covering maps, $\cD(\modgr)$ denotes the derived $\infty$-category of discrete $(G;R)$-modules, and $\tn{colim}_i \; \cD(\Mod(G_i;R))$ is the filtered colimit of the derived categories $\cD(\Mod(G_i;R))$ with transition maps induced by restriction along the quotient maps $G_i \twoheadrightarrow G_j$, taken in $\prl$. In fact, for the left-hand square of this diagram the proposition holds for any connective commutative ring spectrum $R$. \\
By work of Clausen--Mathew \cite[§4.1]{cm21}, both rows agree if the virtual cohomological dimension of $G$ is finite, in the case that $R$ is discrete. Analogous statements can be made for general connective $R \in \calg(\spec)$, see \Cref{prop:sheavesborelcomplete}. \\
Furthermore, note that for the sheaf categories in the above diagram the two notions \emph{hypersheaf} and \emph{hypercomplete sheaf} (the latter defined with respect to the canonical t-structure on sheaves) agree. We will show this in \Cref{prop:hypersheafhypercomplete} by reducing the case with coefficients in a ring spectrum to the sheaves of spectra case. \\

\noindent We finally come full circle and note that by Grothendieck's equivalence \ref{eq:fetspione}, étale sheaves on finite étale $S$-schemes are the same as sheaves on finite discrete $G$-sets when $G$ is the étale fundamental group $\pi_1^{\tn{ét}}(S)$. It is then a crucial (and well known, see e.g. \cite[Proposition 3.1.4]{cd16}) observation that étale sheaves on (finite) étale $S$-schemes already have transfers (\Cref{thm:etaletransfers}):
$$\sheaves_{\tn{ét}}^{(\wedge)}(\corfets{R};\spec) \simeq \sheaves_{\tn{ét}}^{(\wedge)}(\fets;R).$$
Here $\corfets{R}$ is the full subcategory of the category of smooth schemes with finite correspondences $\corsms{R}$ spanned by the finite étale $S$-schemes. We give an independent proof of this result that does not rely on the identifications of hom-sets made in \cite{cd16}. It is equivariant in spirit - morally speaking we can think of étale sheaves as Borel complete $G$-spectra, since in the finite group case they are right Kan extended from the free orbit. They hence have an overlying $G$-spectrum which canonically admits transfers. For $R$ a discrete commutative ring and $G=\pi_1^{\tn{ét}}(S)$ the full picture then looks as follows:

\begin{thm*}[\Cref{thm:bigdiagram}]
There is a commutative diagram of left adjoints in which all horizontal functors are equivalences.
\[\begin{tikzcd}
	{\cD\Perm(G;R)} & {\Mod_{\underl{R}}(\specg)} & {\sheaves_{\tn{Nis}}(\corfets{R};\spec)}  \\
	{\tn{colim}_i \, \cD(\Mod(G_i;R))} & {\Mod_{R_G}(\specg)_{\tn{lwBorel}}} & {\sheaves_{\tn{ét}}(\corfets{R};\spec)}  \\
	{\cD(\Mod(G;R))} & {\Mod_{R_G}(\specg)_{\tn{Borel}}} & {\sheaves_{\tn{ét}}^{\wedge}(\corfets{R};\spec)} 
	\arrow["\sim", from=1-1, to=1-2]
	\arrow[from=1-1, to=2-1]
	\arrow["\sim", from=1-2, to=1-3]
	\arrow[from=1-2, to=2-2]
	\arrow["\tsf{a}_{\tn{ét}}", from=1-3, to=2-3]
	\arrow["\sim", from=2-1, to=2-2]
	\arrow["L^h", from=2-1, to=3-1]
	\arrow["\sim", from=2-2, to=2-3]
	\arrow["L^h", from=2-2, to=3-2]
	\arrow["(-)^{\wedge}_\tn{ét}", from=2-3, to=3-3]
	\arrow["\sim", from=3-1, to=3-2]
	\arrow["\sim", from=3-2, to=3-3]
\end{tikzcd}\]
\end{thm*}

\noindent Each of the categories of sheaves with transfers on the right then maps to the respective version of smooth Artin motives constructed as a localising subcategory of effective motives. As stated above, this functor is always an equivalence in the Nisnevich case (\Cref{thm:damnisasmodunderlr}). \\
In the hypercomplete étale case we need to impose assumptions on base scheme and coefficients to obtain an equivalence: If $R$ has positive characteristic which is invertible on $S$, the Suslin--Voevodsky rigidity theorem in the form of \cite[Theorem 4.5.2]{cd16} applies. This reduces to checking when étale hypersheaves on finite étale $S$-schemes map fully faithfully to those on étale $S$-schemes. This holds when $S$ is the spectrum of a henselian local ring, and when it is a $K(\pi,1)$-scheme it holds on a certain rigid subcategory, see \Cref{cor:dametassheaves}. In the non-hypercomplete étale case analogous statements would require computations of mapping spaces in the corresponding category of effective motives, which we won't make. \\
By work of Cisinski--Déglise \cite[§2]{cd16} which generalises a theorem of Voevodsky \cite[§5, Theorem 4.1.12]{vsf00}, all rows as well as the respective categories of smooth Artin motives and effective motives collapse to a single row when we use coefficients in a $\bQ$-algebra. We summarise these statements in \Cref{thm:bigdiagramright}. \\

\noindent \tbf{Organisation of the paper.} The first two sections are entirely devoted to equivariant homotopy theory. In \Cref{sec:borelfinite} we review the identification of the Borel complete subcategory for finite groups as a functor category, as established by Mathew--Naumann--Noel. We do this with coefficients and then prove basic change of group results. \Cref{sec:borelprofinite} then generalises this theory to profinite groups, where we define the two differing notions alluded to above. The main result of this section is that one is the hypercompletion of the other. The latter is then used in \Cref{sec:sheavesandreps} to connect the Borel complete categories of equivariant spectra for profinite groups to categories of étale sheaves and representations. \\
In \Cref{sec:motives} we dive into the theory of correspondences and motives, establish an identification of smooth Artin motives for the Nisnevich topology, and study the effect of étale sheafification in the equivariant and representation theoretic setting. We end with Appendix \ref{sec:sheaveswithtransfers} by comparing our definitions of motives with the more classical ones in the literature. \\

\noindent \tbf{Notations and conventions.} We freely use the language of $\infty$-categories as developed by Lurie \cite{lur09,lur17,lur18}. We sometimes speak of `categories' instead of `$\infty$-categories' when it is clear that the context is $\infty$-categorical. \\
We write $\spc$ for the $\infty$-category of spaces/homotopy types and $\spec$ for the $\infty$-category of spectra. The $G$-equivariant analogues are denoted $\spcg$ and $\specg$, the latter is what is sometimes referred to as `genuine' $G$-spectra. We tacitly identify discrete rings and Mackey functors with their Eilenberg-MacLane spectra when viewing them as (equivariant) spectra. Mapping spaces in an $\infty$-category $\cC$ are denoted $\Map_{\cC}$, and by $\Hom_{\cC}$ we mean $\pi_0 \Map_{\cC}$, the hom-set in the homotopy category of $\cC$. \\

\noindent \tbf{Acknowledgements.} I would like to thank my advisor Martin Gallauer for his guidance during the work on this project and for his excellent comments on drafts of this paper. I want to thank Timo Richarz for suggesting to extend Voevodsky's description of Artin motives to general base schemes. Furthermore, I am indebted to Drew Heard, Achim Krause and Marius Nielsen for excellent discussions, many helpful remarks, and for their hospitality during a research visit to Trondheim and Oslo. I thank Scott Balchin, David Barnes, Tobias Barthel and Luca Pol for comments on an earlier version of this article, and I want to thank Raphaël Ruimy for pointing out that in a previous version of this article assumptions for \Cref{cor:dametassheaves} were missing, which led to significant improvements of \Cref{sec:motives}. \\
The author is supported by the Warwick Mathematics Institute Centre for Doctoral Training and gratefully acknowledges funding from the University of Warwick.

\section{Borel completeness for finite groups}\label{sec:borelfinite}

This section mostly consists of recollections and immediate generalisations of the results on Borel completeness in \cite{mnn17}. We first recall some facts about module $\infty$-categories and equivariant spectra for profinite groups that we will use throughout the paper and then consider Borel completeness for modules over equivariant ring spectra in the finite group case. At the end of the section we establish change of group results for Borel complete subcategories.

\begin{rec}\label{rec:modules}
If $\cC$ is a presentably symmetric monoidal\footnote{That is, $\cD$ is presentable and admits a symmetric monoidal structure for which the tensor product preserves colimits in both variables.} stable $\infty$-category and $A \in \calg(\cC)$, then the $\infty$-category $\Mod_A(\cC)$ of $A$-modules in $\cC$ is presentably symmetric monoidal and stable \cite[Theorem 3.4.4.2]{lur17}. It comes with a free-forgetful adjunction
$$(F_A,U_A): \cC \rightleftarrows \Mod_A(\cC),$$
the free functor $F_A$ is symmetric monoidal, and the composite $U_AF_A$ is the functor $(-) \otimes A: \cC \to \cC$ \cite[Corollary 4.2.3.7, 4.2.4.8, Theorem 4.5.2.1, §4.5.3]{lur17}. If $\cC$ is compactly generated by a set of objects $\cG$, then $\Mod_A(\cC)$ is compactly generated by the set $F_A(\cG)$. If the objects of $\cG$ are also dualisable, then the objects of $F_A(\cG)$ are. If the unit of $\cC$ is compact then $A$ is compact in $\Mod_A(\cC)$, so in this case an object of $\Mod_A(\cC)$ is compact if and only if it is dualisable, and $\Mod_A(\cC)$ is rigidly-compactly generated by the set $F_A(\cG)$. \\
A symmetric monoidal functor $\alpha: \cC \to \cD$ restricts to commutative algebra objects $\calg(\cC) \to \calg(\cD)$, and for any $A \in \calg(\cC)$ there is an induced symmetric monoidal functor
$$\bar{\alpha}: \Mod_A(\cC) \to \Mod_{\alpha(A)}(\cD)$$
on module categories. If $\alpha$ preserves colimits then so does the induced $\bar{\alpha}$ \cite[Corollary 4.2.3.5]{lur17}, and as explained in \cite[Remark 1.1.11]{erg22} both squares in the following diagram commute.
\[\begin{tikzcd}
	\cC & \cD \\
	{\Mod_A(\cC)} & {\Mod_{\alpha(A)}(\cD)}
	\arrow["\alpha", from=1-1, to=1-2]
	\arrow[""{name=0, anchor=center, inner sep=0}, "{F_A}"', shift right=2, from=1-1, to=2-1]
	\arrow[""{name=1, anchor=center, inner sep=0}, "{F_{\alpha(A)}}"', shift right=2, from=1-2, to=2-2]
	\arrow[""{name=2, anchor=center, inner sep=0}, "{U_A}"', shift right=2, from=2-1, to=1-1]
	\arrow["\bar{\alpha}", from=2-1, to=2-2]
	\arrow[""{name=3, anchor=center, inner sep=0}, "{U_{\alpha(A)}}"', shift right=2, from=2-2, to=1-2]
	\arrow["\dashv"{anchor=center}, draw=none, from=0, to=2]
	\arrow["\dashv"{anchor=center}, draw=none, from=1, to=3]
\end{tikzcd}\]
If $\alpha$ preserves limits, then so does $\bar{\alpha}$ \cite[Corollary 4.2.3.3]{lur17}, and a right adjoint $\beta: \cD \to \cC$ to $\alpha$ induces a right adjoint $\bar{\beta}: \Mod_{\alpha(A)}(\cD) \to \Mod_A(\cC)$ to the functor $\bar{\alpha}: \Mod_A(\cC) \to \Mod_{\alpha(A)}(\cD)$. As a right adjoint to a symmetric monoidal functor the functor $\beta: \cD \to \cC$ is lax symmetric monoidal, and when $B \in \calg(\cD^{\otimes})$, then we also have an induced lax symmetric monoidal functor
$$\tilde{\beta}: \Mod_B(\cD) \to \Mod_{\beta(B)}(\cC).$$
It factors as $\Mod_B(\cD) \stackrel{\varepsilon_B^*}{\longrightarrow} \Mod_{\alpha \beta(B)}(\cD) \stackrel{\bar{\beta}}{\longrightarrow} \Mod_{\beta(B)}(\cC)$, where $\varepsilon_B^*$ is the restriction of scalars along the counit $\varepsilon_B: \alpha \beta(B) \to B$, and hence $\tilde{\beta}$ has a symmetric monoidal left adjoint $\tilde{\alpha}: \Mod_{\beta(B)}(\cC) \to \Mod_B(\cD)$ given by the composition $\Mod_{\beta(B)}(\cC) \stackrel{\bar{\alpha}}{\longrightarrow} \Mod_{\alpha \beta(B)}(\cD) \stackrel{\varepsilon_{B,!}}{\longrightarrow} \Mod_B(\cD)$, where $\varepsilon_{B,!}$ is the extension of scalars along $\varepsilon_B: \alpha \beta(B) \to B$.
\end{rec}

There are multiple models for the $\infty$-category of $G$-spectra $\specg$ for a profinite group $G$ which are all known to be equivalent \cite[§A]{bbb24}. Our preferred model in this paper will be the one of spectral Mackey functors \cite{bar17,bgs20}, since in this model the relation to the Borel complete subcategory is particularly clear. We could hence take the functor category
$$\fun^{\times}(\Spans(G)^{\op},\spec)$$
as a definition for $\specg$. Here $\Spans(G)$ denotes the effective Burnside $\infty$-category of finite continuous $G$-sets \cite[§3, Example B]{bar17} and the superscript $\times$ indicates that we are taking the full subcategory of finite product preserving functors.\footnote{Since $\Spans(G)$ is self-dual, we could omit the $(-)^\op$ in the functor category, which is the more commonly used convention for spectral Mackey functors. Later on we will consider many (pre)sheaf categories for sites that are not self-dual, and the above choice works better with the flow of these.}

\begin{rec}
The \emph{$\infty$-category of $G$-spectra} $\specg$ for a profinite group $G$ is stable and presentably symmetric monoidal. We denote the symmetric monoidal product by $\otimes$ and the internal mapping $G$-spectrum by $\imap_{\specg}(-,-)$. It receives a symmetric monoidal left adjoint suspension functor $\Sigma^{\infty}: \spcgp \to \specg$ from pointed $G$-spaces. We write $\Sigma^{\infty}_+: \spcg \to \specg$ for the composition $\Sigma^{\infty} \circ (-)_+$, where the functor $(-)_+: \spcg \to \spcgp$ adds a disjoint base point with trivial action. \\
The sphere spectrum $\Sigma^{\infty}_+ G/G$ is denoted $\bS_G = \bS^0_G$. It is the unit of the symmetric monoidal structure. The $\infty$-category $\specg$ is rigidly compactly generated by the suspension spectra of orbits $\Sigma^{\infty}_+G/H$, for open subgroups $H \leq G$. It has a natural t-structure for which the t-structure homotopy objects agree with the homotopy Mackey functors $\pi_n^{H}(-) = \Hom_{\specg}(\Sigma^{\infty + n}_+ G/H,-)$. This t-structure is both left and right complete, compatible with filtered colimits, and it is compatible with the symmetric monoidal structure on $\specg$, in the sense that $\specg_{\geq 0}$ is a symmetric monoidal subcategory.
\end{rec}

\begin{rec}
For any homomorphism of profinite groups (i.e., continuous group homomorphism) $\varphi: G \to G'$  there is a symmetric monoidal left adjoint functor $\varphi^*: \spec^{G'} \to \specg$ with lax symmetric monoidal right adjoint $\varphi_*: \spec^{G} \to \spec^{G'}$. \\
If $\varphi$ is the inclusion of a closed subgroup $H \hookrightarrow G$, the left adjoint $\varphi^*: \spec^{G} \to \spec^{H}$ is the \emph{restriction functor} $\tn{res}^G_H$, and its right adjoint\footnote{When $H$ is furthermore of finite index, then $\tn{res}^G_H$ also has a left adjoint, the \emph{induction} $\tn{ind}^G_H$. This goes in line with the next remark, where $H$ being of finite index ensures that $(G \times X)/H$ is a finite $G$-set for a finite $H$-set $X$.} is the \emph{coinduction functor} $\tn{coind}^G_H$. If $\varphi: G \twoheadrightarrow G/N$ is a quotient homomorphism for $N \trianglelefteq G$ a closed normal subgroup then $\varphi^*: \spec^{G/N} \to \spec^{G}$ is the \emph{inflation functor} $\tn{infl}^G_{G/N}$, and its right adjoint $\varphi_*$ is the \emph{categorical fixed point functor} $(-)^N: \spec^{G} \to \spec^{G/N}$. 
\end{rec}

We explain what the above functors mean in terms of spectral Mackey functors.

\begin{rem}\label{rem:changeofgroupspectralmackey}
For a closed subgroup $H\leq G$ the functor $\overl{\tn{res}}^G_H: \Spans(G) \to \Spans(H)$ that restricts the $G$-action to an $H$-action induces the adjunction $(\tn{res}^G_H,\tn{coind}^G_H)$ through restriction and left Kan extension
$$((\overl{\tn{res}}^G_H)_!,(\overl{\tn{res}}^G_H)^*): \fun^{\times}(\Spans(G)^{\op},\spec) \rightleftarrows \fun^{\times}(\Spans(H)^{\op},\spec).$$
When $H$ additionally is of finite index (i.e., open) the functor $\overl{\tn{res}}^G_H$ has a left adjoint $\overl{\tn{ind}}^G_H: \Spans(H) \to \Spans(G), \; X \mapsto (G \times X)/H,$ where $H$ acts by $h(g,x)=(gh^{-1},hx)$. For the induced adjunction
$$((\overl{\tn{ind}}^G_H)_!,(\overl{\tn{ind}}^G_H)^*): \fun^{\times}(\Spans(H)^{\op},\spec) \rightleftarrows \fun^{\times}(\Spans(G)^{\op},\spec)$$
we thus have $(\overl{\tn{ind}}^G_H)^* \simeq (\overl{\tn{res}}^G_H)_! \simeq \tn{res}^G_H$ and its left adjoint $(\overl{\tn{ind}}^G_H)_!$ is usually referred to as the \emph{induction functor} $\tn{ind}^G_H$. It is equivalent to $\tn{coind}^G_H$ via the Wirthmüller isomorphism, see \cite[§A.3]{nar17}, \cite[§A]{cmnn24}. \\
Analogously, when $N \trianglelefteq G$ is a closed normal subgroup there is a functor on span categories $\overl{\tn{infl}}^G_{G/N}: \Spans(G/N) \to \Spans(G), \; (G/N)/(H/N) \mapsto G/H$. It induces
$$((\overl{\tn{infl}}^G_{G/N})_!,(\overl{\tn{infl}}^G_{G/N})^*): \fun^{\times}(\Spans(G/N)^{\op},\spec) \rightleftarrows \fun^{\times}(\Spans(G)^{\op},\spec),$$
which recovers the adjunction $(\tn{infl}^G_{G/N},(-)^N)$, see \cite[Example B]{bar17}.
\end{rem}

Let us now fix a finite group $G$ and a commutative equivariant ring spectrum $R \in \calg(\specg)$. We will return to the profinite case in the next section.

\begin{rec}\label{rec:coinddual}
Since any finite $G$-set $T$ is a commutative coalgebra in $\spcg$, the internal mapping spectrum $\bD(T) \defeq \imap_{\specg}(\Sigma^{\infty}_+T, \bS_G)$ has the structure of a commutative algebra object in $\specg$. By \cite[Lemma 3.8]{bchnl25} there is an equivalence 
$$\bD(G/H) \otimes R \stackrel{\sim}{\longrightarrow} \tn{coind}^G_H\tn{res}^G_H(R)$$
in $\calg(\specg)$ which is adjoint to the map $\tn{res}^G_H(\bD(G/H) \otimes R )\to \tn{res}^G_H(R)$ that is obtained from the projection $\tn{res}^G_H(\bD(G/H)) \to \bS_H$ by base change. As $\bD(G/H) \otimes R$ is the image of $\bD(G/H)$ under the free functor $F_R: \specg \to \Mod_R(\specg)$, we can view $\bD(G/H) \otimes R$ as a commutative algebra object of $\Mod_R(\specg)$.
\end{rec}

\begin{rem}
Since $\bD(G/H) \otimes R$ is a dualisable $R$-module with dual $\Sigma^{\infty}_+G/H \otimes R$ there is an equivalence of $R$-modules $\bD(G/H) \otimes R \simeq \imap_{\Mod_R(\specg)}(\Sigma^{\infty}_+G/H \otimes R, R)$, and the underlying $G$-spectrum of this is $\imap_{\specg}(\Sigma^{\infty}_+G/H, R)$.
\end{rem}

The definition for completeness with respect to an algebra object $A$ we will use in this paper is the one of \cite[§2.2]{mnn17}. When $A$ is dualisable, \cite[Proposition 3.11]{mnn17} shows that it agrees with the perhaps more common one that defines complete objects as the twofold right orthogonal of torsion objects (which are the objects in the localising tensor ideal generated by $A$). The dualisability assumption holds in the finite group case.

\begin{defi}\label{defi:borelcomplete}
An $R$-module $X \in \Mod_R(\specg)$ is \emph{Borel complete} if it is complete with respect to the algebra object $\bD(G) \otimes R$. That is, for all $Y \in \Mod_R(\specg)$ with $Y \otimes_R (\bD(G) \otimes R) \simeq 0$ the mapping space $\Map_{\Mod_R(\specg)}(Y,X)$ is contractible. We write $\Mod_R(\specg)_{\tn{Borel}} \subseteq \Mod_R(\specg)$ for the full subcategory on the Borel complete $R$-modules.
\end{defi}

\begin{rem}\label{rem:borelcompleteforget}
The case $R=\bS_G$ is the now classical notion of Borel completeness for $G$-spectra \cite[§6.3]{mnn17}. 
By \cite[§2.2]{mnn17} the subcategory $\Mod_R(\specg)_{\tn{Borel}}$ is presentable and the inclusion $\Mod_R(\specg)_{\tn{Borel}} \hookrightarrow \Mod_R(\specg)$ has a left adjoint $L_{G;R}$ called \emph{Borel completion}. There is a unique symmetric monoidal structure on $\Mod_R(\specg)_{\tn{Borel}}$ that makes $L_{G;R}$ into a symmetric monoidal functor.
\end{rem}

\begin{rem}\label{rem:tensorcondition}
Since the forgetful functor $U_R$ is conservative and the free-forgetful adjunction satisfies a projection formula the condition $Y \otimes_R (\bD(G) \otimes R) \simeq 0$ is equivalent to $U_R(Y) \otimes \bD(G) \simeq 0$.
\end{rem}

\begin{rem}\label{rem:completeunderlyingeasydirection}
If a module $X \in \Mod_R(\specg)$ is Borel complete, then the underlying $G$-spectrum $U_R(X) \in \specg$ is Borel complete. Indeed, assume that $Y \in \specg$ satisfies $Y \otimes \bD(G) \simeq 0$, then $F_R(Y) \otimes_R (\bD(G) \otimes R) \simeq F_R(Y \otimes \bD(G)) \simeq 0$ and hence we have that $\Map_{\specg}(Y,U_R(X)) \simeq \Map_{\Mod_R(\specg)}(F_R(Y),X) \simeq 0$. \\
Similarly, when $X \in \Mod_R(\specg)$ is Borel complete, $U_R(X) \in \specg$ is complete with respect to the algebra object $\bD(G) \otimes R \in \calg(\specg)$. We will show the converse to the first statement in \Cref{cor:completionunderlying}.
\end{rem}

As in the case without coefficients, we can give an explicit description of the Borel completion as a functor $L_{G;R}: \Mod_R(\specg) \to \Mod_R(\specg)_{\tn{Borel}} \hookrightarrow \Mod_R(\specg)$.

\begin{lem}\label{lem:completionexplicit}
Let $X \in \Mod_R(\specg)$. The map $X \to \imap_{\Mod_R(\specg)}(\Sigma^{\infty}_+EG \otimes R, X)$ exhibits the target as the Borel completion of $X$.
\end{lem}
\begin{proof}
By \cite[Proposition 2.21]{mnn17} the completion $L_{G;R}(X)$ is computed as $\tn{Tot}(X \otimes_R \tn{CB}^{\bullet}(\bD(G) \otimes R))$, the totalisation of the cobar construction of $\bD(G) \otimes R$ tensored with $X$. We now note that there are equivalences
\begin{align*}
X \otimes_R (\bD(G) \otimes R)^{\otimes_R \bullet} &\simeq X \otimes_R (\imap_{\Mod_R(\specg)}(\Sigma^{\infty}_+G^{\bullet}, \bS_G) \otimes R) \\
&\simeq X \otimes_R \imap_{\Mod_R(\specg)}(\Sigma^{\infty}_+G^{\bullet} \otimes R, R) \\
&\simeq \imap_{\Mod_R(\specg)}(\Sigma^{\infty}_+G^{\bullet} \otimes R, X).
\end{align*}
Totalising gives $\imap_{\Mod_R(\specg)}(\Sigma^{\infty}_+EG \otimes R, X)$, using the standard simplicial decomposition of the free contractible $G$-space $EG$ as in \cite[Proposition 6.6]{mnn17}.
\end{proof}

\begin{cor}\label{cor:completionunderlying}
A module $X \in \Mod_R(\specg)$ is Borel complete if and only if the underlying $G$-spectrum $U_R(X)$ is.
\end{cor}
\begin{proof}
By \Cref{lem:completionexplicit} $X$ is Borel complete if and only if the completion map 
$$X \to \imap_{\Mod_R(\specg)}(\Sigma^{\infty}_+EG \otimes R, X)$$ 
is an equivalence. Since $U_R$ is conservative, this happens if and only if the map 
$$U_R(X) \to U_R(\imap_{\Mod_R(\specg)}(\Sigma^{\infty}_+EG \otimes R, X)) \simeq \imap_{\specg}(\Sigma^{\infty}_+EG, U_R(X))$$ is an equivalence. But this is precisely the completion map for $U_R(X)$.
\end{proof}

Recall from \cite[Proposition 6.17]{mnn17} that the $\infty$-category of Borel complete $G$-spectra has a particularly simple description as spectra with a $G$-action. We can generalise this to the situation with coefficients. 

\begin{lem}\label{lem:resasbasechange}
There is a symmetric monoidal equivalence 
$$\Mod_{\bD(G)\otimes R}(\specg) \stackrel{\sim}{\longrightarrow} \Mod_{\tn{res}^G_1(R)}(\spec),$$
and the equivalence functor fits into the following commutative diagram:
\[\begin{tikzcd}
	{\Mod_{R}(\specg)} & \\
	{\Mod_{\tn{res}^G_1(R)}(\spec)} & {\Mod_{\bD(G)\otimes R}(\specg)}
	\arrow["{\tn{res}^G_1}"', from=1-1, to=2-1]
	\arrow["{F_{\bD(G) \otimes R}}", from=1-1, to=2-2]
	\arrow["\sim"', from=2-2, to=2-1]
\end{tikzcd}\]
\end{lem}
\begin{proof}
The case for $R = \bS$ is \cite[Theorem 5.32]{mnn17} (see also \cite[Theorem 1.1]{bds15} for the original version), and a more general statement than the one above is shown in \cite[Proposition 3.9]{bchnl25}.
\end{proof}

\begin{rem}\label{rem:nonequivariantlycontractible}
When $R = \bS_G$, \Cref{lem:resasbasechange} shows that for a $G$-spectrum $Y$ the condition $Y \otimes \bD(G) \simeq 0$  is equivalent to $Y$ being \emph{non-equivariantly contractible} (that is, $\tn{res}^G_1(Y) \simeq 0$) \cite[Proposition 6.15]{mnn17}. Since $\bD(G)$ is dualisable these conditions also agree with the one that $Y$ is \emph{$\bD(G)^{-1}$-local} \cite[§3.2]{mnn17}. \\
More generally, if $Y$ is an $R$-module we have $Y \otimes_R (\bD(G) \otimes R) \simeq 0$ if and only if $\tn{res}^G_1(Y) \simeq 0$ in $\Mod_{\tn{res}^G_1(R)}(\spec)$, which holds if and only if $\tn{res}^G_1(U_R(Y)) \simeq U_R(\tn{res}^G_1(Y)) \simeq 0$ in $\spec$.
\end{rem}

\begin{prop}\label{prop:borelcompleteasfun}
There is an equivalence of symmetric monoidal $\infty$-categories $$\Mod_R(\specg)_{\tn{Borel}} \simeq \fun(BG,\Mod_{\tn{res}^G_1(R)}(\spec)),$$
where the right-hand side carries the pointwise monoidal structure.
\end{prop}
\begin{proof}
This is the same proof as in \cite[Proposition 6.17]{mnn17}, where the statement is shown for $R = \bS_G$. We note that $\bD(G) \otimes R$ is a dualisable object of $\Mod_R(\specg)$ since it is the image of the dualisable object $\bD(G)$ under the free functor. Hence by \cite[Theorem 2.30]{mnn17} the $\infty$-category of $(\bD(G) \otimes R)$-complete objects is computed as the totalisation of the cosimplicial extension of scalars diagram
\begin{equation}\label{diag:totalisation}
\begin{tikzcd}
	{\Mod_{\bD(G)\otimes R}(\specg)} & {\Mod_{\bD(G \times G) \otimes R}(\specg)} & {\cdots}
	\arrow[shift right=2, from=1-1, to=1-2]
	\arrow[shift left=2, from=1-1, to=1-2]
	\arrow[from=1-2, to=1-1]
	\arrow[from=1-2, to=1-3]
	\arrow[shift left=4, from=1-2, to=1-3]
	\arrow[shift right=4, from=1-2, to=1-3]
	\arrow[shift right=2, from=1-3, to=1-2]
	\arrow[shift left=2, from=1-3, to=1-2]
\end{tikzcd}
\end{equation}
which is obtained by applying the module functor to the cobar construction for the algebra object $\bD(G) \otimes R \in \calg(\Mod_R(\specg))$. By \Cref{lem:resasbasechange} we can identify $\Mod_{\bD(G)\otimes R}(\specg)$ with $\Mod_{\tn{res}^G_1(R)}(\spec)$, and similarly $\Mod_{\bD(G^n)\otimes R}(\specg)$ identifies with 
$$\Mod_{\imap_{\spec}(\Sigma^{\infty}_+G^{n-1}, \,\bS) \otimes \tn{res}^G_1(R)}(\spec) \simeq \prod_{G^{n-1}} \Mod_{\tn{res}^G_1(R)}(\spec).$$
The cosimplicial diagram \ref{diag:totalisation} then becomes
\[\begin{tikzcd}
	{\Mod_{\tn{res}^G_1(R)}(\spec)} & {\prod_{G} \Mod_{\tn{res}^G_1(R)}(\spec)} & \cdots
	\arrow[shift right=2, from=1-1, to=1-2]
	\arrow[shift left=2, from=1-1, to=1-2]
	\arrow[from=1-2, to=1-1]
	\arrow[from=1-2, to=1-3]
	\arrow[shift left=4, from=1-2, to=1-3]
	\arrow[shift right=4, from=1-2, to=1-3]
	\arrow[shift right=2, from=1-3, to=1-2]
	\arrow[shift left=2, from=1-3, to=1-2]
\end{tikzcd}\]
and its totalisation is precisely the functor category $\fun(BG,\Mod_{\tn{res}^G_1(R)}(\spec))$, using the standard bar construction simplicial diagram that realises to $BG$. 
\end{proof}

\begin{rem}\label{rem:forgetfulonborelcompletes}
The observation made in \Cref{rem:borelcompleteforget} says that the forgetful functor $U_R: \Mod_R(\specg) \to \specg$ induces a functor $\Mod_R(\specg)_{\tn{Borel}} \to \specg_{\tn{Borel}}$, which we can equivalently describe as the postcomposition
$$\fun(BG,\Mod_{\tn{res}^G_1(R)}(\spec)) \to \fun(BG,\spec)$$
with the forgetful functor $U_{\tn{res}^G_1(R)}: \Mod_{\tn{res}^G_1(R)}(\spec) \to \spec$.
\end{rem}

In particular, we see that the Borel complete subcategory of $\Mod_R(\specg)$ only depends on the underlying non-equivariant spectrum $\tn{res}^G_1(R)$ of $R$. We can make this precise as follows.

\begin{cor}\label{cor:justrestriction}
There is an equivalence of symmetric monoidal $\infty$-categories
$$\Mod_R(\specg)_{\tn{Borel}} \simeq \Mod_{\tn{infl}^G_{G/G}\tn{res}^G_1(R)}(\specg)_{\tn{Borel}}.$$
\end{cor}
\begin{proof}
Since $\tn{res}^G_1\tn{infl}^G_{G/G} \simeq \tn{id}$ as symmetric monoidal functors $\spec \to \spec$ there is an equivalence of algebra objects $\tn{res}^G_1(R) \simeq \tn{res}^G_1\tn{infl}^G_1\tn{res}^G_1(R)$. The statement then follows from \Cref{prop:borelcompleteasfun}.
\end{proof}

\begin{ex}\label{ex:manylambdas}
For $\Lambda$ a discrete commutative ring, we can consider the Eilenberg-MacLane spectrum associated to its constant Mackey functor $\underl{\Lambda}$, the Eilenberg-MacLane spectrum associated to the $\Lambda$-linear Burnside Mackey functor $\bA_{\Lambda}$, and the inflation $\Lambda_G$ of its non-equivariant Eilenberg-MacLane spectrum to $\specg$. Then all of the three Borel complete categories associated to these equivariant ring spectra are equivalent to $\fun(BG,\Mod_{\Lambda}(\spec))$.
\end{ex}

We will hence restrict our study of the Borel complete category with coefficients to coefficients of the from $R_G \defeq \tn{infl}^G_{G/G}(R)$, where $R \in \calg(\spec)$ is a non-equivariant ring spectrum. Note that this of course includes the case $R_G = \bS_G$.

\begin{rem}
Consider the adjunction $(L_{G;\bS_G},\iota): \specg \rightleftarrows \specg_{\tn{Borel}}$. In view of \Cref{prop:borelcompleteasfun} we can express it in terms of spectral Mackey functors as
$$(\theta^*,\theta_*): \fun^{\times}(\Spans(G)^{\op},\spec) \rightleftarrows \fun(BG,\spec),$$
where $\theta^*$ restricts along the map $\theta: BG \to \Spans(G)^{\op}$ that picks out the orbit $G/1$, hence $\theta^*(X)$ is given by the mapping spectrum $\map_{\specg}(\Sigma^\infty_+G/1,X)$, and the functor $\theta_*$ is given by right Kan extension \cite[§8]{bgs20}. One may also view the functor $\theta^*$ as the one that the restriction $\tn{res}^G_1: \specg \to \spec$ factors through by remembering the $G$-action on the section of the orbit $G/1$. \\
The functor $\theta_!$ given by left Kan extension is fully faithful as well and has essential image the $G$-spectra with geometric fixed points concentrated at the trivial group \cite[§2]{hau24}. By \cite[§4]{psw22} for $R=\bZ$ and \cite[§14]{bhs23} for the general case there is a symmetric monoidal equivalence 
$$\Mod_{R_G}(\specg) \simeq \fun^{\times}(\Spans(G)^{\op},\Mod_R(\spec)),$$
and we can express the adjunction $(L_{G;R},\iota): \Mod_{R_G}(\specg) \rightleftarrows \Mod_{R_G}(\specg)_{\tn{Borel}}$ as one of functor categories
$$(\theta^*,\theta_*): \fun^{\times}(\Spans(G)^{\op},\Mod_R(\spec)) \rightleftarrows \fun(BG,\Mod_R(\spec)),$$
where the functor $\theta: BG \to \Spans(G)^{\op}$ is as before. Now $\theta^*$ can be expressed as the mapping $R$-module spectrum $\map_{\Mod_R(\specg)}(\Sigma^\infty_+G,-)$. In view of \Cref{rem:forgetfulonborelcompletes} there is a diagram
\[\begin{tikzcd}
	{\fun^{\times}(\Spans(G)^{\op},\Mod_R(\spec))} & {\fun^{\times}(\Spans(G)^{\op},\spec)} \\
	{\fun(BG,\Mod_R(\spec))} & {\fun(BG,\spec)}
	\arrow["{U_R}", from=1-1, to=1-2]
	\arrow[""{name=0, anchor=center, inner sep=0}, "{\theta^*}"', shift right=2, from=1-1, to=2-1]
	\arrow[""{name=1, anchor=center, inner sep=0}, "{\theta^*}"', shift right=2, from=1-2, to=2-2]
	\arrow[""{name=2, anchor=center, inner sep=0}, "{\theta_*}"', shift right=2, hook, from=2-1, to=1-1]
	\arrow["{U_R}", from=2-1, to=2-2]
	\arrow[""{name=3, anchor=center, inner sep=0}, "{\theta_*}"', shift right=2, hook, from=2-2, to=1-2]
	\arrow["\dashv"{anchor=center}, draw=none, from=0, to=2]
	\arrow["\dashv"{anchor=center}, draw=none, from=1, to=3]
\end{tikzcd}\]
in which both squares commute. The horizontal functors postcompose with the forgetful functor $\Mod_{R}(\spec) \to \spec$. \Cref{cor:completionunderlying} says that $X$ in the top left is in the essential image of $\theta_*$ if and only if $U_R(X)$ is in the essential image of $\theta_*$.
\end{rem}

\begin{cor}\label{cor:tstructureonborelcomplete}
When the coefficients $R$ are connective, there is a t-structure on $\Mod_{R_G}(\specg)_{\tn{Borel}}$ which makes the localisation functor $L_{G;R}$ t-exact. The heart is given by the category $\Mod(G;\pi_0(R))$ of $\pi_0(R)$-modules with a linear $G$-action. 
\end{cor}
\begin{proof}
First note that $\Mod_R(\spec)$ has a t-structure which is detected on the level of underlying spectra \cite[Proposition 7.1.1.13]{lur17}. Then the functor category $\fun(BG,\Mod_R(\spec))$ inherits a t-structure which is detected pointwise. Similarly, the t-structure on $\fun^{\times}(\Spans(G)^{\op},\Mod_R(\spec))$ is detected pointwise after forgetting the $R$-module structure \cite[§6]{bgs20}. It follows that the restriction functor $\theta^*$ is t-exact. For the statement about the heart we use \cite[Proposition 7.1.1.13]{lur17} and the equivalence $\fun(BG,\Mod_R(\spec))^{\heartsuit} \simeq \fun(BG,\Mod_R(\spec)^{\heartsuit})$.
\end{proof}

\begin{rem}
One can also check by hand that the essential image of the two aisles $\Mod_{R_G}(\spec^{G})_{\geq 0}$ and $\Mod_{R_G}(\spec^{G})_{\leq 0}$ under the Borel completion functor $L_{G;R}$ define a t-structure on $\Mod_{R_G}(\spec^{G})_{\tn{Borel}}$, the key ingredient being that the composition $\iota\circ L_{G;R}$ is left t-exact: whenever $X \in \Mod_{R_G}(\spec^{G})$ is coconnective, the completion 
$$\imap_{\Mod_{R_G}(\specg)}(\Sigma^{\infty}_+EG \otimes R_G, X)$$
is coconnective as well, since the equivariant infinite loop space of its underlying $G$-spectrum is the mapping space $\Map_{\specg}(\Sigma^{\infty}_+EG, U_R(X))$, which is discrete since $\Sigma^{\infty}_+EG$ is connective.
In more categorical terms, $\Mod_R(\specg)_{\tn{Borel}}$ takes part in a recollement \cite[Definition A.8.1, Remark A.8.19]{lur17}
\[\begin{tikzcd}
	{\Mod_{R_G}(\specg)[(\bD(G)\otimes R_G)^{-1}]} & {\Mod_{R_G}(\spec^{G})} & {\Mod_{R_G}(\specg)_{\tn{Borel}}.}
	\arrow[""{name=0, anchor=center, inner sep=0}, shift right=2, hook, from=1-1, to=1-2]
	\arrow[""{name=1, anchor=center, inner sep=0}, shift right=2, from=1-2, to=1-1]
	\arrow[""{name=2, anchor=center, inner sep=0}, "{L_{G;R_G}}", shift left=2, from=1-2, to=1-3]
	\arrow[""{name=3, anchor=center, inner sep=0}, "{\iota}", shift left=2, hook', from=1-3, to=1-2]
	\arrow["\dashv"{anchor=center, rotate=-90}, draw=none, from=2, to=3]
	\arrow["\dashv"{anchor=center, rotate=-90}, draw=none, from=1, to=0]
\end{tikzcd}\]
On the left is the full subcategory of $\Mod_{R_G}(\spec^{G})$ on $(\bD(G)\otimes R_G)^{-1}$-local objects which identifies with the kernel of $L_{G;R_G}$, see \cite[Proposition 3.11]{mnn17} and \cite[§1]{npr24}. Then the t-structure on $\Mod_{R_G}(\specg)_{\tn{Borel}}$ from above is restricted from the one on $\Mod_{R_G}(\spec^{G})$, in the sense of \cite[Proposition 1.4.12]{bbd81}.
\end{rem}

\begin{rem}
By \cite[Proposition 4.14]{npr24} the recollement of the previous remark is base changed from the one for $G$-spectra along $F_{R_G}: \specg \to \Mod_{R_G}(\specg)$. In particular, the completion functor 
$$L_{G;R_G}: \Mod_{R_G}(\spec^{G}) \to \Mod_{R_G}(\specg)_{\tn{Borel}}$$ 
identifies with the functor
$$\Mod_{R_G}(\spec^{G}) \otimes_{\specg} L_{G;\bS_G}: \Mod_{R_G}(\spec^{G}) \to \Mod_{R_G}(\spec^{G}) \otimes_{\specg} \specg_{\tn{Borel}}.$$
\end{rem}

Before turning our attention to the case of profinite groups, we establish some change of group results we will use later.

\begin{lem}\label{lem:categoricalfixedpointsborelcomplete}
Let $N \unlhd G$ be a normal subgroup. Denote by 
$$(-)^N:\Mod_{R_G}(\specg) \to \Mod_{R_{G/N}}(\spec^{G/N})$$ 
the categorical fixed point functor that is right adjoint to the functor induced by inflation $\tn{infl}^G_{G/N}: \spec^{G/N} \to \specg$ on module categories. Then $(-)^N$ restricts to a functor 
$$\Mod_{R_G}(\specg)_{\tn{Borel}} \to \Mod_{R_{G/N}}(\spec^{G/N})_{\tn{Borel}}.$$
\end{lem}
\begin{proof}
Let $X \in \Mod_{R_G}(\specg)$ be Borel complete and let $Y \in \Mod_{R_{G/N}}(\specg)$ such that $\tn{res}^{G/N}_1(Y) \simeq 0$. Then $\tn{res}^G_1\tn{infl}^G_{G/N}(Y) \simeq \tn{res}^{G/N}_1(Y) \simeq 0$, and hence
$$\Map_{\Mod_{R_{G/N}}(\spec^{G/N})}(Y,X^N) \simeq \Map_{\Mod_{R_{G}}(\specg)}(\tn{infl}^G_{G/N}(Y),X) \simeq 0,$$
so $X^N \in \Mod_{R_{G/N}}(\specg)$ is Borel complete.
\end{proof}

We denote the functor induced by categorical fixed points by
$$(-)^{hN}:\Mod_{R_G}(\specg)_{\tn{Borel}} \to \Mod_{R_{G/N}}(\spec^{G/N})_{\tn{Borel}}.$$ 
We will explain the notation that alludes to homotopy fixed points below. It follows formally that $(-)^{hN}$ has a left adjoint given by 
$$h\tn{infl}^G_{G/N} \defeq L_{G;R}\circ \tn{infl}^G_{G/N}\circ \iota: \Mod_{R_{G/N}}(\spec^{G/N})_{\tn{Borel}} \to \Mod_{R_G}(\specg)_{\tn{Borel}}.$$
We hence obtain a diagram in which the respective squares of left and right adjoints commute:
\begin{equation}\label{diag:fixedpointsinflationborel}
\begin{tikzcd}
	{\Mod_{R_G}(\specg)} & {\Mod_{R_G}(\specg)_{\tn{Borel}}} \\
	{\Mod_{R_{G/N}}(\spec^{G/N})} & {\Mod_{R_{G/N}}(\spec^{G/N})_{\tn{Borel}}.}
	\arrow[""{name=0, anchor=center, inner sep=0}, "{L_{G;R}}", shift left=2, from=1-1, to=1-2]
	\arrow[""{name=1, anchor=center, inner sep=0}, "{(-)^N}", shift left=2, from=1-1, to=2-1]
	\arrow[""{name=2, anchor=center, inner sep=0}, shift left=2, hook', from=1-2, to=1-1]
	\arrow[""{name=3, anchor=center, inner sep=0}, "{(-)^{hN}}", shift left=2, from=1-2, to=2-2]
	\arrow[""{name=4, anchor=center, inner sep=0}, "{\tn{infl}^G_{G/N}}", shift left=2, from=2-1, to=1-1]
	\arrow[""{name=5, anchor=center, inner sep=0}, "{L_{G/N;R}}", shift left=2, from=2-1, to=2-2]
	\arrow[""{name=6, anchor=center, inner sep=0}, "{h\tn{infl}^G_{G/N}}", shift left=2, from=2-2, to=1-2]
	\arrow[""{name=7, anchor=center, inner sep=0}, shift left=2, hook', from=2-2, to=2-1]
	\arrow["\dashv"{anchor=center, rotate=-90}, draw=none, from=0, to=2]
	\arrow["\dashv"{anchor=center, rotate=-90}, draw=none, from=5, to=7]
	\arrow["\dashv"{anchor=center}, draw=none, from=4, to=1]
	\arrow["\dashv"{anchor=center}, draw=none, from=6, to=3]
\end{tikzcd}
\end{equation}

It is a well known fact that on Borel complete $G$-spectra categorical fixed points and homotopy fixed points of the underlying spectrum with a $G$-action agree, and in fact this characterises Borel complete $G$-spectra \cite[Proposition 6.19]{mnn17}. Taking residual actions into account, we make this precise in our setting as follows.

\begin{prop}\label{prop:categoricalfixedpointsasrke}
Under the equivalence of \Cref{prop:borelcompleteasfun}, the adjunction 
$$(h\tn{infl}^G_{G/N}, (-)^{hN}):\Mod_{R_{G/N}}(\spec^{G/N})_{\tn{Borel}} \rightleftarrows \Mod_{R_G}(\specg)_{\tn{Borel}}$$
of \Cref{lem:categoricalfixedpointsborelcomplete} identifies with the adjunction
$$(\beta^*,\beta_*): \fun(B(G/N),\Mod_R(\spec)) \rightleftarrows \fun(BG,\Mod_R(\spec)),$$
where $\beta^*$ is given by restriction along the canonical map $\beta: BG \to B(G/N)$ induced by the quotient homomorphism $G \twoheadrightarrow G/N$ and $\beta_*$ is given by right Kan extension.
\end{prop}
\begin{proof}
We show the commutativity of the following diagram.
\begin{equation}\label{diag:inflationcompletion}
\begin{tikzcd}[scale cd=0.96]
	{\Mod_{R_G}(\specg)} & {\Mod_{R_G}(\specg)_{\tn{Borel}} \simeq \fun(BG,\Mod_R(\spec))} \\
	{\Mod_{R_{G/N}}(\spec^{G/N})} & {\Mod_{R_{G/N}}(\spec^{G/N})_{\tn{Borel}} \simeq \fun(B(G/N),\Mod_R(\spec))}
	\arrow["{L_{G;R}}", from=1-1, to=1-2]
	\arrow["{\tn{infl}^G_{G/N}}"', from=2-1, to=1-1]
	\arrow["{L_{G/N;R}}", from=2-1, to=2-2]
	\arrow["{\beta^*}"', shift right=5, from=2-2, to=1-2]
\end{tikzcd}
\end{equation}
Recall that in terms of spectral Mackey functors the top and bottom horizontal compositions are given by restriction along the maps $\theta_G: BG \to \Spans(G)^\op$ and $\theta_{G/N}: B(G/N) \to \Spans(G/N)^\op$ that pick out the orbits $G/1$ and $(G/N)/(N/N)$ respectively. Then the restriction functor $\tn{res}^G_1: \Mod_{R_G}(\specg) \to \Mod_R(\spec)$ factors as $\tn{ev}_G\,\theta_G^*$, where $\tn{ev}_G: \fun(BG,\Mod_R(\spec)) \to \Mod_R(\spec)$ evaluates at the unique object of $BG$, and we similarly have $\tn{res}^{G/N}_1 \simeq \tn{ev}_{G/N}\circ \theta_{G/N}^*$. \\
Consider the unit map $\tn{id} \to (-)^N\tn{infl}^G_{G/N}$ of endofunctors of $\Mod_{R_{G/N}}(\spec^{G/N})$. Composing with $\theta_{G/N}^*$ gives a map
\begin{align}\label{eq:thetastransformation}
\theta_{G/N}^* \to \theta_{G/N}^* (-)^N\tn{infl}^G_{G/N} \simeq  \tilde{\theta}_{G/N}^*\tn{infl}^G_{G/N},
\end{align}
where $\tilde{\theta}_{G/N}: B(G/N) \to \Spans(G)^\op$ picks out $G/N$. Here we use the description of categorical fixed points as a restriction along a functor of span categories, see \Cref{rem:changeofgroupspectralmackey}. Now, evaluating spectral $G$-Mackey functors on the map $G/1 \twoheadrightarrow G/N$ provides a natural transformation $\beta^*\tilde{\theta}_{G/N}^* \to \theta_{G}^*$. Through composition with \ref{eq:thetastransformation} we obtain a natural transformation 
\begin{align}\label{eq:finalformation}
\beta^*\theta_{G/N}^* \to \theta_{G}^*\tn{infl}^G_{G/N}.
\end{align}
By the above considerations and using $\tn{ev}_G \, \beta^* \simeq \tn{ev}_{G/N}$, composing this natural transformation with $\tn{ev}_G: \fun(BG,\Mod_R(\spec)) \to \Mod_R(\spec)$ gives a natural transformation $\tn{res}^G_1\tn{infl}^G_{G/N} \to \tn{res}^{G/N}_1$ which is an equivalence. Since $\tn{ev}_G$ is conservative, \ref{eq:finalformation} is an equivalence as well.
\end{proof}

\begin{rem}
Denote by $\alpha_!$ the left Kan extension along the inflation on span categories $\alpha: \Spans(G/N) \to \Spans(G), \; (G/N)/(H/N) \mapsto G/H$. We can illustrate the commutativity of the square of left adjoints
\[\begin{tikzcd}
	{\fun^{\times}(\Spans(G)^{\op},\Mod_R(\spec))} & {\fun(BG,\Mod_R(\spec))} \\
	{\fun^{\times}(\Spans(G/N)^{\op},\Mod_R(\spec))} & {\fun(B(G/N),\Mod_R(\spec)),}
	\arrow["{\theta^*_G}", from=1-1, to=1-2]
	\arrow["{\alpha_!}"', from=2-1, to=1-1]
	\arrow["{\theta_{G/N}^*}", from=2-1, to=2-2]
	\arrow["{\beta^*}"', from=2-2, to=1-2]
\end{tikzcd}\]
on objects as follows: For a spectral $G/N$-Mackey functor $\cF$ we obtain the restriction $\beta^*\theta_{G/N}^*(\cF) \simeq \cF(G/N)$ which has a $G$-action through $G \twoheadrightarrow G/N$. On the other hand, $\alpha_!$ (using \cite[Proposition 3.3.2]{chll24}) sends $\cF$ to the functor that assigns
$$G/H \mapsto \tn{colim}\left(\Spans(G/N)^\op \times_{\Spans(G)^\op} (\Spans(G)^\op)_{\tn{\big /}(G/H)} \stackrel{\cF}{\longrightarrow} \spec \right),$$
and $\theta^*_G$ evaluates this at the orbit $G/1$, for which we get the indexing category 
$$\Spans(G/N)^\op \times_{\Spans(G)^\op} (\Spans(G)^\op)_{\tn{\big /}(G/1)}.$$
By the dual of \cite[{}p.106]{bh21} it has a reflective subcategory
$$G/N\tsf{-set}^\op \times_{\gset^\op} (\gset^\op)_{\tn{\big /}(G/1)},$$
and the latter has a terminal object $G/1 \to G/N$, the canonical projection. As right adjoints are cofinal \cite[Corollary 7.2.3.7]{kerodon} we have $\theta^*_G\alpha_!(\cF) \simeq \cF(G/N)$.
\end{rem}

This now justifies our notation of the induced functors on Borel complete subcategories as \emph{homotopy} fixed points $(-)^{hN}$ and \emph{homotopy} inflation $h\tn{infl}^G_{G/N}$, since in terms of spectra with an action of $G$ (resp. $G/N$) these functors take the limit over $BN$ (resp. inflate the $G/N$-action to a $G$ action via $G \twoheadrightarrow G/N$), which is what is known as homotopy fixed points and inflation. \\
Concretely, if $X$ is a spectrum with a $G$-action its homotopy $N$-fixed points are modelled by the totalisation of the cosimplicial object 
\[\begin{tikzcd}
	X & {\prod_N X} & {\prod_{N^2} X \;\; \cdots}
	\arrow[shift left, from=1-1, to=1-2]
	\arrow[shift right, from=1-1, to=1-2]
	\arrow[from=1-2, to=1-3]
	\arrow[shift left=2, from=1-2, to=1-3]
	\arrow[shift right=2, from=1-2, to=1-3]
\end{tikzcd}\]
which has a residual $G/N$-action. The resulting spectrum $X^{hN}$ can again be viewed as a cosimplicial object over the opposite of the diagram that computes $B(G/N)$. 

\begin{cor}
The functor $h\tn{infl}^G_{G/N}: \Mod_{R_{G/N}}(\spec^{G/N})_{\tn{Borel}} \to \Mod_{R_G}(\specg)_{\tn{Borel}}$ is symmetric monoidal and t-exact with respect to the t-structure of \Cref{cor:tstructureonborelcomplete}.
\end{cor}
\begin{proof}
Since $\beta^*: \fun(B(G/N),\Mod_R(\spec)) \to \fun(BG,\Mod_R(\spec))$ has these properties this follows immediately.
\end{proof}

\begin{cor}
For $M \in \Mod_R(\spec)$, the Borel complete $R_G$-module 
$$\imap_{\Mod_R(\specg)}(\Sigma^{\infty}_+EG \otimes R, \tn{infl}^G_{G/N}(M))$$
obtained by completing the inflation of $M$ to $G$-spectra\footnote{In \cite[Example 6.20]{mnn17} this is denoted $\underl{M}$, we do not use this notation to avoid confusion with constant Mackey functors. When $R_G = \bS_G$ this $G$-spectrum represents Borel equivariant $M$-cohomology on $G$-spaces.} is canonically equivalent to the right Kan extension $\theta_*(M)$, where now $M$ is viewed as an object of $\fun(BG,\Mod_R(\spec))$ with trivial $G$-action.
\end{cor}

\section{...and for profinite groups}\label{sec:borelprofinite}

The aim of this section is to develop a theory of Borel completeness for profinite groups. We first consider a `naive' notion which we call \emph{levelwise} Borel completeness - categorically this is the limit of the Borel complete subcategories along fixed points, running over all open normal subgroups of the group. We then introduce a second notion that is given by completeness with respect to the colimit of the (duals of the) orbits of all open normal subgroups. We finally identify the latter notion with the hypercompletion of the former, in a t-structure sense. \\
Throughout this section $G$ is a profinite group. We write $(N_i)_{i \in I}$ for the cofiltered system of open normal subgroups of $G$ under inclusion and let $G_i \defeq G/N_i$, so that $G = \lim_i G_i$. If $H \leq G$ is a closed subgroup, then we have $H \cong \lim_i H_i$ and $H$ is itself profinite, since $H_i \defeq HN_i/N_i \cong H/(H \cap N_i)$ is a subgroup of the finite group $G_i$. As before, $R \in \calg(\spec)$ is a commutative non-equivariant ring spectrum which we now assume to be connective. Again we write $R_{K} \defeq \tn{infl}^K_{K/K}(R)$ for its inflation to a group $K$, whenever this makes sense.

\begin{rec}\label{rec:specgascolimit}
There is a functor 
$$\spec^{G_{\bullet}}: I^\op \to \calg(\prlst), \; i \mapsto \spec^{G_i}, \; (N_i \subseteq N_j) \mapsto (\tn{infl}_{G_j}^{G_i}: \spec^{G_j} \to \spec^{G_i}),$$
and as explained in \cite[§6]{bbb24}, \cite[Corollary 5.7]{bt26} its colimit is equivalent to $\specg$. Equivalently, the latter is the limit of the adjoint diagram of categorical fixed point functors in $\prr$ (using that the forgetful functor $\calg(\prlst) \to \prl$ preserves filtered colimits \cite[Remark 4.8.1.24, Corollary 3.2.3.2]{lur17}). As before, the canonical left adjoints into the colimit are denoted $\tn{infl}^G_{G_i}: \spec^{G_i} \to \specg$, and their right adjoints $(-)^{N_i}: \specg \to \spec^{G_i}$. In the presence of coefficients there is an analogous functor
$$\Mod_{R_{G_{\bullet}}}(\spec^{G_{\bullet}}): I^\op \to \calg(\prlst)$$
which sends $(N_i \subseteq N_j)$ to the induced inflation on module categories. Its colimit is equivalent to $\Mod_{R_G}(\specg)$. This follows from \cite[Lemma 3.5.6]{agv22} or by combining \cite[Theorem 5.6]{bt26} with \cite[§14]{bhs23}. We employ the same notation as above for the canonical left adjoints into this colimit and their right adjoints. \\
In both cases the colimit can also be computed in $\calg(\prl_{\omega,\tn{st}})$, the $\infty$-category $\Mod_{R_{G}}(\spec^{G})$ is rigidly compactly generated by the free modules associated to the orbits of open normal subgroups. Because of $\calg(\prl_{\omega,\tn{st}}) \simeq \calg(\catinf^{\tn{perf}})$ and by \cite[Lemma 7.3.5.1]{lur17} the filtered colimit of the compact parts is computed in $\catinf$. In particular, a compact $R_{G}$-module is inflated from $G_i$ for some $i \in I$.
\end{rec}

\begin{rem}
Whenever we write down diagrams of $\infty$-categories of equivariant spectra as above we use the fact that there is a functor
$$\spec^{(-)}: \glo^\op \to \calg(\prlst), \; L \mapsto \spec^L, \; (\alpha: L \to L') \mapsto (\alpha^*: \spec^{L'} \to \spec^L),$$
where $\glo$ the global category of finite groups, group homomorphisms and conjugations,\footnote{The notation $\glo$ is also used in the literature for the global category of compact Lie groups.} see e.g. \cite[§10]{lnp25}. It is a (2,1)-category which we view as an $\infty$-category via the Duskin nerve. We can then restrict this functor to the 1-category $I^\op$ to obtain the desired filtered diagram.
\end{rem}

\begin{defi}\label{defi:levelwiseborelcomplete}
We define the $\infty$-category of \emph{levelwise Borel complete $R_G$-modules} as the $\calg(\prlst)$-colimit 
$$\Mod_{R_G}(\specg)_{\tn{lwBorel}} \defeq \tn{colim}_i \; \Mod_{R_{G_i}}(\spec^{G_i})_{\tn{Borel}}$$
along the homotopy inflations $h\tn{infl}^{G_i}_{G_j}:\Mod_{R_{G_j}}(\spec^{G_j})_{\tn{Borel}} \to \Mod_{R_{G_i}}(\spec^{G_i})_{\tn{Borel}}$ for each inclusion $N_i \subseteq N_j$.
\end{defi}

We denote the canonical functors $\Mod_{R_{G_i}}(\spec^{G_i})_{\tn{Borel}} \to \Mod_{R_G}(\specg)_{\tn{lwBorel}}$ by $h\tn{infl}^G_{G_i}$ and write $(-)^{hN_i}$ for their right adjoints.

\begin{rem}
After forgetting the monoidal structure, $\Mod_{R_G}(\specg)_{\tn{lwBorel}}$ is given by the $\prrst$-limit 
$$\tn{lim}_i \; \Mod_{R_{G_i}}(\spec^{G_i})_{\tn{Borel}}$$ 
along homotopy fixed points $(-)^{hN_{ij}}:\Mod_{R_{G_i}}(\spec^{G_i})_{\tn{Borel}} \to \Mod_{R_{G_j}}(\spec^{G_j})_{\tn{Borel}}$, where $N_{ij} = N_j/N_i = \tn{ker}(G/N_i \twoheadrightarrow (G/N_i)/(N_j/N_i) = G/N_j)$. As a limit of the fully faithful functors $\Mod_{R_{G_i}}(\spec^{G_i})_{\tn{Borel}} \hookrightarrow \Mod_{R_{G_i}}(\spec^{G_i})$ the induced functor
$$\iota_{\tn{lw}}: \Mod_{R_G}(\specg)_{\tn{lwBorel}} \to \Mod_{R_G}(\specg)$$
is fully faithful as well. We can hence view an object of $\Mod_{R_G}(\specg)_{\tn{lwBorel}}$ as a compatible collection of $R_{G_i}$-modules $(X_i)_{i \in I}$ (i.e., $(X_i)^{N_{ij}} \simeq X_j$) such that each $X_i$ is Borel complete. We denote the symmetric monoidal left adjoint of $\iota_{\tn{lw}}$ by $L_{\tn{lw}}$ and call it \emph{levelwise Borel completion}. It makes the following square commute:
\begin{equation}\label{diag:levelwisecompletion}
\begin{tikzcd}
	{\Mod_{R_{G_i}}(\spec^{G_i})} & {\Mod_{R_{G_i}}(\spec^{G_i})_{\tn{Borel}}} \\
	{\Mod_{R_{G}}(\spec^{G})} & {\Mod_{R_G}(\specg)_{\tn{lwBorel}}.}
	\arrow["{L_{G_i;R_{G_i}}}", from=1-1, to=1-2]
	\arrow["{\tn{infl}^G_{G_i}}", from=1-1, to=2-1]
	\arrow["{h\tn{infl}^G_{G_i}}", from=1-2, to=2-2]
	\arrow["{L_{\tn{lw}}}", from=2-1, to=2-2]
\end{tikzcd}
\end{equation}
\end{rem}

\begin{warn}
It follows from \cite[Proposition 2.27]{mnn17} that for each $i\in I$ the stable $\infty$-category $\Mod_{R_{G_i}}(\spec^{G_i})_{\tn{Borel}}$ has a compact generator $\Sigma^{\infty}_+G_i \otimes R_{G_i}$. In terms of the description $\fun(BG_i,\Mod_R(\spec))$ of \Cref{prop:borelcompleteasfun} this generator is given by the group algebra $R[G_i] = \Sigma^{\infty}_+ G_i \otimes R \in \alg(\spec)$, see also \cite[Remark 6.18]{mnn17}. The homotopy inflation functors in general do not preserve these generators, and the colimit of \Cref{defi:levelwiseborelcomplete} is generally not compactly generated.
\end{warn}

We will now show that there is a canonical t-structure on $\Mod_{R_G}(\specg)_{\tn{lwBorel}}$, and we will record certain properties of this t-structure we will use later on.

\begin{prop}\label{prop:tstructureonlevelwiseborelcomplete}
There is a t-structure on $\Mod_{R_G}(\specg)_{\tn{lwBorel}}$ which is induced by the t-structures of \Cref{cor:tstructureonborelcomplete}. It is right complete, compatible with filtered colimits and compatible with the symmetric monoidal structure.\footnote{In the sense of \cite[Example 2.2.1.3]{lur17}.}\\
The canonical functors $\Mod_{R_{G_i}}(\spec^{G_i})_{\tn{Borel}} \to\Mod_{R_G}(\specg)_{\tn{lwBorel}}$ are t-exact, and the heart is given by the category $\Mod(G;\pi_0(R))$ of discrete $\pi_0(R)$-modules on which $G$ acts continuously.
\end{prop}
\begin{proof}
In our situation \cite[Lemma 3.2.18]{rs20} shows the existence of the t-structure and the t-exactness of the structural maps into the colimit. The proof we give ensures the desired properties of this t-structure.\\
For each $i \in I$, the t-structure on $\fun(BG_i,\Mod_R(\spec))$ is right complete and compatible with filtered colimits since $R$ is assumed to be connective, which implies that the t-structure on $\Mod_R(\spec)$ is \cite[Proposition 7.1.1.13]{lur17}. Hence by \cite[§C.1.2]{lur18} there is a canonical equivalence to the stabilisation 
$$\fun(BG_i,\Mod_R(\spec)) \simeq \spec(\fun(BG_i,\Mod_R(\spec)_{\geq 0}).$$
Since each of the functors $h\tn{infl}^{G_i}_{G_j}$ is t-exact, \cite[Lemma C.2.4.4]{lur18} implies that they restrict to left exact left adjoints on connective parts
$$(h\tn{infl}^{G_i}_{G_j})_{\geq 0}: \fun(BG_j,\Mod_R(\spec)_{\geq 0}) \to \fun(BG_i,\Mod_R(\spec)_{\geq 0}),$$
and by \cite[Lemma C.3.1.1]{lur18} the functor induced by each $(h\tn{infl}^{G_i}_{G_j})_{\geq 0}$ on stabilisations is again given by $h\tn{infl}^{G_i}_{G_j}$. Now, the $\prl$-colimit
$$\cC \defeq \tn{colim}_i \; \fun(BG_i,\Mod_R(\spec)_{\geq 0})$$
along the functors $(h\tn{infl}^{G_i}_{G_j})_{\geq 0}$ is one of Grothendieck prestable $\infty$-categories and left exact left adjoints, so by \cite[Proposition C.3.3.5]{lur18} the colimit is again Grothendieck prestable and the canonical maps into the colimit 
$$(h\tn{infl}^{G}_{G_i})_{\geq 0}: \fun(BG_i,\Mod_R(\spec)_{\geq 0}) \to \cC$$
are left exact left adjoints. Since the functor $\prl \to \prlst$ preserves colimits \cite[Proposition 4.8.2.18]{lur17}, the stabilisation $\spec(\cC)$ is equivalent to $\Mod_{R_G}(\specg)_{\tn{lwBorel}}$. So $\Mod_{R_G}(\specg)_{\tn{lwBorel}}$ is the stabilisation of a Grothendieck prestable $\infty$-category and thus carries a canonical right complete t-structure that is compatible with filtered colimits, see e.g. \cite[Remark C.3.1.5]{lur18}. It follows from \cite[Lemma C.3.1.1, Proposition C.3.2.1]{lur18} that the stabilisations of the structure maps into the colimit
$$h\tn{infl}^{G}_{G_i} \defeq \spec((h\tn{infl}^{G}_{G_i})_{\geq 0}): \fun(BG_i,\Mod_R(\spec) \to \Mod_{R_G}(\specg)_{\tn{lwBorel}}$$
are t-exact. For each $i$, the connective part $\fun(BG_i,\Mod_R(\spec)_{\geq 0})$ is a symmetric monoidal subcategory of its stabilisation, and the transition functors $(h\tn{infl}^{G_i}_{G_j})_{\geq 0}$ are symmetric monoidal. We can hence compute $\cC$ as a filtered colimit in $\calg(\prl)$, and the unit map 
$$\cC \simeq \cC \otimes_{\prl} \spc \to \cC \otimes_{\prl} \spec \simeq \spec(\cC)$$ 
induced by $\Sigma^\infty_+: \spc \to \spec$ is symmetric monoidal, which means that the t-structure on $\spec(\cC)$ is compatible with the symmetric monoidal structure. This monoidal structure agrees with the one of \Cref{defi:levelwiseborelcomplete}. \\
By \cite[§C.1.2]{lur18} the heart of the resulting t-structure is given by the discrete objects of its prestable part, and by \cite[Proposition C.3.2.1]{lur18} all transition functors in the diagram of prestable parts preserve discrete objects. Since the $0$-truncation functor commutes with filtered colimits and using \Cref{cor:tstructureonborelcomplete} the heart is the filtered colimit of the $\Mod(G_i;\pi_0(R))$ along inflation, which is $\Mod(G;\pi_0(R))$. This concludes the proof.
\end{proof}

\begin{cons}\label{cons:colimitoforbits}
For $H$ a closed subgroup of $G$ we consider the adjunction $(\tn{res}^G_H, \tn{coind}^G_H): \Mod_{R_G}(\specg) \rightleftarrows \Mod_{R_H}(\spec^H)$. When $N_i \leq G$ is open and normal, we can view the coinduction $\tn{coind}^G_{N_i}(R_{N_i})$ as a commutative algebra object of $\Mod_{R_G}(\specg)$. Whenever $N_i \subseteq N_j$, there is an equivalence $\tn{coind}^G_{N_i} \simeq \tn{coind}^G_{N_j}\tn{coind}^{N_j}_{N_i}$, and the unit map $R_{N_j} \to \tn{coind}^{N_j}_{N_i}\tn{res}^{N_j}_{N_i}(R_{N_j}) \simeq \tn{coind}^{N_j}_{N_i}(R_{N_i})$ induces an algebra map
$$\tn{coind}^G_{N_j}(R_{N_j}) \to \tn{coind}^G_{N_j}\tn{coind}^{N_j}_{N_i}(R_{N_i}) \simeq \tn{coind}^G_{N_i}(R_{N_i})$$
by applying $\tn{coind}^G_{N_j}$. This determines a filtered diagram $I^\op \to \calg(\Mod_{R_G}(\specg))$, and we let $A_R \in \calg(\Mod_{R_G}(\specg))$ be its colimit. As in the finite group case (\Cref{rec:coinddual}), there is an equivalence
\begin{equation}\label{eq:dualcoind}
\bD(G/N_i) \otimes R_G \stackrel{\sim}{\longrightarrow} \tn{coind}^G_{N_i}\tn{res}^G_{N_i}(R_G) \simeq \tn{coind}^G_{N_i}(R_{N_i})
\end{equation}
in $\calg(\specg)$. Like in the proof of \cite[Lemma 3.8]{bchnl25}, the equivalence \ref{eq:dualcoind} boils down to the fact that the projection formula for the adjunction $(\tn{res}^G_{N_i}, \tn{coind}^G_{N_i})$ holds, see e.g. \cite[Theorem 1.3]{bds16}. In terms of the left-hand side of \ref{eq:dualcoind} the map $\tn{coind}^G_{N_j}(R_{N_j}) \to \tn{coind}^G_{N_i}(R_{N_i})$ is induced by the map of coalgebras in finite $G$-sets $G/N_i \twoheadrightarrow G/N_j$ on function duals, tensored with $R_G$. Since the tensor product on $\specg$ commutes with colimits in both variables it is clear that $A_R \simeq A_{\bS} \otimes R_G$.
\end{cons}

Recall that in the finite group case \Cref{lem:resasbasechange} expressed the restriction functor to an arbitrary subgroup as a base change. This is still true in the profinite case when we work with closed subgroups. 

\begin{lem}\label{lem:resasbasechangeprofinite}
Let $H \leq G$ be a closed subgroup. Then there is a symmetric monoidal equivalence 
$$\Mod_{\tn{coind}^{G}_H(R_{H})}(\specg) \stackrel{\sim}{\longrightarrow} \Mod_{R_{H}}(\spec^{H}),$$
and the equivalence functor fits into the following commutative diagram:
\[\begin{tikzcd}
	{\Mod_{R_G}(\specg)} & \\
	{\Mod_{R_{H}}(\spec^{H})} & {\Mod_{\tn{coind}^{G}_H(R_{H})}(\specg).}
	\arrow["{\tn{res}^G_H}"', from=1-1, to=2-1]
	\arrow["{F_{\tn{coind}^{G}_H(R_{H})}}", from=1-1, to=2-2]
	\arrow["\sim"', from=2-2, to=2-1]
\end{tikzcd}\]
When $H$ additionally is of finite index (i.e., it is an open subgroup), then the algebra object $\tn{coind}^{G}_H(R_{H})$ identifies with $\bD(G/H)\otimes R_G$.
\end{lem}
\begin{proof}
This is shown as \cite[Proposition 3.9]{bchnl25} by combining \cite[Lemma 3.9]{bs20} and \cite[Proposition 5.29]{mnn17}. \\
Note that the coinduction functor $\tn{coind}^G_H:\spec^{H} \to \specg$ still preserves colimits since its left adjoint $\tn{res}^G_H$ preserves compact objects (see also the description given in \Cref{rem:changeofgroupspectralmackey}). It is conservative since all $\tn{coind}^{G_i}_{HN_i/N_i}:\spec^{HN_i/N_i} \to \spec^{G_i}$ are, and the family of categorical fixed point functors $(-)^{H\cap N_i}: \spec^{H} \to \spec^{HN_i/N_i}$ indexed by $i \in I$ is jointly conservative. The description of the algebra object in the finite index case is the same as in \Cref{cons:colimitoforbits}.
\end{proof}

\begin{rem}
The previous lemma is a striking difference to the case where $G$ is a compact Lie group and $H$ is an arbitrary closed subgroup, here $\tn{coind}^G_H$ in general fails to be conservative \cite[Warning p.1034]{mnn17}. In the compact Lie case the restriction $\tn{res}^G_H$ always has a left adjoint $\tn{ind}^G_H$ that is related to $\tn{coind}^G_H$ via a Wirthmüller isomorphism. On the contrary, in the profinite case $\tn{res}^G_H$ for $H$ again an arbitrary closed subgroup might not have a left adjoint, although its right adjoint $\tn{coind}^G_H$ always preserves colimits, cf. \Cref{rem:changeofgroupspectralmackey}. This goes in line with the fact that tensoring with $\tn{coind}^{G}_H(R_{H})$, or equivalently applying $\tn{coind}^G_H\tn{res}^G_H$, might not preserve limits - we will see below that there is generally no reason to expect that $\tn{coind}^{G}_H(R_{H})$ is compact (and hence neither dualisable).
\end{rem}

When we work with an arbitrary closed subgroup $H \leq G$, the only caveat of the previous lemma is that $\tn{coind}^G_H(R_H)$ is not a very explicit description of the algebra object we are working with. For $H=1$ we will give a description of the functor $\tn{coind}^G_1\tn{res}^G_1$ as tensoring with the algebra object $A_R$ from \Cref{cons:colimitoforbits}. We first prove a general statement about filtered colimits of monadic adjunctions.\\
Recall from \cite[Construction 5.23]{mnn17} (or from \Cref{rec:modules}) that if we have an adjunction $(L,R): \cM \rightleftarrows \cN$ of presentably symmetric monoidal $\infty$-categories where $L$ is symmetric monoidal, then the right adjoint applied to the unit $R(\bbone_{\cN})$ has a commutative algebra structure and there is an induced adjunction
$$(\bar{L},\bar{R}): \Mod_{R(\bbone_{\cN})}(\cM) \rightleftarrows \cN$$
for which the left adjoint $\bar{L}$ factors as $\Mod_{R(\bbone_{\cN})}(\cM) \stackrel{L}{\longrightarrow} \Mod_{LR(\bbone_{\cN})}(\cN) \stackrel{\varepsilon_!}{\longrightarrow} \cN$,
where $\varepsilon: LR(\bbone_{\cN}) \to \bbone_{\cN}$ is the counit map. We will say that the adjunction $(L,R)$ is \emph{monadic} if the induced adjunction $(\bar{L},\bar{R})$ is an equivalence.

\begin{lem}\label{lem:filteredcolimitofmonadicadjunctions}
Let $\cC$ be a filtered $\infty$-category that has an initial object $X_0$ and let $\cE: \cC \to \calg(\prl)$ be a functor. Write $f_{X} = \cE(X_0 \to X)$ and $g_X$ for its right adjoint. Assume that for every $X \in \cC$ the adjunction 
$$(f_{X},g_{X}): \cE(X_0) \rightleftarrows \cE(X)$$
is monadic. Set $\cE_{\infty} = \tn{colim}_{\cC} \; \cE$, where the colimit is taken in $\calg(\prl)$. Then the induced adjunction
$$(f_{{\infty}},g_{\infty}): \cE(X_0) \rightleftarrows \cE_{\infty}$$
is monadic. In particular, there is an equivalence $\cE_{\infty} \simeq \Mod_{g_{\infty}(\bbone_{\cE_{\infty}})}(\cE(X_0))$, and the commutative algebra object $g_{\infty}(\bbone_{\cE_{\infty}})$ is equivalent to the filtered colimit of the $g_{X}(\bbone_{\cE(X)})$ in $\calg(\cE(X_0))$.
\end{lem}
\begin{proof}
This is an application of \cite[Lemma 3.5.6]{agv22} which computes filtered colimits of module categories with extensions of scalars between them. \\
We write $A_X$ for $g_{X}(\bbone_{\cE(X)}) \in \calg(\cE(X_0))$. The assignment $X \mapsto A_X$ determines a filtered diagram $\epsilon: \cC \to \calg(\cE(X_0))$, and we write $A_{\infty}$ for its colimit. Then by \cite[Theorem 4.8.5.11]{lur17} the functor $\cE$ identifies with the composition
$$\cC \stackrel{\epsilon}{\longrightarrow} \calg(\cE(X_0)) \stackrel{\what{\Theta}_*}{\longrightarrow} \calg(\prl)_{\cE(X_0)/} \stackrel{\tn{fgt}}{\longrightarrow} \calg(\prl), \; X \mapsto \Mod_{A_X}(\cE(X_0)),$$
where $\what{\Theta}_*$ is as in \cite[Notation 4.8.5.10]{lur17}. Applying \cite[Lemma 3.5.6]{agv22} to the constant functor $X \mapsto \cE(X_0)$ and the diagram of algebra objects $\epsilon$ gives an equivalence
$$\cE_{\infty} \simeq \Mod_{A_{\infty}}(\cE(X_0))$$
under which the left adjoint $f_{\infty}: \cE(X_0) \to \cE_{\infty}$ identifies with the free module functor $A_{\infty} \otimes (-)$. In particular, by applying its right adjoint $g_{\infty}$ to the unit of $\cE_{\infty}$ we see that $g_{\infty}(\bbone_{\cE_{\infty}}) \simeq A_{\infty}$ in $\calg(\cE_{\infty})$. This concludes the proof.
\end{proof}

\begin{rem}\label{rem:monads}
Since in the previous lemma the monads $g_Xf_X$ and $g_{\infty}f_{\infty}$ in $\tn{End}(\cE(X_0))$ are determined by the algebra objects $g_{X}(\bbone_{\cE(X)})$ and $g_{\infty}(\bbone_{\cE_{\infty}})$ they tensor with, the lemma implies that the monad $g_{\infty}f_{\infty}$ is the colimit of the monads $g_Xf_X$ taken in endofunctors of $\cE(X_0)$.
\end{rem}

\begin{lem}\label{lem:spectraascolimit}
There is a symmetric monoidal equivalence 
$$\tn{colim}_i \; \Mod_{R_{N_i}}(\spec^{N_i}) \simeq \Mod_{R}(\spec),$$
where the colimit is taken in $\calg(\prlst)$ along the restrictions $\tn{res}^{N_j}_{N_i}$ for $N_i \subseteq N_j$.
\end{lem}
\begin{proof}
Consider the filtered diagram 
\begin{align*}
&\Mod_{R_{N_{\bullet}}}(\spec^{N_\bullet}): I^\op \to \calg(\prlst), \; i \mapsto \Mod_{R_{N_i}}(\spec^{N_i}), \\
&(N_i \subseteq N_j) \mapsto (\tn{res}_{N_i}^{N_j}: \Mod_{R_{N_j}}(\spec^{N_j}) \to \Mod_{R_{N_i}}(\spec^{N_i})).
\end{align*}
Restricting from each $N_i$ to the trivial subgroup induces a symmetric monoidal left adjoint $\varphi: \tn{colim}_i \; \Mod_{R_{N_i}}(\spec^{N_i}) \to \Mod_{R}(\spec)$. For $X \in \Mod_{R}(\spec)$, its inflation $\tn{infl}^{N_i}_1(X)$ to $\Mod_{R_{N_i}}(\spec^{N_i})$ (for any $i$) has restriction $\tn{res}^{N_i}_1\tn{infl}^{N_i}_1(X) \simeq X$, hence $\varphi$ is essentially surjective. Since the restrictions in the colimit diagram preserve compact objects their colimit is compactly generated, and by \cite[Proposition 5.3.5.11]{lur09} it suffices to show that $\varphi$ is fully faithful when restricted to compacts. \\
The compact part of $\tn{colim}_i \; \Mod_{R_{N_i}}(\spec^{N_i})$ is $\cC \defeq \tn{colim}_i \; \Mod_{R_{N_i}}^{\omega}(\spec^{N_i})$, where the filtered colimit is computed in $\catinf$. For two objects $\bar{X}, \bar{Y} \in \cC$ we can assume without restriction that they are of the form $\varphi_k(X)$ and $\varphi_k(Y)$ respectively, where $X$ and $Y$ are compact objects of $\Mod_{R_{N_k}}(\spec^{N_k})$ for the same index $k$ and $\varphi_k$ is the structure map $\Mod_{R_{N_k}}^{\omega}(\spec^{N_k}) \to \cC$. We can furthermore reduce to the case where the objects $X$ and $Y$ are compact generators, i.e. of the form $\Sigma^{\infty}_+ N_k/H \otimes R_{N_k}$ and $\Sigma^{\infty}_+ N_k/H' \otimes R_{N_k}$ for two open subgroups $H,H' \leq N_k$. We hence consider the map
\begin{equation}\label{eq:varphimappingspaces}
\Map_{\cC}(\bar{X},\bar{Y}) \stackrel{\varphi}{\longrightarrow} \Map_{\Mod_{R}^{\omega}(\spec)}(\varphi(\bar{X}),\varphi(\bar{Y}))
\end{equation}
induced by $\varphi$ on mapping spaces. Note that the functor $(I^\op)_{k/} \to I^\op$ is cofinal \cite[Example 5.4.5.9, Lemma 5.4.5.12]{lur09}, hence by \cite[Proposition 4.1.1.8]{lur09} the left-hand side of \ref{eq:varphimappingspaces} identifies with 
$$\tn{colim}_{i \geq k} \; \Map_{\Mod_{R_{N_i}}^{\omega}(\spec^{N_i})}(\Sigma^{\infty}_+ \tn{res}^{N_k}_{N_i}(N_j/H) \otimes R_{N_i},\Sigma^{\infty}_+ \tn{res}^{N_k}_{N_i}(N_j/H') \otimes R_{N_i}),$$
where we understand $i \geq k$ as $N_i \subseteq N_k$. Now there is a $k' \geq k$ such that for all $i \geq k'$ the normal subgroup $N_i$ is contained in both $H$ and $H'$ and hence acts trivially on $N_k/H$ and $N_k/H'$. Using the same cofinality argument as before, we can thus further reduce to showing that the map
\begin{equation}\label{eq:varphimappingspacesreduced}
\tn{colim}_{i} \; \Map_{\Mod_{R_{N_i}}^{\omega}(\spec^{N_i})}(R_{N_i},R_{N_i}) \stackrel{\varphi}{\longrightarrow} \Map_{\Mod_{R}^{\omega}(\spec)}(R,R)
\end{equation}
is an equivalence. Since each $\Mod_{R_{N_i}}^{\omega}(\spec^{N_i})$ is itself a colimit along inflation, the left-hand side of \ref{eq:varphimappingspacesreduced}, where we write $N_{ij} \defeq N_iN_j/N_j$, identifies with
\begin{align}
&\colimit{i, \, \tn{res}} \; \colimit{j, \, \tn{infl}} \; \Map_{\Mod_{R_{N_{ij}}}^{\omega}(\spec^{N_{ij}})}(R_{N_{ij}},R_{N_{ij}}) \notag \\
\simeq \; &\colimit{j, \, \tn{infl}} \; \colimit{i, \, \tn{res}} \; \Map_{\Mod_{R_{N_{ij}}}^{\omega}(\spec^{N_{ij}})}(R_{N_{ij}},R_{N_{ij}}). \label{eq:finalmappingspace}
\end{align}
But for each $j$, there is an $i'$ such that $N_i \subseteq N_j$ for all $i \geq i'$, and in this case $N_{ij}=N_iN_j/N_j$ is trivial. Using cofinality once more, this shows that \ref{eq:finalmappingspace} is given by $\Map_{\Mod_{R}^{\omega}(\spec)}(R,R)$, and the map \ref{eq:varphimappingspacesreduced} is an equivalence, as desired.
\end{proof}

\begin{cor}\label{cor:identifycolimitoforbits}
There is an equivalence $\tn{coind}^G_1(R) \simeq A_R$ in $\calg(\Mod_{R_G}(\specg))$.
\end{cor}
\begin{proof}
Apply \Cref{lem:filteredcolimitofmonadicadjunctions} to the diagram of \Cref{lem:spectraascolimit}, using \Cref{lem:resasbasechangeprofinite}.
\end{proof}

\begin{cor}\label{cor:endofunctors}
There are equivalences of endofunctors of $\Mod_{R_G}(\specg)$
$$(A_R \otimes_{R_G} -) \simeq \tn{coind}^G_1\tn{res}^G_1 \simeq \tn{colim}_i \, \tn{coind}^G_{N_i}\tn{res}^G_{N_i}.$$
\end{cor}
\begin{proof}
This follows from \Cref{lem:resasbasechangeprofinite}, \Cref{cor:identifycolimitoforbits} and \Cref{rem:monads}.
\end{proof}

\begin{ex}
Using adjunction and \Cref{cor:identifycolimitoforbits} we obtain an equivalence
$$\Map_{\spec}(\bS,\bS) \simeq \tn{colim}_i \; \Map_{\specg}(\Sigma^{\infty}_+ G/N_i,\bS_G),$$
which on $\pi_0$ recovers the fact that the filtered colimit of the group completions of the monoids of finite $N_i$-sets under disjoint union $\pi_0^{N_i}(\bS_G) \cong \bA_{\bZ}(G/N_i)$ is $\bZ$.
\end{ex}

Let us fix some terminology for t-structures we will use throughout this paper.

\begin{defi}\label{def:hypercompletetstructure}
Let $\cC$ be a stable $\infty$-category with a t-structure $(\cC_{\geq 0},\cC_{\leq 0})$ and denote by $\pi_n(-): \cC \to \cC^{\heartsuit}$ its integer indexed homotopy object functors. We say that an object $X \in \cC$ is
\begin{enumerate}
    \item \emph{acyclic} if $\pi_n(X) = 0$ for all $n \in \bZ$,
    \item \emph{hypercomplete} if $\Map_{\cC}(Y,X) \simeq \ast$ for all acyclic $Y \in \cC$.
\end{enumerate}
We call $\cC$ \emph{hypercomplete} if all of its objects are hypercomplete. 
\end{defi}

\begin{rem}
That $\cC$ is hypercomplete is equivalent to all acyclic objects being zero. If the t-structure on $\cC$ is right complete then hypercompleteness is equivalent to the t-structure being left separated, i.e. $\bigcap_{n \in \bZ} \cC_{\geq n} = 0$. Again assuming right completeness, hypercompleteness is implied by Postnikov completeness, which asserts that for every object $X$ the canonical map $X \to \lim_{n\geq 0} \tau_{\leq n} X$ is an equivalence. The latter is precisely the statement that the canonical map $\cC \to \tn{lim}_{n \geq 0} \; \cC_{\leq n}$ is fully faithful. This map being an equivalence is the condition that the t-structure on $\cC$ is left complete.
\end{rem}

\begin{rem}
It is clear from the definition that the full subcategory of hypercomplete objects $\cC^h \subseteq \cC$ is closed under limits. If $\cC$ is additionally assumed to be presentable and its t-structure is compatible with filtered colimits then the inclusion $\cC^h \hookrightarrow C$ has an accessible left adjoint we denote by $L^h: \cC \to \cC^h$ \cite[Proposition 2.14]{cm21}, \cite[Proposition C.1.4.1]{lur18}. Note that $\cC^h$ inherits a t-structure that makes the functor $L^h$ t-exact. If $\cC$ is furthermore presentably symmetric monoidal and this symmetric monoidal structure is compatible with the t-structure (i.e., the connective part is a symmetric monoidal subcategory), then \cite[Lemma 2.22]{cm21} gives a unique symmetric monoidal structure on $\cC^h$ that makes the hypercompletion functor symmetric monoidal.
\end{rem}

Recall that for a finite group Borel complete spectra were defined as the right orthogonal of non-equivariantly contractible spectra. We now give characterisations for the latter notion in the profinite group case. One of them uses the t-structure from \Cref{prop:tstructureonlevelwiseborelcomplete}, which in contrast to the t-structure on $G$-spectra is only right but generally not left complete. In particular, there can be non-zero acyclic objects.

\begin{prop}\label{prop:noneqcontractibleprofinite}
For an $R_G$-module $X$, the following are equivalent:
\begin{enumerate}
    \item $X$ is non-equivariantly contractible, i.e. $\tn{res}^G_1(X) \simeq 0$ in $\Mod_R(\spec)$.
    \item $\tn{res}^G_1U_{R_G}(X) \simeq U_R\tn{res}^G_1(X)$ is contractible in $\spec$.
    \item $X \otimes_{R_G} A_R \simeq 0$ in $\Mod_{R_G}(\specg)$.
    \item $U_R(X \otimes_{R_G} A_R) \simeq U_R(X) \otimes A_{\bS}$ is contractible in $\spec$.
    \item The levelwise completion $L_{\tn{lw}}(X)$ is acyclic, that is $\pi_n(L_{\tn{lw}}(X))=0$ for all $n \in \bZ$, where the $\pi_n(-)$ are the t-structure homotopy modules of \ref{prop:tstructureonlevelwiseborelcomplete}.
\end{enumerate}
\end{prop}
\begin{proof}
That (1) and (2) are equivalent is clear from the conservativity of the forgetful functor and \Cref{rec:modules}, similarly that (3) and (4) are equivalent follows from the conservativity of the forgetful functor and the projection formula for the free-forgetful adjunction. The equivalence of assertions (1) and (3) follows from the conservativity of $\tn{coind}^G_1$ (see \Cref{lem:resasbasechangeprofinite}) and \Cref{cor:endofunctors}. \\
We now show that (3) and (5) are equivalent. Let $X$ be an $R_G$-module. By \cite[Lemma 6.3.3.6]{lur09} we have $X \simeq \tn{colim}_i \, \tn{infl}^G_{G_i}(X^{N_i})$, so 
$$L_{\tn{lw}}(X) \simeq \tn{colim}_i \, L_{\tn{lw}}\tn{infl}^G_{G_i}(X^{N_i}) \stackrel{\ref{diag:levelwisecompletion}}{\simeq} \tn{colim}_i \, h\tn{infl}^G_{G_i}\theta_i^*(X^{N_i}) \simeq \tn{colim}_i \, h\tn{infl}^G_{G_i}\sigma_i^*(X),$$
where $\theta_i: BG_i \to \Spans(G_i)^{\op}$ picks out the orbit $G_i/1$ and $\sigma_i: BG_i \to \Spans(G)^{\op}$ picks out the orbit $G/N_i$, using that each of the triangles
\[\begin{tikzcd}
	& {\Spans(G)^{\op}} \\
	{BG_i} & {\Spans(G_i)^{\op}}
	\arrow["{\sigma_i}", from=2-1, to=1-2]
	\arrow["{\theta_i}", from=2-1, to=2-2]
	\arrow["{\overl{\tn{infl}}^G_{G_i}}"', from=2-2, to=1-2]
\end{tikzcd}\]
commutes. (Recall from \Cref{rem:changeofgroupspectralmackey} that categorical fixed points precompose with the vertical functor $\overl{\tn{infl}}^G_{G_i}$ of span categories.) Since the t-structure from \Cref{prop:tstructureonlevelwiseborelcomplete} is compatible with filtered colimits it follows for its homotopy modules
\begin{align*}
\pi_n(L_{\tn{lw}}(X)) &\cong \tn{colim}_i \, \pi_n(h\tn{infl}^G_{G_i}\sigma_i^*(X)) \cong \tn{colim}_i \, (h\tn{infl}^G_{G_i})^{\heartsuit} (\pi_n(\sigma_i^*(X))) \\
&\cong \tn{colim}_i \, (h\tn{infl}^G_{G_i})^{\heartsuit} \pi_n^{N_i}(X) \cong \tn{colim}_i \, \pi_n^{N_i}(X).
\end{align*}
Here $(h\tn{infl}^G_{G_i})^{\heartsuit}$ denotes the restriction of $h\tn{infl}^G_{G_i}$ to hearts, it is given by the ordinary inflation from $(G_i;\pi_0(R))$-modules to discrete $(G;\pi_0(R))$-modules. We use that $h\tn{infl}^G_{G_i}$ is t-exact and hence commutes with homotopy objects. In the last step each $\pi_n^{N_i}(X)$ is viewed as a discrete $(G;\pi_0(R))$-module.\\
On the other hand $X \otimes_{R_G} A_R \simeq \tn{colim}_i \, \imap_{\Mod_R(\specg)}(\Sigma^{\infty}_+G/N_i \otimes R, X)$, and this $R_G$-module is contractible if and only if its underlying $G$-spectrum is, for which contractibility is equivalent to the vanishing of all homotopy Mackey functors. For the $n^{\tn{th}}$ one, using that the forgetful functor $\Mod_R(\specg) \to \specg$ commutes with colimits, the $N_j$-sections are given by 
\begin{align}
\pi^{N_j}_n(&\tn{colim}_i \, \imap_{\specg}(\Sigma^{\infty}_+G/N_i, X)) \notag \\
\cong \;&\tn{colim}_i \, \Hom_{\specg}(\Sigma^{\infty +n}_+G/N_j, \imap_{\specg}(\Sigma^{\infty}_+G/N_i, X)) \notag \\
\cong \;&\tn{colim}_i \, \Hom_{\specg}(\Sigma^{\infty +n}_+G/N_j \times G/N_i, X) \notag \\
\cong \;&\tn{colim}_i \bigoplus_{[N_j\backslash G /N_i]} \pi_n^{N_i\cap N_j}(X).
\label{eq:mackeyintersection}
\end{align}
Assume the vanishing of $\pi_n(L_{\tn{lw}}(X))$, so for all $x\in \pi_n^{N_i}(X)$ there is an $N_l \subseteq N_i$ such that $x$ vanishes in $\pi_n^{N_l}(X)$. Fix $j$. If $y$ is an element of  $\bigoplus_{k \in [N_j\backslash G /N_i]} \; \pi_n^{N_i\cap N_j}(X)$ with (finitely many) components $y_k \in \pi_n^{N_i \cap N_j}(X)$, then there are $N_{k} \subseteq N_i\cap N_j$ such that $y_k$ vanishes in $\pi_n^{N_{k}}(X)$, and hence $y$ vanishes when passing to the sections for the group $\bigcap_k N_{k} = (\bigcap_k N_{k}) \cap N_j \subseteq N_i \cap N_j$. \\
Conversely, assume that \ref{eq:mackeyintersection} vanishes for all $N_j$, then in particular it vanishes for $N_j=G$. Then \ref{eq:mackeyintersection} reads $\tn{colim}_i \, \pi_n^{N_i}(X)$, which is precisely $\pi_n(L_{\tn{lw}}(X))$. Using this argument for all $n$ shows that (3) and (5) are equivalent, concluding the proof.
\end{proof}

\begin{rem}
If $X \in \Mod_{R_G}(\specg)$ is compact, then it is inflated from some $G_i$, i.e. $X \simeq \tn{infl}^G_{G_i}(X')$. Then $X \otimes_{R_G} A_R \simeq 0$ is equivalent to $\tn{res}^{G_i}_1(X') \simeq 0$. However, if we write a general $X \in \Mod_{R_G}(\specg)$ as a filtered colimit of compacts $X_j$, then $X \otimes_{R_G} A_R \simeq 0$ need not imply $X_j \otimes A_R \simeq 0$ for any $j$.
\end{rem}

We can now extend \Cref{defi:borelcomplete} to the profinite case.

\begin{defi}
We define the $\infty$-category of \emph{Borel complete $R_G$-modules} 
$$\Mod_{R_G}(\specg)_{\tn{Borel}}$$
as the full subcategory of $\Mod_{R_G}(\specg)$ on the $A_R$-complete $R_G$-modules. That is, an $R_G$-module $X$ is Borel complete if for all $Y \in \Mod_{R_G}(\specg)$ with $Y \otimes_{R_G} A_R \simeq 0$ the mapping space $\Map_{\Mod_{R_G}(\specg)}(Y,X)$ is contractible.
\end{defi}

\begin{rem}
As in the finite group case, the general theory of \cite[§2.2]{mnn17} ensures that $\Mod_{R_G}(\specg)_{\tn{Borel}}$ is presentable and that its inclusion into $\Mod_{R_G}(\specg)$ has a left adjoint $L_{G;R}$ we call \emph{Borel completion}. There is a unique symmetric monoidal structure on $\Mod_{R_G}(\specg)_{\tn{Borel}}$ that makes $L_{G;R}$ symmetric monoidal. \\
However, a large proportion of the theory does not carry over, since there is no reason for the algebra object $A_R$ that we defined as an infinite colimit to be dualisable.\footnote{Or more generally, that tensoring with $A_R$ preserves limits.} In particular, $\Mod_{R_G}(\specg)_{\tn{Borel}}$ is not necessarily the right orthogonal to $A_R^{-1}$-local objects as defined in \cite[Definition 3.10]{mnn17}, and we do not have nice descriptions of the completion functor or the $\infty$-category $\Mod_{R_G}(\specg)_{\tn{Borel}}$ as a functor category.
\end{rem}

Nevertheless, we can show a generalisation of \Cref{cor:completionunderlying} via an argument that always holds for complete objects in module categories in the rigidly-compactly generated setting.

\begin{lem}\label{lem:completionunderlyingprofinite}
An $R_G$-module $X$ is Borel complete if and only if its underlying $G$-spectrum $U_{R_G}(X) \in \specg$ is Borel complete.
\end{lem}
\begin{proof}
For ease of notation, we write $U=U_{R_G}$ and $F=F_{R_G}$. One implication is as in \Cref{rem:completeunderlyingeasydirection}: Assume that the $R_G$-module $X$ is $A_R$-complete and let $Y \in \specg$ satisfy $Y \otimes A_{\bS} \simeq 0$. Then $F(Y) \otimes_R A_R \simeq F(Y) \otimes_R F(A_{\bS}) \simeq F(Y \otimes A_{\bS}) \simeq 0$ and
$$\Map_{\specg}(Y,U(X)) \simeq \Map_{\Mod_{R_G}(\specg)}(F(Y),X) \simeq 0,$$
hence $U(X)$ is $A_{\bS}$-complete. For the other direction, note that by \cite[§2.2]{mnn17} a map $f:U(X) \to M$ of $G$-spectra exhibits $M$ as the $A_{\bS}$-completion of $U(X)$ if and only if the following two conditions are satisfied:
\begin{enumerate}
    \item The map $f$ becomes an equivalence after tensoring with $A_{\bS}$.
    \item The $G$-spectrum $M$ is $A_{\bS}$-complete.
\end{enumerate}
Let $c: X \to \iota L_{G;R_G}(X)$ be the $A_R$-completion map for $X$. From the first part it follows that $U(\iota L_{G;R_G}(X))$ is $A_{\bS}$-complete. Using the projection formula for the free-forgetful adjunction the map $U(c) \otimes A_{\bS}$ identifies with
$$U(c \otimes_R A_R): U(X \otimes_R A_R) \to U(\iota L_{G;R_G}(X) \otimes_R A_R),$$
which is an equivalence since $c \otimes_R A_R$ is one. Hence $U(c): U(X) \to U\iota L_{G;R_G}(X)$ exhibits $U\iota L_{G;R_G}(X)$ as the $A_{\bS}$-completion of $U(X)$. Now if $U(X)$ is $A_{\bS}$-complete then this map is an equivalence, and since $U$ is conservative $c$ is an equivalence, which means that $X$ is $A_R$-complete, as desired.
\end{proof}

We finally connect the $\infty$-category of Borel completes we just defined to the previously defined $\infty$-category of levelwise Borel complete modules.

\begin{thm}\label{thm:borelandhyperlewelwise}
There is an inclusion $\Mod_{R_G}(\specg)_{\tn{Borel}} \hookrightarrow \Mod_{R_G}(\specg)_{\tn{lwBorel}}$ which induces a symmetric monoidal equivalence
$$\Mod_{R_G}(\specg)_{\tn{Borel}} \simeq \Mod_{R_G}(\specg)^h_{\tn{lwBorel}}.$$
\end{thm}
\begin{proof}
We show that both categories agree as full subcategories of $\Mod_{R_G}(\specg)$. Let $X \in \Mod_{R_G}(\specg)_{\tn{Borel}}$, we first show that it is levelwise Borel complete. This is as in \Cref{lem:categoricalfixedpointsborelcomplete} - $X$ defines a compatible family $(X^{N_i})_{i \in I}$ for which we show that every component is Borel complete. So for $i \in I$ let $Y \in \Mod_{R_{G_i}}(\spec^{G_i})$ such that $\tn{res}^{G_i}_1(Y) \simeq 0$. Then $\tn{res}^G_1\tn{infl}^G_{G_i}(Y) \simeq \tn{res}^{G_i}_1(Y) \simeq 0$ and hence $\tn{infl}^G_{G_i}(Y) \otimes_{R_G} A_R \simeq 0$ by \Cref{prop:noneqcontractibleprofinite}. Then 
$$\Map_{\Mod_{R_{G_i}}(\spec^{G_i})}(Y,X^{N_i}) \simeq \Map_{\Mod_{R_{G}}(\specg)}(\tn{infl}^G_{G_i}(Y),X) \simeq 0,$$
so $X^{N_i} \in \Mod_{R_{G_i}}(\spec^{G_i})$ is Borel complete and $X$ is levelwise Borel. We hence have an inclusion 
$$\Mod_{R_G}(\specg)_{\tn{Borel}} \subseteq \Mod_{R_G}(\specg)_{\tn{lwBorel}}.$$
Now, by definition $X \in \Mod_{R_G}(\specg)_{\tn{lwBorel}}$ is hypercomplete if and only if the mapping space
$\Map_{\Mod_{R_G}(\specg)_{\tn{lwBorel}}}(Z,X)$
is contractible for all acyclic $Z$, by \Cref{prop:noneqcontractibleprofinite} equivalently for all $Z$ of the form $L_{\tn{lw}}(Y)$ with $Y \otimes_{R_G} A_R \simeq 0$ (using that the levelwise completion is essentially surjective). By adjunction this holds for $X$ if and only if
$\Map_{\Mod_{R_G}(\specg)}(Y,X)$
is contractible for all $R_G$-modules $Y$ for which $Y \otimes_{R_G} A_R \simeq 0$ holds, which precisely means that $X$ is Borel complete. This shows that $X$ is in $\Mod_{R_G}(\specg)_{\tn{Borel}}$ if and only if it is hypercomplete, and hence the two $\infty$-categories in the statement agree as full subcategories of $R_G$-modules. Since they are both symmetric monoidal localisations of $\Mod_{R_G}(\specg)$ it follows from the universal property of symmetric monoidal localisations \cite[Proposition 3.2.2]{hin16} that the induced equivalence functor is symmetric monoidal.
\end{proof}

We will see that under assumptions on $G$ and $R$ the $\infty$-category of levelwise Borel complete $R_G$-modules is already hypercomplete, and hence in this case the two notions of Borel completeness we considered agree. We will discuss this in the next section.

\section{Sheaves and representations}\label{sec:sheavesandreps}

In \cite[§4.1]{cm21}, Clausen and Mathew consider the classifying topos for a profinite group and prove various results about hypercompleteness for sheaves of spectra on the latter. The purpose of this section is to connect the previous results on Borel completeness for profinite groups to the results of Clausen and Mathew. We first recall some generalities.

\begin{rec}
For a small $\infty$-category $\cC$, we denote by $\cP(C) = \fun(\cC^{\op},\spc)$ its \emph{presheaf $\infty$-category} and by $y: \cC \hookrightarrow \cP(C)$ the Yoneda embedding. If $(\cC, \tau)$ is an $\infty$-site and $W$ is the class of $\tau$-covering sieves $T \to y(X)$ where $X \in \cC$, then a \emph{$\tau$-sheaf on $\cC$} is a $W$-local object in $\cP(\cC)$. The full subcategory $\sheaves_{\tau}(\cC)$ of $\cP(\cC)$ spanned by such sheaves is the \emph{$\infty$-category of $\tau$-sheaves on $\cC$}. We denote the left adjoint sheafification functor by $\tsf{a}_{\tau}: \cP(\cC) \to \sheaves_{\tau}(\cC)$. \\
If $W$ instead is the set of maps $\tn{colim}_{[n] \in \Delta} \; y(\cU_n) \to y(X)$ for $\cU_{\bullet}$ a hypercover of an object $X$, then an $W$-local object of $\cP(\cC)$ is called a \emph{$\tau$-hypersheaf on $\cC$}. The full subcategory of $\cP(\cC)$ they span is denoted $\sheaves^{\wedge}_{\tau}(\cC)$. Each hypersheaf is in particular a sheaf, and we denote the localisation functor by $(-)^{\wedge}: \sheaves_{\tau}(\cC) \to \sheaves^{\wedge}_{\tau}(\cC)$, the composition with $\tsf{a}_{\tau}$ by $\tsf{a}_{\tau}^{\wedge}: \cP(\cC) \to \sheaves_{\tau}^{\wedge}(\cC)$.
\end{rec}

Replacing $\spc$ by an $\infty$-category with small limits, we get the following.

\begin{rec}
Let $(\cC, \tau)$ be an $\infty$-site and let $\cD$ an $\infty$-category with small limits. A \emph{$\tau$-(hyper)sheaf} on $\cC$ with values in $\cD$ is a functor $\cF: \cC^{\op} \to \cD$ such that its canonical limit-preserving extension $\bar{\cF}: \cP(\cC)^{\op} \to \cD$ factors through the opposite of the localisation functor $\tsf{a}^{(\wedge)}_{\tau}: \cP(\cC) \to \sheaves^{(\wedge)}_{\tau}(\cC)$. The \emph{$\infty$-category of $\tau$-(hyper)sheaves on $\cC$ with values in $\cD$} is the full subcategory $\sheaves_{\tau}^{(\wedge)}(\cC; \cD)$ of $\cP(\cC;\cD)=\fun(\cC^{\op},\cD)$ on $\tau$-(hyper)sheaves $\cC^{\op} \to \cD$. \\
If $\cD$ is furthermore presentable, then $\sheaves_{\tau}^{\wedge}(\cC; \cD)$ is presentable and there is a left adjoint to the inclusion $\sheaves_{\tau}^{(\wedge)}(\cC; \cD) \subseteq \cP(\cC;\cD)$ denoted $\tsf{a}_{\tau}^{(\wedge)}: \cP(\cC;\cD) \to \sheaves_{\tau}^{(\wedge)}(\cC;\cD)$. If the topology is clear from the context we omit the $\tau$ from the notation.
\end{rec}

\begin{rem}\label{rem:tstructureonsheaves}
When $\cD = \Mod_R(\spec)$ for a connective $\bE_{\infty}$-ring $R$ we write $\cP(\cC;R)$ and $\sheaves_{\tau}^{(\wedge)}(\cC; R)$ and for the respective presheaf and (hyper)sheaf $\infty$-categories with values in $R$-modules. By \cite[Proposition 1.3.2.7, 2.1.1.1]{lur18} there is a t-structure on $\sheaves_{\tau}(\cC; R)$ for which a sheaf $\cF$ is connective if the sheafification of the presheaf of abelian groups $X \mapsto \pi_n\cF(X)$ vanishes for all $n<0$ and it is coconnective if all its sections are after applying the forgetful functor $\Mod_R(\spec) \to \spec$. This t-structure is right complete and compatible with filtered colimits.
\end{rem}

\begin{prop}\label{prop:hypersheafhypercomplete}
A sheaf in $\sheaves_{\tau}(\cC; R)$ is a hypersheaf if and only if it is hypercomplete (in the sense of \Cref{def:hypercompletetstructure}).
\end{prop}
\begin{proof}
For sheaves of spectra this follows from \cite[Lemma 6.5.3.11]{lur09}, \cite{dhi04,tv03}, \cite[Corollary 6.5.3.13]{lur09} and \cite[Proposition 1.3.3.3]{lur18}, as discussed in \cite[Example 2.5]{cm21}. We show that the case with coefficients reduces to the one of spectral sheaves.\\
Since the forgetful functor $U_R: \Mod_R(\spec) \to \spec$ preserves and detects limits, \cite[Remark 2.3.3]{agv22} implies that the induced functor $U_R: \cP(\cC;R) \to \cP(\cC;\spec)$ preserves and detects sheaves and hypersheaves. Hence a sheaf in $\sheaves_{\tau}(\cC; R)$ is a hypersheaf if and only if it is one after applying $U_R$. \\
It remains to show the analogous statement for hypercomplete sheaves. It follows from the description of the t-structure on $\sheaves_{\tau}(\cC; R)$ given in \cite[§2.1.1]{lur18} that the functor $U_R: \sheaves_{\tau}(\cC; R) \to \sheaves_{\tau}(\cC; \spec)$ is t-exact. Under the identification $\sheaves_{\tau}(\cC; R) \simeq \Mod_{\underl{R}}(\sheaves_{\tau}(\cC; \spec))$, where $\underl{R}$ is the sheafification of the constant presheaf with value $R$, the functor $U_R$ identifies with the forgetful functor 
$$\Mod_{\underl{R}}(\sheaves_{\tau}(\cC; \spec)) \to \sheaves_{\tau}(\cC; \spec).$$
It has a (automatically right t-exact) left adjoint we denote $F_R$, cf. \cite[Proposition 2.1.0.3]{lur18}. Now the argument is analogous to the one given in \Cref{lem:completionunderlyingprofinite} for complete objects in a  module category. \\
Let $\cF \in \sheaves_{\tau}(\cC; R)$ be hypercomplete. Then $U_R(\cF)$ is hypercomplete. Indeed, if $\cG \in \sheaves_{\tau}(\cC; \spec)$ is acyclic then $F_R(\cG)$ is acyclic since $F_R$ is right t-exact and the t-structures at play are right complete. We have
$$\Map_{\sheaves_{\tau}(\cC; \spec)}(\cG,U_R(\cF)) \simeq \Map_{\sheaves_{\tau}(\cC;R)}(F_R(\cF),\cG) \simeq 0,$$
and hence $U_R(\cF)$ is hypercomplete. For the other direction, a map $f:U_R(\cF) \to \cM$ of sheaves of spectra exhibits $\cM$ as the hypercompletion of $U_R(\cF)$ if and only if the following two conditions are satisfied (cf. \cite[Proposition 2.32]{cm21}):
\begin{enumerate}
    \item The map $f$ is a $\pi_{\ast}$-isomorphism.
    \item The sheaf $\cM$ is hypercomplete.
\end{enumerate}
Let $c: \cF \to \iota(\cF^h)$ be the hypercompletion map for $\cF$. From the first part it follows that $U_R\iota(\cF^h)$ is hypercomplete. Since $U_R$ commutes with homotopy groups as a t-exact functor $U_R(c)$ is a $\pi_{\ast}$-isomorphism. Hence $U_R(c): U_R(\cF) \to U_R(\iota(\cF^h))$ exhibits $U_R(\iota(\cF^h))$ as the hypercompletion of $U_R(\cF)$. Now if $U_R(\cF)$ is hypercomplete then this map is an equivalence, and since $U_R$ is conservative $c$ is an equivalence, which means that $\cF$ is hypercomplete, which concludes the proof.
\end{proof}

\begin{nota}
For a profinite group $G$, we denote by $\gset$ the category of finite continuous $G$-sets endowed with the discrete topology and $G$-equivariant maps. It becomes a site by declaring covering sieves to be families of jointly surjective maps.
\end{nota}

When $G$ is finite, the sheaf category with ring spectrum coefficients we obtain for this site has a familiar description. 

\begin{lem}\label{lem:sheavesfunbg}
Let $R \in \calg(\spec)$ be connective and let $G$ be a finite group. Then there is a t-exact equivalence
$$\sheaves(\gset;R) \simeq \fun(BG,\Mod_R(\spec)).$$
\end{lem}
\begin{proof}
This is \cite[Example 4.4]{cm21}, we just need to show that the equivalence is t-exact. Let $\alpha: BG \to \gset^\op$ be the functor that picks out the free $G$-set $G/1$. Then restriction along $\alpha$ and right Kan extension is an adjoint equivalence.
\[\begin{tikzcd}[column sep=30, row sep=30]
	& {\sheaves(\gset;R)} \\
	{\fun(BG,\Mod_R(\spec))} & {\cP(\gset;R)}
	\arrow["{\alpha^*}", shift right=3, shorten <=10pt, from=1-2, to=2-1]
	\arrow[""{name=0, anchor=center, inner sep=0}, "\iota", shift left=2, hook', from=1-2, to=2-2]
	\arrow["{\alpha_*}", shift left=5, shorten >=10pt, from=2-1, to=1-2]
	\arrow["\sim"{marking, allow upside down}, shift left=4, draw=none, from=2-1, to=1-2]
	\arrow[""{name=3, anchor=center, inner sep=0}, "{\tsf{a}}", shift left=2, from=2-2, to=1-2]
	\arrow[""{name=4, anchor=center, inner sep=0}, "{\hat{\alpha}^*}"', from=2-2, to=2-1]
\end{tikzcd}\]
Here $\tsf{a}$ is the sheafification functor and $\hat{\alpha}^*$ is the restriction along $\alpha$ on presheaves. We have $\alpha^* \simeq \hat{\alpha}^*\iota$, and hence $\alpha_* \simeq \tsf{a}\hat{\alpha}_! $ by passing to left adjoints, where $\hat{\alpha}_!$ is the left adjoint of $\hat{\alpha}^*$. Here we use that $\alpha_*$ is both left and right adjoint to $\alpha^*$ since they are mutually inverse equivalences. Since $\hat{\alpha}^*$ is t-exact, $\hat{\alpha}_!$ is right t-exact, and since the sheafification $\tsf{a}$ is t-exact it follows that $\alpha_*$ is right t-exact. \\
On the other hand, the right adjoint $\hat{\alpha}_*$ of $\hat{\alpha}^*$ already lands in sheaves, and hence $\alpha_*$ factors as $\tsf{a}\hat{\alpha}_*$. Again since $\hat{\alpha}^*$ is t-exact, $\hat{\alpha}_*$ is left t-exact, and hence $\alpha_*$ is left t-exact as well, hence t-exact.
\end{proof}

\begin{rem}
It follows from \cite[Lemma 4.3]{cm21} that the right Kan extension $\alpha_*\cF$ of a functor $\cF: BG \to \Mod_R(\spec)$ to $\gset^\op$ sends the $G$-set $G/H$ to the $H$-homotopy fixed points $\cF(\ast_{BG})^{hH}$.
\end{rem}

Together with the observation that for a profinite group $G$, the site $\gset$ is a filtered colimit of the sites associated to the finite quotients of $G$, we get:

\begin{prop}\label{prop:sheavesborelcomplete}
Let $R \in \calg(\spec)$ be connective and let $G= \tn{lim}_i \, G_i$ be a profinite group. Then the inclusion of hypercomplete objects gives a commutative diagram
\[\begin{tikzcd}
	{\sheaves^{\wedge}(\gset;R)} & {\Mod_{R_G}(\specg)_{\tn{Borel}}} \\
	{\sheaves(\gset;R)} & {\Mod_{R_G}(\specg)_{\tn{lwBorel}}}
	\arrow["\sim", from=1-1, to=1-2]
	\arrow[hook, from=1-1, to=2-1]
	\arrow[hook, from=1-2, to=2-2]
	\arrow["\sim", from=2-1, to=2-2]
\end{tikzcd}\]
in which both horizontal functors are symmetric monoidal equivalences.
\end{prop}
\begin{proof}
We first consider the lower horizontal functor. By \cite[Construction 4.5]{cm21} the site $\gset$ decomposes into a filtered colimit of the $G_i\tsf{-set}$ along the canonical maps $G_j\tsf{-set} \to G_i\tsf{-set}$ that inflate the $G_j$-action to a $G_i$-action along $G_i \twoheadrightarrow G_j$ whenever $N_i \subseteq N_j$. Then the sheaf category $\sheaves(\gset;R)$ can be written as the cofiltered $\prr$-limit
$$\tn{lim}_i \; \sheaves(G_i\tsf{-set};R) \stackrel{\ref{lem:sheavesfunbg}}{\simeq} \tn{lim}_i \; \fun(BG_i,\Mod_R(\spec))$$
with transition maps $\sheaves(G_i\tsf{-set};R) \to \sheaves(G_j\tsf{-set};R)$ given by restriction along the above functors of sites, cf. \cite[§3.1]{cm21}. In terms of the right-hand side, these are the maps $(-)^{h(N_j/N_i)}$. Passing to left adjoints, this is precisely the $\calg(\prl)$-colimit that defines $\Mod_{R_G}(\specg)_{\tn{lwBorel}}$. By \cite[Remark 1.3.2.8]{lur18} the transition functors in the colimit diagram for $\sheaves(\gset;R)$ are t-exact, hence the t-structure on the latter is the filtered colimit of the t-structures on the $\sheaves(G_i\tsf{-set};R)$, as in \Cref{prop:tstructureonlevelwiseborelcomplete}. Since each of the equivalences coming from \Cref{lem:sheavesfunbg} is t-exact it follows that the equivalence 
$$\sheaves(\gset;R) \simeq \Mod_{R_G}(\specg)_{\tn{lwBorel}}$$
is t-exact. Now the top horizontal equivalence follows from \Cref{thm:borelandhyperlewelwise} and \Cref{prop:tstructureonlevelwiseborelcomplete} by passing to hypercomplete objects.
\end{proof}

In \cite[§4.1]{cm21}, it is studied when hypercompletion on $\sheaves(\gset;\spec)$ is smashing, and criteria for sheaves of spectra to be hypercomplete in terms of weakly nilpotent group actions are given. Using this we can give assumptions under which the vertical functors in \Cref{prop:sheavesborelcomplete} are actually equivalences. We denote by $R^s$ the sheafification of the constant presheaf of spectra with value $R$. 

\begin{prop}[Clausen--Mathew]\label{prop:smashinghypercompletion}
Under each of the following assumptions the two vertical functors in \Cref{prop:sheavesborelcomplete} are equivalences:
\begin{enumerate}
    \item The virtual cohomological dimension of $G$ is finite and $R^s$ is hypercomplete. (The latter in particular holds when the ring spectrum $R$ is truncated.)
    \item The group $G$ has finite cohomological dimension $d$ and $R^s$ satisfies either of the two (equivalent) conditions:
    \begin{enumerate}
        \item For every normal containment $N \unlhd H$ of open subgroups the spectrum with $H/N$-action $R^s(G/N)$ is d-nilpotent.
        \item There is a $d' \geq 0$ such that for every open normal subgroup $N \unlhd G$ the spectrum with $G/N$-action $R^s(G/N)$ is $d'$-nilpotent.
    \end{enumerate}
\end{enumerate}
\end{prop}
\begin{proof}
Part (1) is \cite[Corollary 4.28]{cm21}, see also \cite[Corollary 4.29]{cm21} for the statement about truncatedness. The second assertion follows from the first and \cite[Theorem 4.26]{cm21}.
\end{proof}

\begin{rem}
Loosely speaking, the condition of $d$-nilpotency for a spectrum $X$ with an action of a finite group $G$ means that $X$ can be built through retracts and extensions from an induced spectrum $\oplus_G Y$ in $d+1$ many steps. For the precise definition we refer to \cite[Definition 4.18]{cm21}, see also \cite{mnn17,mnn19}. The statement of \cite[Lemma 4.20]{cm21} gives a characterisation in terms of the $d$-skeleton of the universal space for free $G$-actions $EG$.
\end{rem}

\begin{rem}
When $G$ is finite, the colimit diagram computing the levelwise Borel complete category has a terminal object and
$$\Mod_{R_G}(\specg)_{\tn{lwBorel}} \simeq \Mod_{R_G}(\specg)_{\tn{Borel}}.$$
Note also that the t-structure on $\Mod_R(\spec)$ is left complete, and hence the functor category $\fun(BG,\Mod_R(\spec))$ is already hypercomplete. This is consistent with part (1) of \Cref{prop:smashinghypercompletion}, as in this case $R^s$ is hypercomplete since every object of $\fun(BG,\Mod_R(\spec))$ is Postnikov-complete (that is, equivalent to the limit of its Postnikov tower).
\end{rem}

\begin{rem}
Under the assumption that the virtual cohomological dimension of $G$ is finite, Clausen and Mathew also show that hypercompletion on $\sheaves(\gset;R)$ (for general connective $R \in \calg(\spec)$) is smashing and agrees with Postnikov completion \cite[Corollary 4.28, Proposition 2.10]{cm21}. In view of \Cref{prop:sheavesborelcomplete} this means that in this case the hypercompletion
$$L^h: \Mod_{R_G}(\specg)_{\tn{lwBorel}} \to \Mod_{R_G}(\specg)_{\tn{Borel}}$$
is computed as the Postnikov completion $X \mapsto \lim_{n} \tau_{\leq n} X \simeq X \otimes \lim_n \tau_{\leq n} R^s$.
\end{rem}

When $R$ is a discrete commutative ring, we can give another description of the sheaf categories in \Cref{prop:sheavesborelcomplete}. As before, we write $\modgr$ for the (Grothendieck abelian) category of discrete $(G;R)$-modules, where $G$ is a profinite group, see e.g. \cite[§2]{bg23b}. Its derived $\infty$-category $\cD(\modgr)$ carries a natural t-structure \cite[Proposition 1.3.5.21]{lur17}, and although this t-structure is generally not left complete, $\cD(\modgr)$ is hypercomplete, since it has no non-zero acyclics.

\begin{prop}\label{prop:sheavesgrmodules}
Let $R$ be a discrete commutative ring and let $G= \tn{lim}_i \, G_i$ be a profinite group. Then the inclusion of hypercomplete objects gives a commutative diagram
\[\begin{tikzcd}
	{\sheaves^{\wedge}(\gset;R)} & {\cD(\modgr)} \\
	{\sheaves(\gset;R)} & {\tn{colim}_i \; \cD(\Mod(G_i;R))}
	\arrow["\sim", from=1-1, to=1-2]
	\arrow[hook, from=1-1, to=2-1]
	\arrow[hook, from=1-2, to=2-2]
	\arrow["\sim", from=2-1, to=2-2]
\end{tikzcd}\]
in which both horizontal functors are symmetric monoidal equivalences.
\end{prop}
\begin{proof}
Again consider the colimit decomposition of $\sheaves(\gset;R)$ from \cite[Construction 4.5]{cm21}. For each $i \in I$, we have 
$$\sheaves(G_i\tsf{-set};R) \simeq \fun(BG_i,\cD(R)) \simeq \cD(\Mod(G_i;R)),$$
where the second equivalence follows from the identification of $\Mod(G_i;R)$ with modules over the group algebra $R[G_i]$ and the Schwede--Shipley theorem \cite[Theorem 7.1.2.1]{lur17}. This equivalence is t-exact. \\
The functor $\sheaves(G_j\tsf{-set};R) \to \sheaves(G_i\tsf{-set};R)$ that is left adjoint to the restriction along the inflation map $G_j\tsf{-set} \to G_i\tsf{-set}$ for $N_i \subseteq N_j$ corresponds to the functor on derived categories induced by $\Mod(G_j;R) \to \Mod(G_i;R)$, the exact functor that inflates a $G_j$-action to a $G_i$-action along $G_i \twoheadrightarrow G_j$. In terms of group algebras, this is a restriction of scalars along the canonical map $R[G_i] \to R[G_j]$. We hence obtain an equivalence 
$$\sheaves(G\tsf{-set};R) \simeq \tn{colim}_i \; \cD(\Mod(G_i;R)),$$
where the colimit is taken in $\calg(\prlst)$ along the above maps. Using the same argument as in \Cref{prop:tstructureonlevelwiseborelcomplete} we obtain a t-structure on $\tn{colim}_i \; \cD(\Mod(G_i;R))$ that is compatible with filtered colimits and has $\modgr$ as a heart. It makes the above equivalence t-exact. It remains to show that $\cD(\modgr)$ identifies with the full subcategory of hypercomplete objects of the above colimit. \\
The inflation functors $\Mod(G_i;R) \to \modgr$ are exact and provide t-exact functors $\tn{infl}^G_{G_i}: \cD(\Mod(G_i;R)) \to \cD(\Mod(G;R))$. Then there is an induced symmetric monoidal left adjoint 
$$\alpha: \tn{colim}_i \; \cD(\Mod(G_i;R)) \to \cD(\modgr).$$
Write $\alpha_i: \cD(\Mod(G_i;R)) \to \tn{colim}_i \; \cD(\Mod(G_i;R))$ for the canonical map and $\beta_i$ for its right adjoint. Let $M \in \tn{colim}_i \; \cD(\Mod(G_i;R))$. Then \cite[Lemma 6.3.3.6]{lur09} implies that $M \simeq \tn{colim}_i \; \alpha_i \beta_i(M)$ and $\alpha(M) \simeq \tn{colim}_i \; \tn{infl}^G_{G_i} \beta_i(M)$, so
$$\pi_*(\alpha(M)) \cong \tn{colim}_i (\tn{infl}^G_{G_i})^\heartsuit \pi_*(\beta_i(M)) \cong \pi_*(M).$$
In particular, $\alpha$ is t-exact. Since the t-structure on $\cD(\modgr)$ is hypercomplete, we have that
$$\alpha(M) \simeq 0 \Leftrightarrow \pi_*(\alpha(M)) \cong 0 \Leftrightarrow \pi_*(M) \cong 0.$$
This means that the kernel of $\alpha$ precisely consists of the acyclic objects, and we need to show that the right adjoint of $\alpha$ is fully faithful. Denote by $\widecheck{\cD}(\modgr)$ the unseparated derived $\infty$-category of \cite[§C.5.8]{lur18}. There is a canonical localisation functor that quotients by the acyclic complexes
$$\tilde{\alpha}: \widecheck{\cD}(\modgr) \twoheadrightarrow  \widecheck{\cD}(\modgr)/\widecheck{\cD}(\modgr)_{\tn{ac}} \simeq \cD(\modgr).$$
Its right adjoint is fully faithful and identifies the target with the hypercomplete objects of the source. Combining the universal property of the prestable part of the unseparated derived $\infty$-category \cite[Corollary C.5.8.9]{lur18} with \cite[Proposition C.3.1.1, C.3.2.1]{lur18} we see that $\tilde{\alpha}$ factors as $\alpha \circ \check{\alpha}$. Here 
$$\check{\alpha}: \widecheck{\cD}(\modgr) \to \tn{colim}_i \; \cD(\Mod(G_i;R))$$
is the unique t-exact left adjoint that is given by the identity on hearts. Note that $\modgr$ is itself the filtered colimit of the categories $\Mod(G_i;R)$ along the left exact left adjoint inflation functors $\Mod(G_j;R) \to \Mod(G_i;R)$ induced by the quotient maps $G_i \twoheadrightarrow G_j$. So by \cite[Proposition C.5.5.20, C.3.3.5]{lur18} we can also write $\widecheck{\cD}(\modgr)$ as a $\prl$-colimit of the $\widecheck{\cD}(\Mod(G_i;R))$, and the right adjoint of $\check{\alpha}$ is a $\prr$-limit of fully faithful functors 
$$\cD(\Mod(G_i;R)) \hookrightarrow \widecheck{\cD}(\Mod(G_i;R)).$$
So the right adjoints of $\check{\alpha}$ and $\tilde{\alpha}$ are fully faithful, and hence the right adjoint of $\alpha$ must be fully faithful as well. This finishes the proof.
\end{proof}

\begin{rem}
Recall the classical equivalence $\modgr \simeq \tn{Sh}(\gset;R)$, where the latter is the 1-category of sheaves of (discrete) $R$-modules on $G$-sets, and $G$ again is a profinite group. Then there is a fully faithful embedding 
$$\modgr \simeq \tn{Sh}(\gset;R) \hookrightarrow \sheaves(\gset;R),$$
and by \cite[Corollary 2.1.2.3]{lur18} it extends to a fully faithful functor
$$\cD(\modgr) \hookrightarrow \sheaves(\gset;R)$$
which has essential image given by hypersheaves. This provides an alternative approach to the previous proposition.
\end{rem}

In this setting, \Cref{prop:smashinghypercompletion} implies the following.

\begin{cor}\label{cor:smashinghypercompletionreps}
If the virtual cohomological dimension of the profinite group $G$ is finite, the vertical functors in \Cref{prop:sheavesgrmodules} are equivalences.
\end{cor}

\begin{rem}
By definition the unseparated derived $\infty$-category $\widecheck{\cD}(\modgr)$ of \cite[Corollary C.5.8.9]{lur18} agrees with Krause's homotopy category of injectives 
$$\cK(\tn{Inj} \; \modgr)$$ 
of \cite{kra05,kra15}. It generally differs from the filtered colimit of derived categories we considered above, this is already the case for finite groups. \\
When $G$ is finite, and when the coefficient ring $R$ is assumed regular and noetherian, then a theorem of Mathew shows that $\cK(\tn{Inj} \; \modgr)$ coincides with the category of modules over the Borel completion of $R_G$ in $G$-spectra, which in this case also agrees with $\ind\,\fun(BG,\tsf{Perf}(R))$, see \cite[Theorem 3.7]{bar21}, \cite[Theorem A.4]{tre15}. We refer to \cite[Remark 3.35]{bchnl25} for further discussion.
\end{rem}

\section{Smooth Artin motives}\label{sec:motives}

We now start with the geometric part of the paper, in which we will consider $\infty$-categories of smooth Artin motives for the Nisnevich and the étale topology. They are constructed using the theory of finite correspondences \cite[§9]{cd19}. For versions of (smooth) Artin motives for the étale topology without correspondences we refer the reader to \cite{az12}, \cite{cdn23} and \cite{rui25a,rui25b}. \\ 
We will work over a base scheme $S$ that we always assume to be connected, noetherian, and of finite Krull dimension.\footnote{Connectedness is assumed for convenience, it is not of technical relevance and one can restrict to individual connected components when working with étale fundamental groups. The assumption on the Krull dimension is a technical one and will only be used for hypercompleteness of the Nisnevich topos. Assuming $S$ to be noetherian is also used for the latter (for this it would be enough that $S$ is qcqs), but it is also important for the theory of finite correspondences.} All morphisms of schemes are assumed to be separated and of finite type. In particular, all schemes considered are noetherian. Throughout, $R$ is a discrete commutative ring.

\begin{defi}
Let $\sms$ be the category of smooth $S$-schemes with morphisms given by morphisms of $S$-schemes, and let $\fets$ be the full subcategory on finite étale $S$-schemes. The fibre product over $S$ endows both $\sms$ and $\fets$ with a symmetric monoidal structure which commutes with coproducts in both variables. 
\end{defi}

\begin{rec}[{\cite[§8.1.a]{cd19}}]\label{rec:cycles}
For a scheme $X$, a \emph{cycle with domain $X$} is a $\bZ$-linear combination of points of $X$. We write 
$$\langle X \rangle \defeq \sum_{x \in X^{(0)}} \tn{lg}(\cO_{X,x}).x$$ 
for the \emph{cycle associated to $X$}.\footnote{Here $\tn{lg}(\cO_{X,x})$ is the geometric multiplicity of $x \in X$, i.e. the length of the local ring $\cO_{X,x}$ (considered as a module over itself).} The sum runs over the generic points\footnote{Throughout, we understand a generic point as one of an irreducible component.} of $X$. When $Z$ is a closed subscheme of a scheme $X$, the cycle $\langle Z \rangle$ is understood as a cycle with domain $X$. If $f: X \to Y$ is a morphism of schemes, the \emph{pushforward} of a cycle $\alpha=\sum_{i \in I} n_ix_i$ with domain $X$ by $f$ is the cycle with domain $Y$ given by 
$$f_*(\alpha) = \sum_{i \in I} n_id_i.f(x_i),$$
where $d_i$ is the degree of the extension $\kappa(x_i)/\kappa(f(x_i))$ if it is finite and $0$ otherwise. 
\end{rec}

\begin{defi}\label{defi:correspondence}
Let $X$ and $Y$ be smooth $S$-schemes. Then $\tsf{Cor}_S(X,Y)$, the group of \emph{finite $S$-correspondences from $X$ to $Y$}, is the abelian group consisting of the finite and universally integral cycles over $S$ with domain $X \times_S Y$. \\
Using each of the tensor products $\tsf{Cor}_S(X,Y) \otimes R$ as hom-sets defines an additive category $\corsms{R}$ which has objects the smooth $S$-schemes and morphisms the finite $S$-correspondences with coefficients in $R$. Composition is defined using pushforward and intersection of cycles, and the identity of a smooth $S$-scheme $X$ is given by the cycle $\langle \Delta_X \rangle$, where $\Delta_X \subseteq X\times_S X$ is the diagonal \cite[§9.1]{cd19}. \\
By $\corfets{R}$ we denote the full subcategory of $\corsms{R}$ spanned by the finite étale $S$-schemes. Finite direct sums are given by the disjoint union of schemes, and the fibre product over $S$ endows both categories with an additively symmetric monoidal structure \cite[§9.2]{cd19}.
\end{defi}

\begin{rem}\label{rem:nicecorrespondences}
The definition given above is the most general one taken from \cite[§9.1]{cd19}, which goes back to \cite[§3]{sv00}, \cite[Appendix 1A]{mvw06}, and we will not need it in its full detail. We want to highlight three special cases of which the third one will be of importance later:
\begin{enumerate}
    \item[(1)] If the base scheme $S$ is regular, then $\tsf{Cor}_S(X,Y)$ is the free abelian group on integral closed subschemes $Z \subseteq X \times_S Y$ which are finite and equidimensional over $X$ (that is, they are finite over $X$ and dominate an irreducible component of $X$), see \cite[Remark 8.3.29, 9.1.3]{cd19}.
    \item[(2)] Assume that $X$ is smooth and $Y$ étale over $S$, write $p_1: X \times_S Y \to X$ for the canonical (étale) morphism. Then by \cite[Lemma 10.2.6]{cd19} the group $\tsf{Cor}_S(X,Y)$ is free abelian on the set of connected components $U$ of $X \times_S Y$ such that $p_1(U)$ is equal to a connected component of $X$.
    \item[(3)] From (2) it follows that if $X$ is smooth and $Y$ finite étale over $S$, then $\tsf{Cor}_S(X,Y)$ is the free abelian group on the connected components $X \times_S Y$, cf. \cite[Proposition 10.2.5]{cd19}.
\end{enumerate}
The apparent discrepancy between (1) and (2) can be explained as follows: every connected scheme $X$ that is smooth over a regular and hence normal scheme $S$ is already itself normal, and since it is assumed noetherian it is integral.
\end{rem}

\begin{rem}
In \cite[§9]{cd19} finite correspondences with coefficients in a ring $R$ are defined by means of $\Lambda$-universal cycles, where $\Lambda \subseteq \bQ$ is the localisation of $\bZ$ at the integers that are invertible in $R$. By \cite[Remark 9.1.3]{cd19} tensoring the resulting correspondence groups over $\Lambda$ with $R$ results in the hom-modules of $\corsms{R}$ defined above. 
\end{rem}

\begin{rem}\label{rem:isoetalecorrespondences}
The identification made in (2) of \Cref{rem:nicecorrespondences} poses a subtlety: it identifies connected components with the sum of their generic points. For example, the identity of $S$ in $\tsf{Cor}_S(S,S)$ is given by $\langle \Delta_S \rangle$, but if $S$ is only connected but not irreducible $\Delta_S \cong S$ can have multiple irreducible components. Under the identification made in (2) the cycle $\langle \Delta_S \rangle$ corresponds to the unique connected component of $S \times_S S$. In fact, the proof of \cite[Lemma 10.2.6]{cd19} shows that for $X$ a smooth and $Y$ an étale $S$-scheme the map 
$$\bZ[\pi_0^{\tn{fin}}(X \times_S Y/X)] \to \tsf{Cor}_S(X,Y), \; U \mapsto \langle U \rangle = \sum_{u \in U^{(0)}} \tn{lg}(\cO_{U,u}).u$$
is an isomorphism, where $\pi_0^{\tn{fin}}(X \times_S Y/X)$ is the set of connected components $U$ of $X \times_S Y$ such that the projection $p_1(U)$ is equal to a connected component of $X$. When $Y$ is finite étale over $S$ the condition on the projection can be dropped.
\end{rem}

\begin{rec}\label{rec:graphfunctors}
Let $f: X \to Y$ be a morphism of smooth $S$-schemes. Then its graph $\Gamma_f$ is a closed subscheme of $X \times_S Y$, and by \cite[Example 9.1.4]{cd19} the cycle $\langle \Gamma_f \rangle$ with domain $X \times_S Y$ defines an element of $\tsf{Cor}_S(X,Y)$. This construction gives faithful functors $\gamma: \sms \to \corsms{R}$ and $\gamma: \fets \to \corfets{R}$ which are the identity on objects and take a morphism of $S$-schemes $f: X \to Y$ to the cycle $\langle \Gamma_f \rangle$. They are symmetric monoidal \cite[{}9.2.4]{cd19}.
\end{rec}

\begin{rem}
Using the description of correspondence groups given in \Cref{rem:nicecorrespondences} it follows that all objects of $\corfets{R}$ are self-dual with respect to the symmetric monoidal structure of \Cref{defi:correspondence}.
\end{rem}

\begin{nota}
In what follows, $\tau$ will denote either the étale or the Nisnevich topology on $\sms$ or a full subcategory thereof. For the former, covering families are finite families $(X_i \to X)_{i \in I}$ of étale morphisms that are jointly surjective, and for the latter we furthermore demand that for every $x \in X$ there exists $i \in I$ and $x_i \in X_i$ mapping to $x$ such that the induced map of residue fields $\kappa(x) \to \kappa(x_i)$ is an isomorphism. \\
For $\cC$ an additive $\infty$-category (which in our case will be a correspondence category) we will write $\psigma(\cC;\spec)$ for the $\infty$-category of additive presheaves on $\cC$ with values in spectra. We refer to Appendix \ref{sec:sheaveswithtransfers} for some recollections.
\end{nota}

\begin{defi}\label{defi:sheaveswithtransfers}
The objects of $\psigma(\corsms{R};\spec)$ are called \emph{presheaves with transfers}. We define the $\infty$-category  $\sheaves^{(\wedge)}_{\tau}(\corsms{R};\spec)$ of \emph{$\tau$-(hyper)sheaves with transfers} as the pullback 
\[\begin{tikzcd}
	{\sheaves^{(\wedge)}_{\tau}(\corsms{R};\spec)} & {\psigma(\corsms{R};\spec)} \\
	{\sheaves^{(\wedge)}_{\tau}(\sms;\spec)} & {\psigma(\sms;\spec),}
	\arrow[from=1-1, to=1-2, "{\tsf{o}^{(\wedge),\tn{tr}}_{\tau}}", hook]
	\arrow[from=1-1, to=2-1, "\gamma^*"]
	\arrow[from=1-2, to=2-2, "\hat{\gamma}^*"]
	\arrow[from=2-1, to=2-2, "{\tsf{o}_{\tau}^{(\wedge)}}", hook]
    \arrow["\usebox\pullback"{anchor=center, pos=0.125}, draw=none, from=1-1, to=2-2]
\end{tikzcd}\]
where the horizontal functors are the canonical inclusions and vertical functors precompose with the graph functor $\gamma: \sms \to \corsms{R}$. When $\sms$ is replaced by $\fets$ we denote the resulting sheaf category by $\sheaves^{(\wedge)}_{\tau}(\corfets{R};\spec)$.
\end{defi}

Note that in the above pullback we use that sheaves for both topologies already preserve finite products, see e.g. \cite[Proposition A.3.3.1]{lur18}.

\begin{rem}
Under our assumptions on the base scheme $S$ the Nisnevich topos $\sheaves_{\tn{Nis}}(\sms)$ is hypercomplete, which follows from \cite[Theorem 3.7.7.1]{lur18} or \cite[Theorem 3.18]{cm21} using the argument given in \cite[Lemma 5.1]{mat24}. Furthermore, by \cite[Proposition 1.3.3.3]{lur18} a sheaf of spectra is a hypersheaf if and only if the underlying sheaf of spaces is one, so we do not need to distinguish between sheaves and hypersheaves for the Nisnevich topology.
\end{rem}

\begin{rem}\label{rem:tstructuresheaveswithtransfers}
In the étale case, we can express hypersheaves with transfers as the iterated pullback
\[\begin{tikzcd}
	{\sheaves^{\wedge}_{\tn{ét}}(\corsms{R};\spec)} & {\sheaves_{\tn{ét}}(\corsms{R};\spec)} & {\psigma(\corsms{R};\spec)} \\
	{\sheaves^{\wedge}_{\tn{ét}}(\sms;\spec)} & {\sheaves_{\tn{ét}}(\sms;\spec)} & {\psigma(\sms;\spec).}
	\arrow["{\tsf{o}_{\tn{ét}}^{\wedge,\tn{tr}}}", hook, from=1-1, to=1-2]
	\arrow["{\gamma^*}", from=1-1, to=2-1]
	\arrow["{\tsf{o}_{\tn{ét}}^{\tn{tr}}}", hook, from=1-2, to=1-3]
	\arrow["{\gamma^*}", from=1-2, to=2-2]
	\arrow["{\hat{\gamma}^*}", from=1-3, to=2-3]
	\arrow["{\tsf{o}_{\tn{ét}}^{\wedge}}", hook, from=2-1, to=2-2]
	\arrow["{\tsf{o}_{\tn{ét}}}", hook, from=2-2, to=2-3]
    \arrow["\usebox\pullback"{anchor=center, pos=0.125}, draw=none, from=1-1, to=2-2]
    \arrow["\usebox\pullback"{anchor=center, pos=0.125}, draw=none, from=1-2, to=2-3]
\end{tikzcd}\]
Recall from \Cref{prop:hypersheafhypercomplete} that hypercomplete sheaves are precisely the hypercomplete objects in sheaves. There is a right complete t-structure on sheaves with transfers which makes the functor $\gamma^*$ to sheaves t-exact (see \Cref{rem:tstrucuresonsheaveswithtransfers}), and hypersheaves with transfers are precisely the hypercomplete objects with respect to this t-structure, by essentially the same argument as in \Cref{prop:hypersheafhypercomplete}.\\
In view of \Cref{rem:pushoutinsteadofpullback} all the functors in the diagram above have left adjoints, and in particular hypercompletion on étale sheaves induces a left adjoint on étale sheaves with transfers
$$(-)^\wedge_\tn{ét}: \sheaves_{\tn{ét}}(\corsms{R};\spec) \to \sheaves_{\tn{ét}}^\wedge(\corsms{R};\spec).$$
\end{rem}

\begin{rem}\label{rem:representables}
For a smooth (resp. finite étale) $S$-scheme $X$ the presheaf with transfers 
$$R^{\tn{tr}}_S(X) \defeq \bar{y}(X) \simeq \tsf{Cor}_S(-,X) \otimes R$$ is an étale hypersheaf with transfers, where $\bar{y}: \corsms{R} \hookrightarrow \psigma(\corsms{R};\spec)$ (resp. $\bar{y}: \corfets{R} \hookrightarrow \psigma(\corfets{R};\spec)$) is the stable Yoneda embedding of \Cref{rec:psigma}, which follows from \cite[Proposition 2.1.4]{cd16}, see also \cite[Remark 14.6]{bh21}. This gives an induced symmetric monoidal functor
$$R^{\tn{tr}}_S: \corsms{R} \to \sheaves^{\wedge}_{\tau}(\corsms{R};\spec),$$
and replacing $\tsf{Sm}$ by $\tsf{Fét}$ gives its analogue for finite étale schemes.
\end{rem}

\begin{defi}\label{defi:motives}
The $\infty$-category $\mcal{DM}^{\tn{eff},(\wedge)}_{\tau}(S;R)$ of \emph{effective $\tau$-motives} is the full subcategory of $\sheaves_{\tau}^{(\wedge)}(\corsms{R};\spec)$ on the $\cS$-local objects, where $\cS$ is the collection of morphisms $R^\tn{tr}_S(\bA^1_X) \to R^\tn{tr}_S(X)$ for $X \in \tsf{Sm}_{S}$.
\end{defi}

\begin{rem}
By \cite[Proposition 5.5.4.15]{lur09} the $\infty$-category $\mcal{DM}^{\tn{eff},(\wedge)}_{\tau}(S;R)$ is presentable and a localisation of $\sheaves_{\tau}^{(\wedge)}(\corsms{R};\spec)$. We could also define it as a further pullback along the inclusion of $\bA^1$-invariant sheaves with transfers into the latter. We write 
$$L_{\bA^1}: \sheaves_{\tau}^{(\wedge)}(\corsms{R};\spec) \to \mcal{DM}^{\tn{eff},(\wedge)}_{\tau}(S;R)$$
for the left adjoint localisation functor. By \Cref{rem:pushoutinsteadofpullback}, $\sheaves_{\tau}^{(\wedge)}(\corsms{R};\spec)$ has a symmetric monoidal structure, and using \cite[Proposition 2.2.1.9]{lur17} one can check that the functor $L_{\bA^1}$ above is a symmetric monoidal localisation. We obtain a symmetric monoidal composition
$$\cM_{\tau}^{(\wedge)}: \sms \stackrel{\gamma}{\longrightarrow} \corsms{R} \stackrel{R^{\tn{tr}}_S}{\longrightarrow} \sheaves_{\tau}^{(\wedge)}(\corsms{R};\spec) \stackrel{L_{\bA^1}}{\longrightarrow} \mcal{DM}^{\tn{eff},(\wedge)}_{\tau}(S;R),$$
and for $X \in \sms$ we call $\cM_{\tau}^{(\wedge)}(X)$ the \emph{effective motive associated to $X$}. For the Nisnevich topology, it follows from the finite cohomological dimension of the base scheme \cite[{}1.2.5]{ks86} that $\mcal{DM}_{\tn{Nis}}^{\tn{eff}}(S;R)$ is compactly generated by the motives $\cM_{\tn{Nis}}(X)$ for $X$ smooth over $S$.
\end{rem}

\begin{defi}\label{defi:smoothartinmotives}
The $\infty$-category $\mcal{DAM}_{\tau}^{(\wedge)}(S;R)$ of \emph{smooth Artin $\tau$-motives} is defined as the localising subcategory of $\mcal{DM}^{\tn{eff},(\wedge)}_{\tau}(S;R)$ generated by the motives associated to finite étale $S$-schemes. Instead taking the thick subcategory yields $\mcal{DAM}^{(\wedge)}_{\tau, \tn{c}}(S;R)$, its objects are called \emph{constructible smooth Artin motives}. 
\end{defi}

As in the Nisnevich case the generators of $\mcal{DAM}_{\tn{Nis}, \tn{c}}(S;R)$ are compact, we can also describe the latter as the thick subcategory of the compact part $\mcal{DM}^{\tn{eff}}_{\tn{Nis}}(S;R)^\omega$ generated by the motives associated to finite étale $S$-schemes, or as the compact part $\mcal{DAM}_{\tn{Nis}}(S;R)^\omega$, using Neeman--Thomason localisation.\footnote{The compact parts of effective motives and smooth Artin motives are also referred to as \emph{geometric} effective/smooth Artin motives.} Since the motives associated to finite étale schemes are self-dual, in particular dualisable, $\mcal{DAM}_{\tn{Nis}}(S;R)$ is rigidly compactly generated.

\begin{rem}
Recall that one passes from effective motives $\mcal{DM}^{\tn{eff},(\wedge)}_{\tau}(S;R)$ to non-effective motives $\mcal{DM}^{(\wedge)}_{\tau}(S;R)$ by tensor inverting the effective motive associated to the motivic sphere $\bS^{\bA^1} = \tn{cofib}(\cM_{\tau}^{(\wedge)}(S) \to \cM_{\tau}^{(\wedge)}(\bG_{m,S}))$. Here the map $S \to \bG_{m,S}$ is the constant section with value $1$ of the canonical map $\bG_{m,S} \to S$. In general the question when the functor from effective to non-effective motives is fully faithful is subtle, see  e.g. \cite[Theorem 4.3.1]{voe00}, \cite{voe10}, \cite[Corollary 11.2.15]{cd19}. In this paper we will only deal with effective motives.
\end{rem}

\begin{cons}\label{cons:fettosm}
The inclusion functor $\iota: \corfets{R} \hookrightarrow \corsms{R}$ induces an adjunction 
$$(\hat{\iota}_!,\hat{\iota}^*): \psigma(\corfets{R};\spec) \rightleftarrows \psigma(\corsms{R};\spec)$$ 
of categories of presheaves with transfers. Here $\hat{\iota}^*$ is restriction along $\iota$, and $\hat{\iota}_!$ is given by left Kan extension along $\iota$, which restricts to $\psigma(-;R)$ by \cite[Proposition 3.3.2]{chll24}. The pointwise formula for the left Kan extension
$$\hat{\iota}_!F(Y) \simeq \tn{colim}_{(\corfets{R}^\op)_{/Y}} \; F$$
ensures that $\hat{\iota}_!(R^{\tn{tr}}_S(X)) \simeq R^{\tn{tr}}_S(X)$ holds for a finite étale scheme $X$. Using the pullback squares defining (hyper)sheaves with transfers and the analogous adjunction on (pre)sheaves, we obtain an induced adjunction
$$(\iota_!,\iota^*): \sheaves_{\tau}^{(\wedge)}(\corfets{R};\spec) \rightleftarrows \sheaves_{\tau}^{(\wedge)}(\corsms{R};\spec).$$
Here $\iota^*$ is the restriction of $\hat{\iota}^*$ to sheaves and $\iota_!$ is given by the composition of $\hat{\iota}_!$ and (hyper)sheafification `with transfers', see \Cref{rem:pushoutinsteadofpullback} for the latter. 
\end{cons}

Fix a geometric point $s$ of the base scheme $S$, i.e. a morphism $s: \Spec(k) \to S$ where $k$ is an algebraically closed field. From now on, write $G$ for $\pi^{\tn{ét}}_1(S,s)$, the étale fundamental group of the base scheme \cite{sga1}. We write $\underl{R}$ for the Eilenberg-MacLane spectrum of the constant Mackey functor associated to the ring $R$ and $\permgr$ for the additive category of $(G;R)$-permutation modules \cite[§2]{bg23b}. Our aim now is to extend Voevodsky's theorem \cite[§3.4]{voe00} that expresses Artin motives over a field in terms of permutation modules. 

\begin{thm}\label{thm:damnisasmodunderlr}
There is a symmetric monoidal equivalence 
$$\Mod_{\underl{R}}(\specg) \simeq \mcal{DAM}_{\tn{Nis}}(S;R).$$
\end{thm}
\begin{proof}
We first produce a left adjoint $\Mod_{\underl{R}}(\specg) \to \mcal{DAM}_{\tn{Nis}}(S;R)$ and then apply the Barr-Beck theorem. Using \cite[Example 2.15]{bcn25} or \cite[Corollary 3.18]{fuh25} we can work with $\psigma(\permgr;\spec)$ instead of $\Mod_{\underl{R}}(\specg)$ throughout. Recall Grothendieck's equivalence of finite discrete $G$-sets and finite étale $S$-schemes from \cite{sga1}, and consider the symmetric monoidal functor defined on 1-categories of spans thereof: 
\begin{align*}
\Psi: \spans(G) \simeq \spans(\fets) \to \corfets{R}.
\end{align*}
It is the identity on objects and sends a span of finite étale $S$-schemes, represented by a map $f: U \to X \times_S Y$, to the pushforward correspondence 
$$f_*\langle U \rangle = \sum_{u \in U^{(0)}} \tn{lg}(\cO_{U,u}) \cdot d_u. f(u).$$
By \cite[9.1.1, Lemma 10.2.6]{cd19} the cycle $f_*\langle U \rangle$ is an element of $\tsf{Cor}_S(X,Y)$. The identity span $(X \to X \times_S X)$ is sent to $\langle \Delta_X \rangle = \tn{id}_X \in \tsf{Cor}_S(X,X)$, and composition of finite correspondences as described in \cite[Definition 9.1.5]{cd19} makes $\Psi$ functorial. It preserves finite coproducts and hence extends to a functor
$$\Psi: \Omega(G) \to \corfets{R},$$
where $\Omega(G)$ is the Burnside category of finite $G$-sets obtained by group completing the hom-sets in $\spans(G)$ that are abelian monoids under disjoint union. There is a `cohomological' ideal $\cI(G)$ of homomorphisms in $\Omega(G)$ that is generated by elements
\begin{equation}\label{eq:cohidealgeneator}
(G/H \stackrel{\pi^H_K}{\longleftarrow} G/K \stackrel{\pi^H_K}{\longrightarrow} G/H) - ([H:K] \cdot \tn{id}_{G/H})
\end{equation}
for a pair of open subgroups $K \leq H \leq G$, see \cite[§4]{bg23b}. The left-hand span in \ref{eq:cohidealgeneator} corresponds to a span of finite étale $S$-schemes
\begin{equation}\label{eq:spangenerator}
(X/\overl{H} \stackrel{\pi^{\overl{H}}_{\overl{K}}}{\longleftarrow} X/\overl{K} \stackrel{\pi^{\overl{H}}_{\overl{K}}}{\longrightarrow} X/\overl{H})
\end{equation}
for some finite étale cover $X \to S$ which can be chosen so that the kernel of the canonical homomorphism $\varphi: G \to \tn{Aut}_{\fets}(X)$ is contained in $K$, hence in $H$, and $\overl{H}$, $\overl{K}$ are the respective images of $H$ and $K$ under $\varphi$. The map $\pi^{\overl{H}}_{\overl{K}}$ is the canonical projection, induced by the universal property of the quotient by a group action \cite[Proposition 20.87]{gw23}. We now show that the generators in \ref{eq:cohidealgeneator} and hence the ideal $\cI(G)$ are killed by the functor $\Psi$. \\
Let $f: X/\overl{K} \to X/\overl{H} \times_S X/\overl{H}$ correspond to the span \ref{eq:spangenerator}. It factors through the diagonal $\Delta_{X/\overl{H}} \subseteq X/\overl{H} \times_S X/\overl{H}$ which is isomorphic to $X/\overl{H}$. We thus have a commutative diagram
\[\begin{tikzcd}
	{X/\overl{K}} & {\Delta_{X/\overl{H}}} & {X/\overl{H} \times_S X/\overl{H}.} \\
	& {X/\overl{H}}
	\arrow[from=1-1, to=1-2]
	\arrow["f", bend left=20, from=1-1, to=1-3]
	\arrow[from=1-1, to=2-2, "\pi^{\overl{H}}_{\overl{K}}"']
	\arrow[hook, from=1-2, to=1-3]
	\arrow[from=1-2, to=2-2, "\hspace*{-0.01cm}\rotatebox{90}{$\sim$}"]
	\arrow[from=1-3, to=2-2, "p_1"]
\end{tikzcd}\]
It hence suffices to show that $(\pi^{\overl{H}}_{\overl{K}})_*\langle X/\overl{K} \rangle = [H:K] . \langle X/\overl{H} \rangle$ holds. The canonical projection $\pi^{\overl{H}}_{\overl{K}}: X/\overl{K} \to X/\overl{H}$ is surjective and thus maps generic points to generic points. Let $y$ be a generic point of $X/\overl{H}$. Then the coefficient of $y$ appearing in the cycle $(\pi^{\overl{H}}_{\overl{K}})_*\langle X/\overl{K} \rangle$ is 
\begin{align*}
\sum_{x} \; [\kappa(x) : \kappa(y)] \; \cdot \tn{lg}(\cO_{X/\overl{K},x}) &= \tn{lg}(\cO_{X/\overl{H},y}) \cdot \sum_{x} \; [\kappa(x) : \kappa(y)] \tag*{\cite[{}10.52.13]{stacks}} \\
&= \tn{lg}(\cO_{X/\overl{H},y}) \cdot \tn{deg}(\pi^{\overl{H}}_{\overl{K}}) \tag*{\cite[Prop. 12.21]{gw10}}\\
&= \tn{lg}(\cO_{X/\overl{H},y}) \cdot [H:K] \tag*{\cite[Prop. 20.89]{gw23}}
\end{align*}
where the sum runs over all $x \in (\pi^{\overl{H}}_{\overl{K}})^{-1}(y)$. Since $\tn{lg}(\cO_{X/\overl{H},y})$ is precisely the coefficient of $y$ in $\langle \Delta_{X/\overl{H}} \rangle$ it follows that
$$\Psi(f) = f_* \langle X/\overl{K} \rangle = [H:K] . \langle \Delta_{X/\overl{H}} \rangle = [H:K] . \tn{id}_{X/\overl{H}}.$$
Hence $\Psi$ descends to $\Omega(G)/\cI(G)$. Since the target $\corfets{R}$ is $R$-linear, we can extend the source to $\Omega_R(G)/\cI_R(G)$, as in \cite[Corollary 4.21]{bg23b}. Here $\Omega_R(G)$ is the $R$-linearisation of $\Omega(G)$, and $\cI_R(G)$ is the analogous ideal of $\Omega_R(G)$ on the same generators as for $\cI(G)$. By \cite[Proposition 4.17]{bg23b} the free $R$-module functor induces an equivalence $\permgr \simeq \Omega_R(G)/\cI_R(G)$, and we have an additive symmetric monoidal functor
$$\psi: \permgr \stackrel{\Psi}{\longrightarrow} \corfets{R} \stackrel{\iota}{\longrightarrow} \corsms{R},$$
where the second functor is the canonical inclusion. Using \cite[Proposition 3.3.2]{chll24} restriction and left Kan extension provide an adjunction
$$(\psi_!, \psi^*): \psigma(\permgr;\spec) \rightleftarrows \psigma(\corsms{R};\spec).$$
The functor $\iota_!: \sheaves_{\tn{Nis}}(\corfets{R};\spec) \to \sheaves_{\tn{Nis}}(\corsms{R};\spec)$ of \Cref{cons:fettosm} makes the following diagram commute:
\[\begin{tikzcd}
	\corfets{R} & \corsms{R} \\
	{\sheaves_{\tn{Nis}}(\corfets{R};\spec)} & {\sheaves_{\tn{Nis}}(\corsms{R};\spec).}
	\arrow["\iota", hook, from=1-1, to=1-2]
	\arrow["R^{\tn{tr}}_S", from=1-1, to=2-1]
	\arrow["R^{\tn{tr}}_S", from=1-2, to=2-2]
	\arrow["{\iota_!}", shift left, from=2-1, to=2-2]
\end{tikzcd}\]
Hence, the composition $L_{\bA^1}\tsf{a}_{\tn{Nis}}^{\tn{tr}}\psi_!$ lands in the full subcategory $\mcal{DAM}_{\tn{Nis}}(S;R)$ of $\mcal{DM}_{\tn{Nis}}^{\tn{eff}}(S;R)$. Composing $(\psi_!, \psi^*)$ with the adjunctions 
\[\begin{tikzcd}[column sep=1.4cm]
	{\psigma(\corsms{R};\spec)} & {\sheaves_{\tn{Nis}}(\corsms{R};\spec)} & {\mcal{DM}_{\tn{Nis}}^{\tn{eff}}(S;R)}
	\arrow["{(\tsf{a}_{\tn{Nis}}^{\tn{tr}},\tsf{o}_{\tn{Nis}}^{\tn{tr}})}", shift left, from=1-1, to=1-2]
	\arrow[shift left=2, from=1-2, to=1-1]
	\arrow["{(L_{\bA^1},\tn{incl})}", shift left, from=1-2, to=1-3]
	\arrow[shift left, from=1-3, to=1-2]
\end{tikzcd}\]
we thus obtain an adjunction
$$(\alpha, \beta): \psigma(\Spans(G);R) \rightleftarrows \mcal{DAM}_{\tn{Nis}}(S;R).$$
Now, $\alpha$ sends the presheaf associated to a finite $G$-set (i.e., its free $\underl{R}$-module) to the motive associated to its corresponding finite étale $S$-scheme. By construction $\alpha$ is symmetric monoidal, so it preserves dualisable and hence compact objects, and it sends a set of generators to a set of generators. It follows that $\beta$ preserves colimits and is conservative. Since the projection formula holds automatically in this case, \cite[Proposition 5.29]{mnn17} gives a symmetric monoidal equivalence
$$\mcal{DAM}_{\tn{Nis}}(S;R) \simeq \Mod_{\beta(\cM_{\tn{Nis}(S)})}\Mod_{\underl{R}}(\specg) \simeq \Mod_{\beta(\cM_{\tn{Nis}(S)})}(\specg),$$
where the second equivalence is \cite[Corollary 3.4.1.9]{lur17}. It remains to identify the algebra object $\beta(\cM_{\tn{Nis}}(S))$. For some open subgroup $H \leq G$, the $n$-th homotopy is given by
\begin{align}
\pi_n^{H}(\beta(\cM_{\tn{Nis}}(S))) &= \Hom_{\Mod_{\underl{R}}(\specg)}(\Sigma^{\infty +n}_+G/H \otimes \underl{R},\beta(\cM_{\tn{Nis}}(S))) \notag \\
&\cong \Hom_{\mcal{DAM}_{\tn{Nis}}(S;R)}(\cM_{\tn{Nis}}(X/\overl{H}),\cM_{\tn{Nis}}(S)), \label{eq:homotopy}
\end{align}
where $X/\overl{H}$ is the finite étale $S$-scheme corresponding to the $G$-set $G/H$ - here $X$ is a finite étale $S$-scheme chosen so that $H$ contains the kernel of the canonical homomorphism $G \to \tn{Aut}_{\fets}(X)$, and $\overl{H}$ is the image of $H$ under this homomorphism. By \cite[Theorem 11.2.14, Proposition 10.2.5]{cd19} the group \ref{eq:homotopy} is isomorphic to $R$ when $n=0$, and zero otherwise. Furthermore, for an inclusion of open subgroups $H \leq K$ of $G$ the restriction map 
$$\pi_n^{K}(\beta(\cM_{\tn{Nis}}(S))) \to \pi_n^{H}(\beta(\cM_{\tn{Nis}}(S)))$$
is given by precomposition with the map of motives that comes from the canonical degree $[K:H]$-map $X/\overl{H} \to X/\overl{K}$, where $X$ is chosen such that the kernel of $G \to \tn{Aut}_{\fets}(X)$ contains $K$ and hence $H$. Again by \cite[Theorem 11.2.14, Proposition 10.2.5]{cd19} in degree $0$ this identifies with the map 
$$R(\pi_0(X/\overl{H})) \to R(\pi_0(X/\overl{K})),$$
which is the identity since both schemes are connected. It follows that there is an abstract equivalence $\beta(\cM_{\tn{Nis}}(S)) \cong \underl{R}$ (note that \cite[Proposition 11.2.5]{cd19} is precisely the statement that the resulting Mackey functor is cohomological \cite[§16]{tw95}). We need to show that it is induced by a map of algebra objects. \\
Consider the unit map $\varphi: \underl{R} \to \beta(\cM_{\tn{Nis}}(S))$. Since both algebra objects live in the heart $\Mod_{\underl{R}}(\specg)^{\heartsuit} \simeq \cmackgr$, the map $\varphi$ is induced by an algebra map in the heart $\varphi^{\heartsuit}: \underl{R}^\heartsuit \to \beta(\cM_{\tn{Nis}}(S))^\heartsuit$. Since both sides are isomorphic to the Mackey functor $\underl{R}$ and the latter is the initial commutative cohomological Green functor, $\varphi^\heartsuit$ is an isomorphism, and $\varphi$ is an equivalence of commutative algebra objects. Hence, we have $\beta(\cM_{\tn{Nis}}(S)) \simeq \underl{R}$ in $\calg(\specg)$. This concludes the proof.
\end{proof}

\begin{rem}
Using the spectral Mackey functor description of $\specg$, we could also produce a left adjoint from the latter to $\mcal{DAM}_{\tn{Nis}}(S;R)$, in fact with less work on finite correspondences. Then the Barr-Beck theorem applies, but the difficulty then lies in identifying the algebra object we obtain by applying the right adjoint to the unit. This identification is simpler when we already deal with an $\underl{R}$-module.
\end{rem}

\begin{rem}
Another way to obtain a functor $\specg \to \mcal{DAM}_{\tn{Nis}}(S;R)$ is constructing a functor $c_S: \specg \to \mcal{SH}^\tn{eff}(S)$ to effective motivic spectra over $S$, and then passing to effective motives. Over a field, such a functor has been constructed by Heller--Ormsby \cite{ho16}, for a much more general construction see \cite[§10.2]{bh21}. \\
In the case of the real and the complex numbers, it is known \cite{lev14,ho16,ho18} that these functors (now with values in all of $\mcal{SH}$) restrict to equivalences
$$c_{\bC}: \spec \stackrel{\sim}{\longrightarrow} \mcal{SH}_{\nA}(\bC) \hspace{1cm} \tn{and} \hspace{1cm} c_{\bR}: \spec^{C_2} \stackrel{\sim}{\longrightarrow}\mcal{SH}_{\nA}(\bR)$$
onto the localising subcategories generated by motivic spectra associated to finite étale algebras for the respective fields. The inverses are given by complex and real Betti realisation, which are induced by the assignment that sends a complex (resp. real) variety to its complex (resp. real) points.
Combining this with the fact that $C_2$-equivariant Betti realisation sends the motivic cohomology spectrum associated to $\bZ$ to the constant Mackey functor $\underl{\bZ}$ \cite[Theorem 4.17]{ho16}, the compatibility of this with restriction \cite[Recollection 1.19]{bhs22} and the fact that motives over a field of characteristic zero are modules over motivic cohomology \cite{ro08} we obtain induced equivalences
$$c_{\bC}: \Mod_{\bZ}(\spec) \stackrel{\sim}{\longrightarrow} \mcal{DAM}_{\tn{Nis}}(\bC;\bZ) \hspace{0.5cm} \tn{and} \hspace{0.5cm} c_{\bR}: \Mod_{\underl{\bZ}}(\spec^{C_2}) \stackrel{\sim}{\longrightarrow} \mcal{DAM}_{\tn{Nis}}(\bR;\bZ).$$
Note that over a field, we do not need to distinguish between effective and non-effective Artin motives \cite[Remark 7.9]{bg23b}.
\end{rem}

\begin{rem}\label{rem:alternativeproof}
One can also show directly that $\Psi: \spans(\fets) \to \corfets{R}$ which we used above induces an equivalence $\permgr \simeq \corfets{R}$ upon factoring out the cohomological ideal, see \cite[Proposition 4.17]{bg23b} and the ideas indicated in \cite[Remark 6.5]{bg23b}. From there one can pursue a similar proof strategy as in op. cit. by showing that since Nisnevich covers of finite étale schemes split, a presheaf with transfers on finite étale schemes is automatically a Nisnevich sheaf. Then one can argue that these sheaves embed fully faithfully into Nisnevich sheaves with transfers on smooth schemes and are not affected by $\bA^1$-localisation.
\end{rem}

Having dealt with the Nisnevich case, we now consider the étale topology, where we will focus on the hypercomplete case. We denote by $\ets$ the small étale site of the base scheme $S$, its underlying category is the full subcategory of $\sms$ on the étale $S$-schemes.

\begin{cons}\label{cons:fettoet}
We write $i: \fets \hookrightarrow \ets$ for the canonical inclusion functor. It is a morphism of sites and hence restriction along $i$ on presheaves preserves hypersheaves. We obtain an adjunction
$$(i_!,i^*): \sheaves^{\wedge}_{\tn{ét}}(\fets;R) \rightleftarrows \sheaves^{\wedge}_{\tn{ét}}(\ets;R),$$
where the left adjoint $i_!$ is given by left Kan extension on presheaves, followed by hypersheafification. As $i$ is symmetric monoidal, left Kan extension along $i$ on presheaves is as well \cite[Proposition 3.6]{bs24}. Since the symmetric monoidal structure on hypersheaves is obtained by localising the Day convolution on presheaves, the hypersheafification functor $\tsf{a}_{\tn{ét}}^{\wedge}$ is symmetric monoidal, and so is $i_!$. \\
When $X$ is a finite étale $S$ scheme, we denote by $R_S(X)$ the associated étale hypersheaf of $R$-modules $\tsf{a}_{\tn{ét}}^{\wedge}(R\otimes \Sigma^{\infty}_+y(X))$ on $\fets$, where $y: \fets \hookrightarrow \cP(\fets)$ is the Yoneda embedding. We use the same notation for the étale sheaf on $\ets$ associated to an étale $S$-scheme.
\end{cons}

For certain schemes, the étale cohomology of lcc sheaves is entirely controlled by the cohomology of their étale fundamental group. This will be of use when passing from the finite étale to the small étale site below.

\begin{rec}[{cf. \cite[§9]{ag16},\;\cite[§2.1.2]{ach15}}]
Let $X$ be a connected qcqs scheme, together with a geometric point $x: \Spec(k) \to X$. Let $\mfr{p}$ be the set of primes invertible on $X$. Then $X$ is a \emph{$K(\pi,1)$-scheme} if for all $\mfr{p}$-torsion locally constant constructible sheaves of abelian groups $\cF$ on $\tn{Ét}_X$, the natural maps
$$H^i(\pi_1^{\tn{ét}}(X,x),\cF_x) \to H^i_{\tn{ét}}(X,\cF)$$
are isomorphisms for all $i \geq 0$. This notion is independent of the choice of the geometric point $x$.
\end{rec}

\begin{ex}
The following schemes are examples of $K(\pi,1)$-schemes \cite[Example 2.1.11]{ach15}: spectra of fields, schemes of cohomological dimension $\leq 1$, smooth connected curves $X$ over a field $\bF$ such that $X_{\bar{\bF}}$ is not isomorphic to $\bP^1_{\bar{\bF}}$, abelian varieties over a field. Furthermore, by \cite[Theorem 1.1.1]{ach15}, any connected affine $\bF_p$-scheme is $K(\pi,1)$ (using the even stronger definition \cite[Definition 4.1.1]{ach17} of the latter).
\end{ex}

\begin{lem}\label{lem:ffkpione}
Assume that $R$ is $n$-torsion and that $S$ is a $K(\pi,1)$-scheme whose residue characteristics are prime to $n$. Then the functor
$$i_!: \sheaves^{\wedge}_{\tn{ét}}(\fets;R) \to \sheaves^{\wedge}_{\tn{ét}}(\ets;R)$$
is fully faithful when restricted to the thick subcategory of $\sheaves^{\wedge}_{\tn{ét}}(\fets;R)$ generated by the $R(X)$ for $X$ finite étale over $S$.
\end{lem}
\begin{proof}
Using the identification of the sheaf categories at play with the derived $\infty$-categories of the corresponding $1$-categories of sheaves of $R$-modules \cite[Corollary 2.1.2.3]{lur18}, this follows immediately from \cite[Corollary 9.20]{ag16} (see also \cite[Proposition 2.1.5]{ach15}).
\end{proof}

It follows from \cite[Corollary 5.3.2]{ach17} that henselian local rings are $K(\pi,1)$. In this case, we even have full faithfulness on the entire category of hypersheaves.

\begin{lem}\label{lem:ffhenselian}
When $S$ is the spectrum of a henselian local ring, then the functor $i_!: \sheaves^{\wedge}_{\tn{ét}}(\fets;R) \to \sheaves^{\wedge}_{\tn{ét}}(\ets;R)$ is fully faithful.
\end{lem}
\begin{proof}
First, note that when $S$ is the spectrum of a henselian local ring the functor $i: \fets \hookrightarrow \ets$ has the covering lifting property \cite[§A.1, Definition 5]{pst23}. This follows from \cite[§1, Theorem 4.2]{mil80} and the fact that a finite étale morphism from a non-empty to a connected scheme is surjective. We can thus invoke \cite[§A.2, Corollary 4]{pst23}, which implies that the right adjoint $i^*$ of $i_!$ commutes with hypersheafification and preserves colimits.\footnote{The cited reference requires covering families to consist of single morphisms. This assumption can be dropped, see the proof of \cite[§A.2, Proposition 6]{pst23}.} \\
To show that $i_!$ is fully faithful, we can now show that the unit map $\tn{id} \to i^*i_!$ is an equivalence. Since both functors involved commute with finite limits and arbitrary colimits, and the objects $R_S(X)$ for $X$ finite étale over $S$ generate $\sheaves^{\wedge}_{\tn{ét}}(\fets;R)$ as a localising subcategory of itself, we reduced to showing $i^*i_!(R_S(X)) \simeq R_S(X)$ for such $X$. This is clear since $i^*$ is restricted from the analogous functor on presheaves and commutes with sheafification.
\end{proof}

By passing to sheaves on smooth $S$-schemes, adding transfers and then localising with respect to $\bA^1$, we obtain a symmetric monoidal left adjoint
\begin{equation}\label{eq:rigidityfunctor}
\rho_!: \sheaves^{\wedge}_{\tn{ét}}(\ets;R) \to \mcal{DM}^{\tn{eff},\wedge}_{\tn{ét}}(S;R).
\end{equation}
For its construction using derived categories see \cite[§3.1]{cd16}, in Appendix \ref{sec:sheaveswithtransfers} we compare our constructions to the classical setup.

\begin{cor}\label{cor:dametassheaves}
Assume that the ring $R$ has positive characteristic $n$ and that the residue characteristics of $S$ are prime to $n$.
\begin{enumerate}
    \item If $S$ is a $K(\pi,1)$-scheme, the composition $\rho_!i_!$ induces a symmetric monoidal equivalence
    $$\sheaves^{\wedge}_{\tn{ét}}(\fets;R) \supseteq \tn{Thick}(R_S(X) \mid X \tn{ finite étale over } S) \simeq \mcal{DAM}^{\wedge}_{\tn{ét},\tn{c}}(S;R).$$
    \item If $S$ is the spectrum of a henselian local ring, $\rho_!i_!$ induces a symmetric monoidal equivalence
    $$\sheaves^{\wedge}_{\tn{ét}}(\fets;R) \simeq \mcal{DAM}^{\wedge}_{\tn{ét}}(S;R).$$
\end{enumerate}
\end{cor}
\begin{proof}
Under our assumptions on base scheme and coefficients, the rigidity theorem in the form of \cite[Theorem 4.5.2]{cd16} implies that the functor $\rho_!$ is an equivalence. In the respective cases (1) and (2) the statement thus follows from \Cref{lem:ffkpione} and \Cref{lem:ffhenselian} respectively, noting that the category of smooth Artin motives (resp. constructible smooth Artin motives) is the localising (resp. thick) subcategory of $\mcal{DAM}^{\wedge}_{\tn{ét}}(S;R)$ generated by the motives $\cM^\wedge_{\tn{ét}}(X)$ for $X$ finite étale over $S$.
\end{proof}

\begin{rem}
Since the generators are dualisable, it is clear that
\begin{equation}\label{eq:thickdualisable}
\tn{Thick}(R_S(X) \mid X \tn{ finite étale over } S) \subseteq \sheaves^{\wedge}_{\tn{ét}}(\fets;R)^{\tn{dbl}}.
\end{equation}
The latter denotes the full subcategory of dualisable objects, which are the locally constant sheaves with perfect values.\footnote{In the context of étale topologies (and originally for $\ell$-adic sheaves), dualisable sheaves are also called \emph{lisse} sheaves.} In terms of the equivalence 
$$\sheaves^{\wedge}_{\tn{ét}}(\fets;R) \simeq \cD(\modgr),$$
the dualisable objects are given by complexes of discrete $(G;R)$-modules that are perfect when considered as a complex of $R$-modules. If $R$ is regular and noetherian, one can use that these complexes are bounded for the t-structure to show that \ref{eq:thickdualisable} is actually an equality. An argument of this form (for sheaves on the small étale site) is given by Ruimy in \cite[Lemma 3.1.10]{rui25a}, using results of \cite{bg23a} on finite resolutions by permutation modules.
\end{rem}

In the non-hypercomplete case, we can analogously construct a functor
$$\sheaves_{\tn{ét}}(\fets;R) \to \mcal{DM}^{\tn{eff}}_{\tn{ét}}(S;R)$$
which lands in $\mcal{DAM}_{\tn{ét}}(S;R)$, but we will not be able to say that the composition is fully faithful since here no statements about cohomology groups are available.\\
In the Nisnevich case, \Cref{rem:alternativeproof} asserts that a presheaf with transfers on finite étale schemes is automatically a sheaf. In the étale case the situation is dual: an étale sheaf on finite étale schemes automatically has transfers. In order to show this, we need a preliminary construction that refines the pullback diagram of \Cref{defi:sheaveswithtransfers}.

\begin{cons}\label{cons:refinedpullback}
The functor $\gamma: \fets \to \corfets{R}$ factors as
$$\fets \stackrel{\rho}{\longrightarrow} \fets \times \lat_R \stackrel{\gamma \times \sigma}{\longrightarrow} \corfets{R} \times \corfets{R} \stackrel{\times_S}{\longrightarrow} \corfets{R}.$$
Here $\lat_R$ is the category of finite rank free $R$-modules, and by \cite[Remark D.1.1.5]{lur18} we have $\Mod_R(\spec) \simeq \psigma(\lat_R;\spec)$. The functor $\rho$ is the identity on the fist factor and constant with value $R$ on the second, $\sigma$ is the canonical additive functor $\lat_R \to \corfets{R}$ that sends $R$ to $S$, and the bifunctor $(-)\times_S(-)$ comes from the monoidal structure on $\corfets{R}$. We denote the composition functor $\fets \times \lat_R \to \corfets{R}$ by $\delta$. Using $\psigma(\fets;R) \simeq \fun^{\tn{bi-}\times}(\fets^\op \times \lat_R^\op,\spec)$, the pullback square of \Cref{defi:sheaveswithtransfers} refines to the iterated pullback
\begin{equation}\label{diag:refinedpullback}
\begin{tikzcd}
	{\sheaves^{(\wedge)}_{\tn{ét}}(\corfets{R};\spec)} & {\psigma(\corfets{R};\spec)} \\
	{\sheaves^{(\wedge)}_{\tn{ét}}(\fets;R)} & {\psigma(\fets;R)} \\
	{\sheaves^{(\wedge)}_{\tn{ét}}(\fets;\spec)} & {\psigma(\fets;\spec).}
	\arrow["{\tsf{o}^{(\wedge),\tn{tr}}_{\tn{ét}}}", hook, from=1-1, to=1-2]
	\arrow["{\delta^*}", from=1-1, to=2-1]
	\arrow["{\hat{\delta}^*}", from=1-2, to=2-2]
	\arrow["{\tsf{o}_{\tn{ét}}^{(\wedge)}}", hook, from=2-1, to=2-2]
	\arrow["{\rho^*}", from=2-1, to=3-1]
	\arrow["{\hat{\rho}^*}", from=2-2, to=3-2]
	\arrow["{\tsf{o}_{\tn{ét}}^{(\wedge)}}", hook, from=3-1, to=3-2]
    \arrow["\usebox\pullback"{anchor=center, pos=0.125}, draw=none, from=1-1, to=2-2]
    \arrow["\usebox\pullback"{anchor=center, pos=0.125}, draw=none, from=2-1, to=3-2]
\end{tikzcd}
\end{equation}
As before for $\gamma$, the presheaf versions of the restriction along $\rho$ and $\delta$ are denoted $\hat{\rho}^*$ and $\hat{\delta}^*$, and the sheaf versions $\rho^*$ and $\delta^*$.
\end{cons}

\begin{thm}\label{thm:etaletransfers}
The functor $\delta^*: \sheaves_{\tn{ét}}^{(\wedge)}(\corfets{R};\spec) \to \sheaves_{\tn{ét}}^{(\wedge)}(\fets;R)$ is an equivalence.
\end{thm}
\begin{proof}
We first consider the case of sheaves. Let $S_i$ be the finite étale $S$-scheme corresponding to the orbit $G/N_i$, so that the canonical map $G \to \tn{Aut}_{\fets}(S_i) $ has kernel given by $N_i$. Write $\fetslice{S_i}$ for the slice $(\fets)_{S_i/}$. Then as for $G$-sets \cite[Construction 4.5]{cm21} the site $\fets$ decomposes into a filtered colimit of sites $\tn{colim}_i \, \fetslice{S_i}$ (each of which has covering sieves given by jointly surjective families). The transition maps are the natural maps $\fetslice{S_i} \to \fetslice{S_j}$ that precompose with the map $S_i \to S_j$ whenever there is an inclusion $N_i \subseteq N_j$. \\
Note that each of the canonical maps $\fetslice{S_i} \to \fets$ is fully faithful. We write $\corfetsslice{R}{S_i}$ for the full subcategory of $\corfets{R}$ on the schemes in $\fetslice{S_i}$. Since each of the maps of sites descends to these correspondence categories (see \Cref{rem:nicecorrespondences}), we obtain a similar colimit decomposition 
$$\corfets{R} \simeq \tn{colim}_i \, \corfetsslice{R}{S_i},$$
compatible with the functors $\delta_i: \fetslice{S_i} \times \lat_R \to \corfetsslice{R}{S_i}$ that are restricted from the functor $\delta$ of \Cref{rec:graphfunctors}. Using the upper pullback of \ref{diag:refinedpullback}, the decomposition of the (pre)sheaf categories involved induces one for sheaves of transfers as a limit along restriction:
$$\sheaves_{\tn{ét}}(\corfets{R};\spec) \simeq \tn{lim}_i \; \sheaves_{\tn{ét}}(\corfetsslice{R}{S_i};\spec).$$
So it suffices to show that each $\delta_i^*: \sheaves_{\tn{ét}}(\corfetsslice{R}{S_i};\spec) \to \sheaves_{\tn{ét}}(\fetslice{S_i};R)$ which precomposes with the map $\delta_i$ from above is an equivalence. For the rest of the proof we will assume that $G$ is finite. The argument we give then holds verbatim for each of the $\delta_i$. \\
Note that $\lat_R$ is the coproduct completion of its full subcategory $\{R \} \subseteq \lat_R$ on the object $R$, i.e. it is the smallest full subcategory of $\cP(\{R \})$ that contains the object $R$ and is closed under finite coproducts. Hence \cite[Proposition 5.3.6.2]{lur09} implies that restriction and right Kan extension provide an equivalence
\begin{equation}\label{eq:latr}
\Mod_R(\spec) \simeq \fun^{\times}(\lat_R^\op,\spec) \simeq \fun(\{R\}^{\op},\spec).
\end{equation}
For the moment, we will work with the description on the right-hand side of \ref{eq:latr}. Let $\beta: BG \times \{R \} \to \fets \times \{R \}$ be the functor that on the first component picks out the scheme $\tilde{S}$ corresponding to the orbit $G/1$ and is given by the identity on the other component. Write $\alpha: BG \times \{R \} \to \corfets{R}$ for its composition with the functor $\delta: \tsf{Fét}_{S} \times \{R\} \to \tsf{FétCor}(S;R)$. Restriction and right Kan extension along $\alpha$ provide an adjunction 
$$(\hat{\alpha}^*,\hat{\alpha}_*): \cP(\corfets{R};\spec) \rightleftarrows \cP(BG \times \{R \};\spec)$$
which factors as the composition of adjunctions
\[\begin{tikzcd}
	{\cP(\corfets{R};\spec)} & {\cP(\fets \times \{R \};\spec)} & {\cP(BG \times \{R \};\spec).}
	\arrow["{(\hat{\delta}^*,\hat{\delta}_*)}", shift left, from=1-1, to=1-2]
	\arrow[shift left, from=1-2, to=1-1]
	\arrow["{(\hat{\beta}^*,\hat{\beta}_*)}", shift left, from=1-2, to=1-3]
	\arrow[shift left, from=1-3, to=1-2]
\end{tikzcd}\]
By \cite[Example 4.4]{cm21} (which we already used in \Cref{lem:sheavesfunbg}, up to self-duality of $BG$), the functor $\hat{\beta}_*$ is fully faithful and has essential image those additive presheaves that are étale sheaves on the first component (i.e., presheaves corresponding to elements in $\sheaves_{\tn{ét}}(\fets;R) \subseteq \psigma(\fets;R)$ under the equivalence \ref{eq:latr}). We denote the resulting equivalence functor by 
$$\beta_*: \cP(BG \times \{R \};\spec) \stackrel{\sim}{\longrightarrow} \sheaves_{\tn{ét}}(\fets \times \{R \};\spec).$$
The functor $\alpha$ sends the unique object $(\ast_{BG},R)$ of its domain to $\tilde{S}$. By \Cref{rem:nicecorrespondences} the endomorphisms of these two objects are isomorphic to the free module $R(G)$, and $\alpha$ sends a morphism $(g,r)$ to $r \cdot \langle \Gamma_{{\tn{id}}^g}\rangle$, where ${\tn{id}}^g: \tilde{S} \to \tilde{S}$ corresponds to conjugation by $g$. Under the description given in \Cref{rem:nicecorrespondences} this corresponds to the element $r\cdot g$ of $R(G)$ and it follows that $\alpha$ is fully faithful. Hence the right Kan extension functor $\hat{\alpha}_*$ is fully faithful as well. \\
We now show that $\hat{\alpha}_*$ lands in sheaves with transfers. Let $\cF: BG^\op \times \{R \}^\op \to \spec$ be a functor. For a subgroup $H \leq G$, the evaluation of the right Kan extension $(\hat{\alpha}_*\cF)(\tilde{S}/H)$ is computed as the limit of the restriction of $\cF$ to
$$\cQ \defeq (BG^\op \times \{R \}^\op) \times_{\corfets{R}^\op} (\corfets{R}^\op)_{(\tilde{S}/H)\tn{\big /}}.$$
Now, the functor $\delta$ induces an isomorphism of $R$-modules
$$\Hom_{\fets \times \lat_R}((\tilde{S},R),(\tilde{S}/H,R)) \to \tsf{Cor}_S(\tilde{S},\tilde{S}/H)\otimes R,$$
here both sides are isomorphic to $R(G/H)$. It follows that $\delta$ induces an equivalence of $\cQ$ to the diagram
$$(BG^\op \times \{R \}^\op) \times_{\fets^\op \times \{R \}^\op} (\fets^\op \times \{R \}^\op)_{(\tilde{S}/H,R)\tn{\big /}},$$
which computes $\beta_*\cF(\tilde{S}/H) \simeq \cF(\ast_{BG},R)^{hH}$ (see also the proof of \cite[Proposition 7.3.1.16]{kerodon}). So we have $\hat{\delta}^*\hat{\alpha}_*(\cF) \simeq \hat{\beta}_*(\cF)$ and $\hat{\delta}^*\hat{\alpha}_*(\cF)$ is right Kan extended from $BG^\op \times \{R\}^\op$, and thus a sheaf and in particular product preserving. Hence $\hat{\alpha}_*(\cF)$ is a sheaf with transfers. We denote the resulting adjunction by
$$(\alpha^*,\alpha_*): \sheaves_{\tn{ét}}(\corfets{R};\spec) \rightleftarrows \fun(BG \times \{R \}^\op,\spec).$$
Now, for $\cG \in \sheaves_{\tn{ét}}(\fets \times \{R \};\spec)$ there exists $\cF \in \cP(BG \times \{R \};\spec)$ such that $\beta_*(\cF) \simeq \cG$. Since right Kan extensions compose we have $\hat{\delta}_*(\cG) \simeq \hat{\delta}_*\beta_*(\cF) \simeq \alpha_*(\cF)$ which is a sheaf with transfers. Hence $\hat{\delta}_*$ sends sheaves to sheaves with transfers and we denote the resulting functor by 
$$\delta_*: \sheaves_{\tn{ét}}(\fets \times \{R\};\spec) \to \sheaves_{\tn{ét}}(\corfets{R};\spec),$$
it is right adjoint to restriction along $\delta$. Now, since $\alpha_*$ and $\beta_*$ are fully faithful, $\delta_*$ is fully faithful as well. Since $\delta$ is essentially surjective $\hat{\delta}^*$ is conservative, so since $\tsf{o}_{\tn{ét}}^{\tn{tr}}$ is fully faithful $\delta^*$ is conservative as well. We now have an adjunction $\delta^* \dashv \delta_*$ for which the right adjoint is fully faithful and the left adjoint is conservative, so this adjunction is an equivalence. Since
$$\sheaves_{\tn{ét}}(\fets \times \{R\};\spec) \simeq \sheaves_{\tn{ét}}(\fets;R),$$
this finishes the argument for sheaves. The statement for hypersheaves follows from the one for sheaves since the hypersheaf version 
$$\delta^*: \sheaves_{\tn{ét}}^{\wedge}(\corfets{R};\spec) \to \sheaves_{\tn{ét}}^{\wedge}(\fets;R)$$
is the pullback of the functor on sheaves, as in \Cref{rem:tstructuresheaveswithtransfers}. Hence it is an equivalence as well, which concludes the proof.
\end{proof}

\begin{rem}
In the hypercomplete case, the statement of the previous proposition is also implied by \cite[Corollary 2.1.9]{cd16}, which relies on a concrete identification of the effect of the graph functor on `representable' sheaves and sheaves with transfers. We also refer to \cite[Corollary 4.39]{cm21} for some form of converse for sheaves on the small étale site, and to \cite[Corollary C.13]{bh21} for a similar statement using $\infty$-categories of spans. \\
Morally speaking, étale sheaves on finite $G$-sets already have transfers because they are right Kan extended from $BG$ (at least when $G$ is finite) and thus have an overlying Borel complete $G$-spectrum. \\
One can make this line of thoughts precise by explicitly identifying the correspondence category $\corfets{R}$ with $\permgr$, using that by \cite[Corollary 3.17]{fuh25} spectral presheaves on the latter are the same as modules over the constant $R$-linear Mackey functor $\underl{R}$ in $\specg$, and then using that the Borel complete subcategory only sees the restriction of $\underl{R}$ to the trivial group. 
\end{rem}

In the following theorem, we will write $\cD\Perm(G;R)$ for the derived $\infty$-category of $(G;R)$-permutation modules \cite{bg22a}, \cite[§3]{bg23b}, \cite[§4]{fuh25}. It combines the results of the previous sections.

\begin{thm}\label{thm:bigdiagram}
There is a commutative diagram in which horizontal functors are equivalences and vertical functors are localisations.
\[\begin{tikzcd}
	{\cD\Perm(G;R)} & {\Mod_{\underl{R}}(\specg)} & {\sheaves_{\tn{Nis}}(\corfets{R};\spec)}  \\
	{\tn{colim}_i \, \cD(\Mod(G_i;R))} & {\Mod_{R_G}(\specg)_{\tn{lwBorel}}} & {\sheaves_{\tn{ét}}(\corfets{R};\spec)}  \\
	{\cD(\Mod(G;R))} & {\Mod_{R_G}(\specg)_{\tn{Borel}}} & {\sheaves_{\tn{ét}}^{\wedge}(\corfets{R};\spec).} 
	\arrow["\sim", from=1-1, to=1-2]
	\arrow["\alpha", from=1-1, to=2-1]
	\arrow["\sim", from=1-2, to=1-3]
	\arrow["\beta", from=1-2, to=2-2]
	\arrow["\tsf{a}_{\tn{ét}}", from=1-3, to=2-3]
	\arrow["\sim", from=2-1, to=2-2]
	\arrow["L^h", from=2-1, to=3-1]
	\arrow["\sim", from=2-2, to=2-3]
	\arrow["L^h", from=2-2, to=3-2]
	\arrow["(-)^{\wedge}_\tn{ét}", from=2-3, to=3-3]
	\arrow["\sim", from=3-1, to=3-2]
	\arrow["\sim", from=3-2, to=3-3]
\end{tikzcd}\]
If the virtual cohomological dimension of the group $G$ is finite, all vertical functors from the second to the third row are equivalences.
\end{thm}
\begin{proof}
The content of the statement for the lower two rows was discussed already: The commutativity of the two bottom squares was shown in \Cref{prop:sheavesborelcomplete}, \Cref{prop:sheavesgrmodules} and \Cref{thm:etaletransfers}. The assertion about the virtual cohomological dimension of $G$ is \Cref{prop:smashinghypercompletion} and \Cref{cor:smashinghypercompletionreps}. Let us consider the upper part of the diagram. \\
The equivalences in the top row hold by \cite[Corollary 4.8, Corollary 3.18]{fuh25}, \cite[Example 2.15]{bcn25} and \Cref{thm:damnisasmodunderlr} - using that Nisnevich covers in $\fets$ split the left adjoint in the latter theorem factors through $\sheaves_{\tn{Nis}}(\corfets{R};\spec)$, hence is fully faithful, and it clearly hits generators. These equivalences are compatible with the filtered colimit decompositions of the three $\infty$-categories involved we will use below. \\
By passing to Ind-objects, \cite[Proposition 2.4]{bg25b} implies that the $\infty$-category $\cD\Perm(G;R)$ decomposes as a filtered colimit in $\prl$ along inflation. Similarly, $\tn{colim}_i \, \cD(\Mod(G_i;R))$ by definition is a $\prl$-colimit along inflation. The functor $\alpha$ for a finite quotient $G_i$ is given as the Ind-extension of the functor
$$\cK_{\tn{b}}(\tsf{perm}(G_i;R)^\natural) \to \cD_{\tn{b}}(\tsf{mod}(G_i;R)) \to \cD(\tsf{Mod}(G_i;R))$$
that is induced by the inclusion $\tsf{perm}(G_i;R)^\natural \hookrightarrow \tsf{mod}(G_i;R)$, where the latter category denotes finitely generated $(G_i;R)$-modules.\\
Using (the profinite extension of) \cite[§5.2]{gs14}, there is a ring map $R_G \to \underl{R}$ which is determined by being the unit map $\bA_R \to \underl{R}$ in $R$-linear Mackey functors on $\pi_0$. The levelwise completion functor of \Cref{sec:borelprofinite} induces a functor
\[\begin{tikzcd}
	{\beta: \, \Mod_{\underl{R}}(\specg) \simeq \Mod_{\underl{R}}\Mod_{R_G}(\specg)} & {\Mod_{L_{\tn{lw}(\underl{R})}}(\Mod_{R_G}(\specg)_{\tn{lwBorel}}),}
	\arrow["{L_{\tn{lw}}}", from=1-1, to=1-2]
\end{tikzcd}\]
on module categories, where the first equivalence is \cite[Corollary 3.4.1.9]{lur17}. By \cite[Construction 3.8, p.39]{fuh25}, we can write $\Mod_{\underl{R}}(\specg)$ as a $\prl$-colimit along the functors induced by inflation on module categories, followed by an extension of scalars along a ring map $\tn{infl}_{G_j}^{G_i}(\underl{R})\to \underl{R}$ that is determined by being the identity on the sections for the trivial group on $\pi_0$. Under this description, the functor $\beta$ identifies with the one induced on colimits by the functors
$$\beta_i: \Mod_{\underl{R}}\Mod_{R_{G_i}}(\spec^{G_i}) \to \Mod_{L_{G_i;R_{G_i}}(\underl{R})}\left(\Mod_{R_{G_i}}(\spec^{G_i})_{\tn{Borel}}\right).$$
Each $\beta_i$ is the functor induced by $L_{G_i;R_{G_i}}: \Mod_{R_{G_i}}(\spec^{G_i}) \to \Mod_{R_{G_i}}(\spec^{G_i})_{\tn{Borel}}$ on $\underl{R}$-modules. But since the map $R_{G_i} \to \underl{R}$ becomes an equivalence when restricting to the trivial group, $L_{G_i;R_{G_i}}(\underl{R})$ identifies with the unit of $\Mod_{R_{G_i}}(\spec^{G_i})_{\tn{Borel}}$, and $\beta$ lands in $\Mod_{R_G}(\specg)_{\tn{lwBorel}}$.\\
The functor $\tsf{a}_{\tn{ét}}^{\tn{tr}}: \sheaves_{\tn{Nis}}(\corfets{R};\spec) \to \sheaves_{\tn{ét}}(\corfets{R};\spec)$ is the restriction of the étale sheafification functor to Nisnevich sheaves with transfers (which by \Cref{rem:alternativeproof} are just presheaves with transfers), see \Cref{rem:pushoutinsteadofpullback}. \\
We now show that the two upper squares commute. Using the compatibility of the above functors with colimit decompositions, it suffices to check commutativity for each finite quotient $G/N_i=G_i$. By the universal property of the $\infty$-category in the upper row \cite[Corollary 3.17]{fuh25} it suffices to establish commutativity of the restrictions to $\perm(G_i;R)$. The latter is the coproduct completion of its full subcategory on the permutation modules $R(G_i/H)$ for $H\leq G$, and we can furthermore restrict to this subcategory. In this case, there is a commutative diagram
\[\begin{tikzcd}
	{\{R(G_i/H) \mid H \leq G_i\}} & {\{\underl{R} \otimes \Sigma^{\infty}_+G_i/H \mid H \leq G_i\}} & {\{R^{\tn{tr}}_S(\tilde{S}_i/H) \mid H \leq G_i\}} \\
	{\{R(G_i/H) \mid H \leq G_i\}} & {\{R[G/H_i] \mid H \leq G_i\}} & {\{R^{\tn{tr}}_S(\tilde{S}_i/H) \mid H \leq G_i\}}
	\arrow["\sim", from=1-1, to=1-2]
	\arrow["\alpha", from=1-1, to=2-1]
	\arrow["\sim", from=1-2, to=1-3]
	\arrow["\beta", from=1-2, to=2-2]
	\arrow["{\tsf{a}^{\tn{tr}}_{\tn{ét}}}", from=1-3, to=2-3]
	\arrow["\sim", from=2-1, to=2-2]
	\arrow["\sim", from=2-2, to=2-3]
\end{tikzcd}\]
in which every node indicates the full subcategory of the respective node in the upper two rows of the diagram in the statement. The centre of the bottom row we view as a full subcategory of $\fun(BG_i,\Mod_R(\spec))$, and under this description the functor $\beta$ evaluates a spectral presheaf on $\perm(G_i;R)$ on $R(G_i)$. For the two categories on the right, $\tilde{S}_i$ is an $S$-scheme corresponding to the $G$-orbit $G/N_i$, i.e. the kernel of the canonical homomorphism $\varphi: G \to \tn{Aut}_{\fets}(\tilde{S}_i)$ is $N_i$. Note that by \Cref{rem:representables} the elements $R^{\tn{tr}}_S(\tilde{S}_i/H)$ are already étale sheaves with transfers. For the behaviour of the bottom right horizontal functor we refer to the proof of \Cref{thm:etaletransfers} - its inverse evaluates on $\tilde{S}_i$. This concludes the proof.
\end{proof}

Recall the functor $\iota_!: \sheaves_{\tau}^{(\wedge)}(\corfets{R};\spec) \to \sheaves_{\tau}^{(\wedge)}(\corsms{R};\spec)$ from \Cref{cons:fettoet}. Postcomposing it with the $\bA^1$-localisation $L_{\bA^1}$ yields a functor to $\mcal{DAM}^{(\wedge)}_{\tau}(S;R)$. In the hypercomplete étale case, the composition $L_{\bA^1}\iota_!$ identifies with the composition $\rho_!i_!$ of \ref{eq:rigidityfunctor} and \Cref{cor:dametassheaves}, under the equivalence of \Cref{thm:etaletransfers}. We now amend the diagram of the previous theorem to the right.

\begin{thm}\label{thm:bigdiagramright}
There is a commutative diagram
\[\begin{tikzcd}
	{\sheaves_{\tn{Nis}}(\corfets{R};\spec)} & {\mcal{DAM}_{\tn{Nis}}(S;R)} & {\mcal{DM}^{\tn{eff}}_{\tn{Nis}}(S;R)} \\
	{\sheaves_{\tn{ét}}(\corfets{R};\spec)} & {\mcal{DAM}_{\tn{ét}}(S;R)} & {\mcal{DM}^{\tn{eff}}_{\tn{ét}}(S;R)}\\
	{\sheaves_{\tn{ét}}^{\wedge}(\corfets{R};\spec)} & {\mcal{DAM}^{\wedge}_{\tn{ét}}(S;R)} & {\mcal{DM}^{\tn{eff},\wedge}_{\tn{ét}}(S;R).}
	\arrow["{L_{\bA^1}\iota_!}", from=1-1, to=1-2]
	\arrow["\sim"', draw=none, from=1-1, to=1-2]
	\arrow["{\tsf{a}^{\tn{tr}}_{\tn{ét}}}", from=1-1, to=2-1]
	\arrow["{\tsf{a}^{\tn{tr}}_{\tn{ét}}}", from=1-2, to=2-2]
    \arrow["{\tsf{a}^{\tn{tr}}_{\tn{ét}}}", from=1-3, to=2-3]
	\arrow["{L_{\bA^1}\iota_!}", from=2-1, to=2-2]
    \arrow["(\ast)"', draw=none, from=3-1, to=3-2]
	\arrow["{(-)^\wedge_\tn{ét}}", from=2-1, to=3-1]
    \arrow["(\ast\ast)"', draw=none, from=2-1, to=3-1]
	\arrow["{(-)^\wedge_\tn{ét}}", from=2-2, to=3-2]
    \arrow["{(-)^\wedge_\tn{ét}}", from=2-3, to=3-3]
	\arrow["{L_{\bA^1}\iota_!}", from=3-1, to=3-2]
    \arrow[hook, from=1-2, to=1-3]
    \arrow[hook, from=2-2, to=2-3]
    \arrow[hook, from=3-2, to=3-3]
\end{tikzcd}\]
in which the top left horizontal arrow is an equivalence. We furthermore have the following implications:
\begin{enumerate}
    \item Assume that $R$ has positive characteristic $n$ and that the residue characteristics of $S$ are prime to $n$. If $S$ is the spectrum of a henselian local ring, the functor labelled $(\ast)$ is an equivalence. If $S$ is $K(\pi,1)$, the functor $(\ast)$ restricts to an equivalence of thick subcategories generated by the sheaves and motives associated to finite étale $S$-schemes.
    \item If the virtual cohomological dimension of $G$ is finite, then the functor labelled $(\ast\ast)$ is an equivalence.
    \item If $R$ is a $\bQ$-algebra, all vertical functors and hence the horizontal functors on the left side of the diagram are equivalences.
\end{enumerate}
\end{thm}
\begin{proof}
The functors on Artin motives and effective motives are induced by étale sheafification and hypercompletion on smooth $S$-schemes, we denote them by the same symbols as on sheaves with transfers. Commutativity of the diagram is clear by construction. The fact that the top horizontal map is an equivalence follows from \Cref{thm:damnisasmodunderlr}, see also the proof of \Cref{thm:bigdiagram}. Assertion (1) was shown in  \Cref{cor:dametassheaves}, (2) was recorded in \Cref{thm:bigdiagram} using \Cref{prop:smashinghypercompletion} and \Cref{cor:smashinghypercompletionreps}, and (3) is implied by \cite[Proposition 2.2.10]{cd16}.
\end{proof}

\appendix
\section{Sheaves with transfers}\label{sec:sheaveswithtransfers}

Using the language of triangulated categories, (effective) Voevodsky motives (for the Nisnevich topology or more generally for other topologies on smooth schemes) are classically defined by deriving the $1$-category of Nisnevich sheaves with transfers and then localising with respect to $\bA^1$ \cite[§11]{cd19}. On the other hand, definitions using $\infty$-(pre)sheaves are present in the literature \cite{ek20,bh21}. In this appendix we show that our definition of sheaves with transfers (that is, before $\bA^1$-localisation) agrees with the more classical one that uses derived categories, thereby also justifying that we use classical computations of mapping spaces in effective motives. For this, we will focus on hypersheaves, which on the categorical level correspond to derived $\infty$-categories of $1$-categorical sheaves. We keep all standing assumptions from the previous section. \\
We first recall a few facts about the $\infty$-category of additive presheaves with values in an $\infty$-category which for our purposes will always be spectra or an $\infty$-category of modules in spectra.

\begin{rec}\label{rec:psigma}
Let $\cC$ be an $\infty$-category which admits finite coproducts. Let $\psigma(\cC) \subseteq \cP(\cC) = \fun(\cC^{\op},\spc)$ be the full subcategory on presheaves that preserve finite products. Then $\psigma(\cC)$ is compactly generated \cite[Proposition 5.5.8.10(6)]{lur09} and an accessible localisation of $\cP(\cC)$ \cite[Proposition 5.5.8.10(1)]{lur09}, we write 
$$L_{\Sigma}: \cP(\cC) \to \psigma(\cC)$$
for the left adjoint localisation functor. The Yoneda embedding $y: \cC \hookrightarrow \cP(\cC)$ factors through $\psigma(\cC)$ \cite[Proposition 5.5.8.10(2)]{lur09}. \\
Let $\cC$ be additively symmetric monoidal\footnote{That is, $\cC$ is additive and admits a symmetric monoidal structure for which the tensor product is additive in both variables.} and let $\cD$ be a presentably symmetric monoidal stable $\infty$-category. Then the presheaf $\infty$-categories $\cP(\cC)$ and $\cP(\cC;\cD)$ become symmetric monoidal under Day convolution \cite[§2.2.6, 4.8.1]{lur17}. These structures make the map $\cP(\cC) \to \cP(\cC;\cD)$ induced by the unit $\spc \to \cD$ in $\calg(\prl)$ symmetric monoidal, see \cite[Corollary 3.7]{nik16} or \cite[Proposition 3.3]{bs24}. \\
As in \cite[Lemma 3.7]{bgs20} the tensor product on $\cP(\cC)$ preserves $L_{\Sigma}$-equivalences in both variables since it preserves colimits in each variables, and $\psigma(\cC)$ by \cite[Proposition 2.2.1.9]{lur17} inherits a symmetric monoidal structure from $\cP(\cC)$ which informally is given by performing a Day convolution in $\cP(\cC)$ and then applying the localisation functor $L_{\Sigma}$. This makes $L_{\Sigma}: \cP(\cC) \to \psigma(\cC)$ and $y: \cC \hookrightarrow \psigma(\cC)$ symmetric monoidal.\\
If $\cC$ carries a Grothendieck topology $\tau$ such that the tensor product on $\cP(\cC;\cD)$ preserves $a^{(\wedge)}_{\tau}$-equivalences, then tensoring in $\cP(\cC;\cD)$ followed by (hyper)sheafification endows $\sheaves^{(\wedge)}_{\tau}(\cC; \cD)$ with a symmetric monoidal structure which makes it presentably symmetric monoidal. \\
Analogously, the full subcategory $\psigma(\cC;\cD) \subseteq \cP(\cC;\cD)$ on $\cD$-valued presheaves which preserve finite products is a symmetric monoidal localisation of $\cP(\cC;\cD)$, cf. \cite[Remark 2.10]{aok20}. When $\cD = \Mod_R(\spec)$ we write $\psigma(\cC;R)$. By \cite[Corollary 2.10(iii), Example 5.3(ii)]{ggn15} or \cite[Proposition C.1.5.7]{lur18} there is a symmetric monoidal equivalence $\psigma(\cC) \simeq \psigma(\cC;\spec_{\geq 0})$ which is induced by the functor $\Omega^{\infty}: \spec_{\geq 0} \to \spc$. Together with the inclusion of connective spectra into spectra this induces a symmetric monoidal and fully faithful `stable Yoneda embedding'
$$\bar{y}: \cC \hookrightarrow \psigma(\cC) \simeq \psigma(\cC;\spec_{\geq 0}) \hookrightarrow \psigma(\cC;\spec),$$
as in \cite[Corollary 2.9]{aok20}. For $\cD$ presentably symmetric monoidal and stable the essentially unique symmetric monoidal left adjoint $\spec \to \cD$ again induces a symmetric monoidal functor $\psigma(\cC;\spec) \to \psigma(\cC;\cD)$.
\end{rec}

Recall that \Cref{defi:sheaveswithtransfers} defined (hyper)sheaves with transfers as the pullback
\begin{equation}\label{diag:pullbacksheaveswithtransfers}
\begin{tikzcd}
	{\sheaves^{(\wedge)}_{\tau}(\corsms{R};\spec)} & {\psigma(\corsms{R};\spec)} \\
	{\sheaves^{(\wedge)}_{\tau}(\sms;\spec)} & {\psigma(\sms;\spec).}
	\arrow[from=1-1, to=1-2, "{\tsf{o}^{(\wedge),\tn{tr}}_{\tau}}", hook]
	\arrow[from=1-1, to=2-1, "\gamma^*"]
	\arrow[from=1-2, to=2-2, "\hat{\gamma}^*"]
	\arrow[from=2-1, to=2-2, "{\tsf{o}_{\tau}^{(\wedge)}}", hook]
    \arrow["\usebox\pullback"{anchor=center, pos=0.125}, draw=none, from=1-1, to=2-2]
\end{tikzcd}
\end{equation}

The following two remarks are formulated for smooth $S$-schemes, but they hold verbatim after replacing smooth by finite étale schemes. As in \Cref{sec:motives}, $\tau$ stands for either the Nisnevich or the étale topology.

\begin{rem}\label{rem:pushoutinsteadofpullback}
By \Cref{rec:psigma} the functor $\tsf{o}_{\tau}^{(\wedge)}: \sheaves_{\tau}^{(\wedge)}(\sms;\spec) \hookrightarrow \cP(\sms;\spec)$ has a left adjoint (hyper)sheafification functor 
$$\tsf{a}^{(\wedge)}_{\tau}: \cP(\sms;\spec) \to \sheaves_{\tau}^{(\wedge)}(\sms;\spec).$$ 
The forgetful functor $\hat{\gamma}^*: \psigma(\corsms{R};\spec) \to \psigma(\sms;\spec)$ has a left adjoint given by left Kan extension 
$$\hat{\gamma}_!: \psigma(\sms;\spec) \to \psigma(\corsms{R};\spec)$$
along $\gamma: \sms \to \corsms{R}$, which restricts to $\psigma$ by \cite[Proposition 3.3.2]{chll24}. This means that \ref{diag:pullbacksheaveswithtransfers} is a pullback in $\prrst$, and we denote the left adjoints to $\gamma^*$ and $\tsf{o}^{(\wedge),\tn{tr}}_{\tau}$ by $\gamma_!$ and $\tsf{a}_{\tau}^{(\wedge), \tn{tr}}$ respectively. It follows formally that $\gamma_! \simeq \tsf{a}_{\tau}^{(\wedge), \tn{tr}} \circ \hat{\gamma}_!$. We can hence express $\sheaves_{\tau}^{(\wedge)}(\corsms{R};\spec)$ as the $\prlst$-pushout
\[\begin{tikzcd}
	{\sheaves_{\tau}^{(\wedge)}(\corsms{R};\spec)} & {\psigma(\corsms{R};\spec)} \\
	{\sheaves_{\tau}^{(\wedge)}(\sms;\spec)} & {\psigma(\sms;\spec).}
	\arrow[from=1-2, to=1-1, "{\tsf{a}_{\tau}^{(\wedge), \tn{tr}}}"']
	\arrow[from=2-1, to=1-1, "\gamma_!"']
	\arrow[from=2-2, to=1-2, "\hat{\gamma}_!"']
	\arrow[from=2-2, to=2-1, "{\tsf{a}^{(\wedge)}_{\tau}}"']
    \arrow["\usebox\pullback"{anchor=center, pos=0.125}, draw=none, from=1-1, to=2-2]
\end{tikzcd}\]
Furthermore, both functors $\hat{\gamma}_!$ and $\tsf{a}^{(\wedge)}_{\tau}$ admit symmetric monoidal enhancements (using \Cref{rec:psigma} and \cite[Proposition 3.6]{bs24}), and since $\tsf{a}^{(\wedge)}_{\tau}$ is a localisation this pushout can also be computed in $\calg(\prlst)$. The resulting symmetric monoidal structure on $\sheaves_{\tau}^{(\wedge)}(\corsms{R};\spec)$ agrees with the one inherited from $\psigma(\corsms{R};\spec)$ as a symmetric monoidal localisation, in the sense of \cite[Proposition 2.2.1.9]{lur17}.
\end{rem}

\begin{rem}\label{rem:tstrucuresonsheaveswithtransfers}
The stable $\infty$-categories defining the pullback square in \Cref{defi:sheaveswithtransfers} carry t-structures. For $\psigma(\sms;\spec)$ and $\psigma(\corsms{R};\spec)$ both connectivity and coconnectivity are detected levelwise, see e.g. \cite[Definition 2.5]{bcn25}. This makes the restriction 
$$\hat{\gamma}^*: \psigma(\corsms{R};\spec) \to \psigma(\sms;\spec)$$
into a t-exact functor. For $\sheaves_{\tau}^{(\wedge)}(\sms;\spec)$ we already discussed the t-structure in \Cref{rem:tstructureonsheaves}, it makes the (hyper)sheafification
$$\tsf{a}^{(\wedge)}_{\tau}: \psigma(\sms;\spec) \to \sheaves_{\tau}^{(\wedge)}(\sms;\spec)$$
into a t-exact functor. There is an induced t-structure on (hyper)sheaves with transfers for which coconnectivity is detected on the level of presheaves with transfers, i.e. levelwise, and the connective objects are determined by orthogonality. The hearts of these t-structures are the abelian categories of (pre)sheaves (with transfers) with values in abelian groups.

\end{rem}

We will now compare the $\infty$-category of hypersheaves with transfers of \Cref{defi:sheaveswithtransfers} with the derived $\infty$-category of the Grothendieck abelian category of $\tau$-sheaves with transfers 
$$\tn{Sh}_{\tau}(\corsms{\bZ};R) \subseteq \nP_{\Sigma}(\corsms{\bZ};R).$$
Note that in \cite[§10]{cd19} and \cite{cd16}, Cisinski--Déglise use a coefficient ring $\Lambda$ for the finite correspondences, which is a localisation of $\bZ$ at the invertible primes of $R$, or equivalently the maximal subring $\Lambda \subseteq \bQ$ such that $R$ is a $\Lambda$-algebra. They then work with $\Lambda$-linear (pre)sheaves $\corsms{\Lambda}^\op \to \Mod(R)$. But by \cite[Remark 9.1.3]{cd19} we have $\corsms{\Lambda} \cong \corsms{\bZ} \otimes_{\bZ} \Lambda$, and there is an equivalence
$$\nP_{\Lambda\tn{-lin}}(\corsms{\Lambda};R) \simeq \nP_{\Sigma}(\corsms{\bZ};R)$$
which descends to sheaves with transfers. We will work with the latter category.

\begin{thm}\label{thm:derivedsheaveswithtransfers}
There is an equivalence of symmetric monoidal $\infty$-categories
$$\cD(\tn{Sh}_{\tau}(\corsms{\bZ};R)) \simeq \sheaves_{\tau}^{\wedge}(\corsms{R};\spec).$$
\end{thm}
\begin{proof}
By \cite[§10]{cd19} we have a diagram of adjunctions between Grothendieck abelian categories in which we use uppercase letters to distinguish the functors from the ones on $\infty$-categorical (pre)sheaves.
\begin{equation}\label{diag:sheaveswithtransfers}
\begin{tikzcd}[column sep=0.9cm,row sep=0.9cm]
	{\tn{Sh}_{\tau}(\corsms{\bZ};R)} & {\nP_{\Sigma}(\corsms{\bZ};R)} \\
	{\tn{Sh}_{\tau}(\sms;R)} & {\nP(\sms;R),}
	\arrow[""{name=0, anchor=center, inner sep=0}, "{\tsf{O}_{\tau}^{\tn{tr}}}"', shift right=2, hook, from=1-1, to=1-2]
	\arrow[""{name=1, anchor=center, inner sep=0}, "{\Gamma^*}", shift left=2, from=1-1, to=2-1]
	\arrow[""{name=2, anchor=center, inner sep=0}, "{\tsf{A}_{\tau}^{\tn{tr}}}"', two heads, shift right=2, from=1-2, to=1-1]
	\arrow[""{name=3, anchor=center, inner sep=0}, "{\hat{\Gamma}^*}", shift left=2, from=1-2, to=2-2]
	\arrow[""{name=4, anchor=center, inner sep=0}, "{\Gamma_!}", shift left=2, from=2-1, to=1-1]
	\arrow[""{name=5, anchor=center, inner sep=0}, "{\tsf{O}_{\tau}}"', shift right=2, hook, from=2-1, to=2-2]
	\arrow[""{name=6, anchor=center, inner sep=0}, "{\hat{\Gamma}_!}", shift left=2, from=2-2, to=1-2]
	\arrow[""{name=7, anchor=center, inner sep=0}, "{\tsf{A}_{\tau}}"', two heads, shift right=2, from=2-2, to=2-1]
	\arrow["\dashv"{anchor=center, rotate=-90}, draw=none, from=2, to=0]
	\arrow["\dashv"{anchor=center}, draw=none, from=4, to=1]
	\arrow["\dashv"{anchor=center, rotate=-90}, draw=none, from=7, to=5]
	\arrow["\dashv"{anchor=center}, draw=none, from=6, to=3]
\end{tikzcd}
\end{equation}
By definition of $\tn{Sh}_{\tau}(\corsms{\bZ};R)$ the diagram of right adjoints is a pullback. According to \cite[Proposition 10.3.3, 10.3.9]{cd19} the functors $\hat{\Gamma}^*$ and $\Gamma^*$ admit a further right adjoint, in particular they are exact, and the two sheafification functors $\tsf{A}_{\tau}$ and $\tsf{A}^{\tn{tr}}_{\tau}$ are exact as well. We hence obtain a diagram of adjunctions of derived $\infty$-categories
\begin{equation}\label{diag:derivedsheaveswithtransfers}
\begin{tikzcd}[column sep=0.9cm,row sep=0.9cm]
	{\cD(\tn{Sh}_{\tau}(\corsms{\bZ};R))} & {\cD(\nP_{\Sigma}(\corsms{\bZ};R))} \\
	{\cD(\tn{Sh}_{\tau}(\sms;R))} & {\cD(\nP(\sms;R)).}
	\arrow[""{name=0, anchor=center, inner sep=0}, "{\tbf{R}\tsf{O}_{\tau}^{\tn{tr}}}"', hook, shift right=2, from=1-1, to=1-2]
	\arrow[""{name=1, anchor=center, inner sep=0}, "{\Gamma^*}", shift left=2, from=1-1, to=2-1]
	\arrow[""{name=2, anchor=center, inner sep=0}, "{\tsf{A}_{\tau}^{\tn{tr}}}"', two heads, shift right=2, from=1-2, to=1-1]
	\arrow[""{name=3, anchor=center, inner sep=0}, "{\hat{\Gamma}^*}", shift left=2, from=1-2, to=2-2]
	\arrow[""{name=4, anchor=center, inner sep=0}, "{\tbf{L}\Gamma_!}", shift left=2, from=2-1, to=1-1]
	\arrow[""{name=5, anchor=center, inner sep=0}, "{\tbf{R}\tsf{O}_{\tau}}"', hook, shift right=2, from=2-1, to=2-2]
	\arrow[""{name=6, anchor=center, inner sep=0}, "{\tbf{L}\hat{\Gamma}_!}", shift left=2, from=2-2, to=1-2]
	\arrow[""{name=7, anchor=center, inner sep=0}, "{\tsf{A}_{\tau}}"', two heads, shift right=2, from=2-2, to=2-1]
	\arrow["\dashv"{anchor=center, rotate=-90}, draw=none, from=2, to=0]
	\arrow["\dashv"{anchor=center}, draw=none, from=4, to=1]
	\arrow["\dashv"{anchor=center, rotate=-90}, draw=none, from=7, to=5]
	\arrow["\dashv"{anchor=center}, draw=none, from=6, to=3]
\end{tikzcd}
\end{equation}
The right derived functors $\tbf{R}\tsf{O}_{\tau}^{(\tn{tr})}$ are fully faithful by the argument used in the proof of \cite[Corollary 4.1]{ser03}, which boils down to the fact that the derived right adjoints can be computed using a K-injective resolution (which is a fibrant replacement with respect to the injective model structure of \cite[Proposition 1.3.5.3]{lur17}). According to \cite[Proposition 5.5.7.6]{lur09} the $\prr$-pullback of the cospan
\begin{equation}\label{diag:derivedcospan}
\begin{tikzcd}
	{\cD(\tn{Sh}_{\tau}(\sms;R))} & {\cD(\nP(\sms;R))} & {\cD(\nP_{\Sigma}(\corsms{\bZ};R))}
	\arrow["{\hspace*{-0.16cm}\tbf{R}\tsf{O}_{\tau}}", hook, from=1-1, to=1-2]
	\arrow["{\hat{\Gamma}^*}"', from=1-3, to=1-2]
\end{tikzcd}
\end{equation}
is computed in $\what{\tsf{Cat}}_{\infty}$, and hence it is equivalent to the full subcategory 
$$\cC \subseteq \cD(\nP_{\Sigma}(\corsms{\bZ};R))$$
on complexes of presheaves with transfers $F_\bullet$ such that $\hat{\Gamma}^*(F_\bullet)$ lies in the essential image of the fully faithful functor $\tbf{R}\tsf{O}_{\tau}: \cD(\tn{Sh}_{\tau}(\sms;R)) \hookrightarrow \cD(\nP(\sms;R))$. The latter is equivalent to the unit map 
\begin{align}\label{eq:unitpullback}
\hat{\Gamma}^*(F_\bullet) \to \tbf{R}\tsf{O}_{\tau}\tsf{A}_{\tau}\hat{\Gamma}^*(F_\bullet)
\end{align}
being an equivalence. Since the sheafification functors and the functors that forget transfers are exact before derivation, \cite[Proposition 10.3.9]{cd19} implies that the target of this map identifies with $\tbf{R}\tsf{O}_{\tau}\Gamma^*\tsf{A}_{\tau}^{\tn{tr}}(F_\bullet) \simeq \hat{\Gamma}^*\tbf{R}\tsf{O}_{\tau}^{\tn{tr}}\tsf{A}_{\tau}^{\tn{tr}}(F_\bullet)$. The resulting map $\hat{\Gamma}^*(F_\bullet) \to \hat{\Gamma}^*\tbf{R}\tsf{O}_{\tau}^{\tn{tr}}\tsf{A}_{\tau}^{\tn{tr}}(F_\bullet)$ identifies with $\hat{\Gamma}^*$ applied to the unit map
\begin{align}\label{eq:unitpullbacki}
F_\bullet \to \tbf{R}\tsf{O}_{\tau}^{\tn{tr}}\tsf{A}_{\tau}^{\tn{tr}}(F_\bullet).
\end{align}
Since $\hat{\Gamma}^*$ is faithful and exact on abelian categories, it is conservative on derived categories. Hence \ref{eq:unitpullback} is an equivalence if and only if \ref{eq:unitpullbacki} is one. It follows that $\cC$ is equivalent to the essential image of 
$$\tbf{R}\tsf{O}_{\tau}^{\tn{tr}}: \cD(\tn{Sh}_{\tau}(\corsms{\bZ};R)) \hookrightarrow \cD(\nP_{\Sigma}(\corsms{\bZ};R)),$$ 
and the diagram of right adjoints in \ref{diag:derivedsheaveswithtransfers} is a pullback. So in order to show that there is an equivalence 
$$\cD(\tn{Sh}_{\tau}(\corsms{\bZ};R)) \simeq \sheaves_{\tau}^{\wedge}(\corsms{R};\spec)$$
it suffices to identify the cospans of \ref{diag:pullbacksheaveswithtransfers} and \ref{diag:derivedcospan}, up to an refinement of the former as in \Cref{cons:refinedpullback}. Since the left adjoints will identify as symmetric monoidal functors, this equivalence will be symmetric monoidal. We first consider the nodes. By \cite[Corollary 2.1.2.3]{lur18} the inclusion  of sheaves into $\infty$-hypersheaves $\tn{Sh}_{\tau}(\sms;R) \hookrightarrow \sheaves_{\tau}^{\wedge}(\sms;R)$ induces a t-exact equivalence
\begin{equation}\label{eq:derivedsheaves}
\cD(\tn{Sh}_{\tau}(\sms;R)) \simeq \sheaves_{\tau}^{\wedge}(\sms;R).
\end{equation}
The same argument implies that for presheaves there is a t-exact equivalence
\begin{equation}\label{eq:derivedpresheaves}
\cD(\nP(\sms;R)) \simeq \cP(\sms;R).
\end{equation}
Formally extending scalars to $R$ in the additive category $\corsms{\bZ}$ induces an equivalence to $R$-linear presheaves
$$\nP_{\Sigma}(\corsms{\bZ};R) \simeq \nP_{R\tn{-lin}}(\corsms{R};R).$$
We can thus view $\corsms{R}$ as a full subcategory of $\nP_{\Sigma}(\corsms{\bZ};R)$ via the $R$-linear Yoneda embedding 
$$\corsms{R} \hookrightarrow \nP_{R\tn{-lin}}(\corsms{R};R),$$ and its elements form a system of compact projective generators in the sense of \cite[Definition 2.51]{pst23}. Then \cite[Lemma 2.61]{pst23} provides a t-exact equivalence
\begin{equation}\label{eq:derivedsheavespresheavestransfers}
\cD(\nP_{\Sigma}(\corsms{\bZ};R)) \simeq \psigma(\corsms{R};\spec).
\end{equation}
Let us consider the functors between the nodes. We have equivalences
\begin{align*}
&\fun^{\tn{L,t-ex}}(\cD(\nP(\sms;R)),\cD(\tn{Sh}_{\tau}(\corsms{\bZ};R))) \\
\simeq \; &\fun^{\tn{L,lex}}(\cD(\nP(\sms;R))_{\geq 0},\cD(\tn{Sh}_{\tau}(\corsms{\bZ};R))_{\geq 0}) \\
\simeq \; &\fun^{\tn{L,lex}}(\nP(\sms;R),\tn{Sh}_{\tau}(\corsms{\bZ};R)),
\end{align*}
where the first equivalence holds by \cite[Proposition C.3.1.1, C.3.2.1]{lur18} and the second one by \cite[Theorem C.5.4.9]{lur18}.
The hypersheafification functor 
$$\tsf{a}^{\wedge}_{\tau}: \cP(\sms;R) \to \sheaves_{\tau}^{\wedge}(\sms;R)$$ is t-exact, and its restriction to hearts is given by the functor $\tsf{A}_{\tau}$ of \ref{diag:sheaveswithtransfers}. From the above equivalence of functor categories it follows that the adjunctions $(\tsf{A}_{\tau},\tbf{R}\tsf{O}_{\tau})$ and $(\tsf{a}^{\wedge}_{\tau},\tsf{o}_{\tau}^{\wedge})$ identify under the equivalences \ref{eq:derivedsheaves} and \ref{eq:derivedpresheaves}. As in \Cref{cons:refinedpullback}, we now consider the restriction functor
$$\hat{\delta}^*: \psigma(\corsms{R};\spec) \to \cP(\sms;R)$$
that precomposes with the functor
$$\delta: \sms \times \lat_R \stackrel{\gamma \times \sigma}{\longrightarrow} \corsms{R} \times \corsms{R} \stackrel{\times_S}{\longrightarrow} \corsms{R}.$$
Here $\lat_R$ is the category of finite rank free $R$-modules, $\sigma$ is the canonical additive functor $\lat_R \to \corfets{R}$ that sends $R$ to $S$, and the bifunctor $(-)\times_S(-)$ comes from the monoidal structure on $\corsms{R}$. Left Kan extension followed by localisation to $\psigma$ provides a left adjoint to $\hat{\delta}^*$, in particular the latter preserves finite colimits. It also preserves filtered colimits since they are computed pointwise. Hence $\hat{\delta}^*$ is a t-exact left adjoint, and the same argument that we just used for $\tsf{a}^{\wedge}_{\tau}$ now applies: on hearts $\hat{\delta}^*$ is given by the functor $\hat{\Gamma}^*$ of \ref{diag:sheaveswithtransfers}. We can hence identify the functors $\hat{\Gamma}^*$ and $\hat{\delta}^*$ under the equivalences \ref{eq:derivedpresheaves} and \ref{eq:derivedsheavespresheavestransfers}. The observation that the diagram \ref{diag:pullbacksheaveswithtransfers} refines to an iterated pullback involving $\hat{\delta}^*$ as in \Cref{cons:refinedpullback} now finishes the proof.
\end{proof}

\begin{rem}
The same proof goes through for $\fets$, in fact we could replace $\tsf{Sm}$ by any class $\mathscr{P}$ of separated finite type morphisms between noetherian schemes of finite dimension which contains étale separated morphisms of finite type \cite[{}10.0.1]{cd19}, and we could replace the topology $\tau$ with any that is mildly compatible with transfers, in the sense of \cite[Definition 10.3.5]{cd19}.
\end{rem}

\printbibliography
\end{document}